\documentclass{article}
\usepackage{psfrag}
\usepackage[parfill]{parskip}
\usepackage{float, graphicx}
\usepackage[]{epsfig}
\usepackage{amsmath, amsthm, amssymb,}
\usepackage{epsfig}
\usepackage{verbatim}
\usepackage{multicol}
\usepackage{url}

\usepackage{latexsym}
\usepackage{mathrsfs}
\usepackage[colorlinks, bookmarks=true]{hyperref}
\usepackage{graphicx}
\usepackage{amsmath}
\usepackage{enumerate}
\usepackage[normalem]{ulem}
\usepackage{bm}
\usepackage{dsfont}
\usepackage{stmaryrd}
\usepackage{enumitem}
\usepackage{cite}
\usepackage{color}
\usepackage{xcolor}
\usepackage{hyperref}% http://ctan.org/pkg/hyperref
\usepackage{mathtools}% http://ctan.org/pkg/hyperref
\usepackage[margin=1.1in]{geometry}

\makeatletter

\def\@settitle{\begin{center}%
    \bfseries
 \normalfont\LARGE\@title
  \end{center}%
}
\def\@setauthors{\begin{center}%
 \normalsize\@author
  \end{center}%
}

\makeatother

\numberwithin{equation}{section}

\def\O{\mathcal{O}}

\def\rr{\mathbb{R}}

\renewcommand{\cal}{\mathcal}
\newcommand\cA{{\mathcal A}}

\newcommand\cJ{{\mathcal J}}
\newcommand{\cC}{{\cal C}}
\newcommand{\cD}{{\cal D}}
\newcommand{\cE}{{\cal E}}

\newcommand{\cG}{{\cal G}}
\newcommand{\cI}{{\cal I}}

\newcommand{\fa}{{\mathfrak a}}
\newcommand{\fb}{{\mathfrak b}}
\newcommand{\fc}{{\mathfrak c}}
\newcommand{\fd}{{\mathfrak d}}

\newcommand{\fe}{{\mathfrak e}}
\newcommand{\fo}{{\mathfrak o}}
\newcommand{\fh}{{\mathfrak h}}

\newcommand{\fK}{{\frak K}}

\newcommand{\bmu}{{\bm{u}}}

\newcommand{\bfi}{{\bf i}}
\newcommand{\bfj}{{\bf j}}
\newcommand{\bfk}{{\bf k}}
\newcommand{\bfm}{{\bf m}}

\newcommand{\rd}{{\rm d}}

\newcommand{\ri}{\mathrm{i}}

\newcommand{\bC}{{\mathbb C}}
\newcommand{\bE}{\mathbb{E}}

\newcommand{\bN}{\mathbb{N}}
\newcommand{\bP}{\mathbb{P}}

\newcommand{\bR}{{\mathbb R}}

\newcommand{\al}{\alpha}

\newcommand{\la}{\lambda}

\DeclareMathOperator{\Tr}{Tr}

\DeclareMathOperator{\dist}{dist}

\DeclareMathOperator{\OO}{O}
\DeclareMathOperator{\oo}{o}

\renewcommand{\Im}{\mathop{\mathrm{Im}}}

\newcommand{\deq}{\mathrel{\mathop:}=} %define :=
 
\renewcommand{\leq}{\leqslant}
\renewcommand{\geq}{\geqslant}
\newcommand{\floor}[1] {\lfloor {#1} \rfloor}

\newcommand{\nc}{\normalcolor}

\newcommand{\del}{\partial}

\newcommand{\wh}{\widehat}
\newcommand{\wt}{\widetilde}

\newcommand{\qq}[1]{[\![{#1}]\!]}

\newcommand{\beq}{\begin{equation}}
\newcommand{\eeq}{\end{equation}}
\newcommand{\pa}[1]{\left({#1}\right)}

\newcommand{\h}[1]{\{{#1}\}}

\newcommand{\absbb}[1]{\biggl\lvert #1 \biggr\rvert}

\newcommand{\scalar}[2]{\langle{#1} \mspace{2mu}, {#2}\rangle}

\theoremstyle{plain} %plain, definition, remark
\newtheorem{theorem}{Theorem}[section]
\newtheorem*{theorem*}{Theorem}
\newtheorem{lemma}[theorem]{Lemma}
\newtheorem*{lemma*}{Lemma}
\newtheorem{corollary}[theorem]{Corollary}
\newtheorem*{corollary*}{Corollary}
\newtheorem{proposition}[theorem]{Proposition}
\newtheorem*{proposition*}{Proposition}

\newtheorem*{assumption*}{Assumption}

\newtheorem{definition}[theorem]{Definition}
\newtheorem*{definition*}{Definition}
\newtheorem{example}[theorem]{Example}
\newtheorem*{example*}{Example}
\newtheorem{remark}[theorem]{Remark}

\newtheorem*{remark*}{Remark}
\newtheorem*{remarks*}{Remarks}

\newcommand{\col}{\vcentcolon}
\newcommand{\rhosc}{\rho_{\mathrm{sc}}}

\newcommand{\xid}{\xi_{d}}

\newcommand{\msc}{m_{\rm sc}}
\newcommand{\md}{m_d}

\newcommand{\Lambdao}{\Lambda_{\rm o}}
\newcommand{\Lambdad}{\Lambda_{\rm d}}

\def\author#1{\par
    {\centering{\authorfont#1}\par\vspace*{0.05in}}
}

\def\titlefont{\fontsize{13}{15}\bfseries\boldmath\selectfont\centering{}}
\def\authorfont{\fontsize{13}{15}}

\let\affiliationfont\rhfont

\def\address#1{\par
    {\centering{\affiliationfont#1\par}}\par\vspace*{11pt}
}

\def\body{
\setcounter{footnote}{0}
\def\thefootnote{\alph{footnote}}
\def\@makefnmark{{$^{\rm \@thefnmark}$}}
}

\def\title#1{
    \thispagestyle{plain}
    \vspace*{-14pt}
    \vskip 79pt
    {\centering{\titlefont #1\par}}%
    \vskip 1em
}

\def\rhosc{\rho_{\text{sc}}}

\newcommand{\cT}{{\mathcal T}}

\newcommand{\txt}[1]{\text{\rm{#1}}}
\newcommand{\bld}[1]{\boldsymbol{\mathrm{#1}}} %bold
\newcommand{\pbb}[1]{\biggl(#1\biggr)}

\begin{document}

\title{Edge Universality of Random Regular Graphs of Growing Degrees}

\vspace{1.2cm}

\noindent \begin{minipage}[c]{0.5\textwidth}
 \author{Jiaoyang Huang}
\address{University of Pennsylvania\\
   E-mail: huangjy@wharton.upenn.edu}
 \end{minipage}
\begin{minipage}[c]{0.5\textwidth}
 \author{Horng-Tzer Yau}
\address{Harvard University \\
   E-mail: htyau@math.harvard.edu}

 \end{minipage}

\begin{abstract}
We consider the statistics of  extreme eigenvalues of random $d$-regular graphs, with $N^{\mathfrak c}\leq d\leq N^{1/3-{\mathfrak c}}$ for arbitrarily small ${\mathfrak c}>0$. 
We prove that in this regime, the  fluctuations of extreme eigenvalues are given by  the Tracy-Widom distribution. As a consequence, about 69\% of $d$-regular graphs have all nontrivial eigenvalues bounded in absolute value by $2\sqrt{d-1}$. 
\end{abstract}

{
\hypersetup{linkcolor=black}
\setcounter{tocdepth}{1}
\tableofcontents
}

\section{Introduction}\label{s:intro}
\textit{Expander graphs} are sparse graphs with strong connectivity property, and have numerous applications to the design of robust networks and the theory of error-correct coding \cite{hoory2006expander}. The expansion property of a regular graph can be measured by its spectral gap. \textit{Ramanujan graphs} are $d$-regular graphs with the largest spectral gap, i.e. all nontrivial eigenvalues are bounded in absolute value by $2\sqrt{d-1}$.  They are the best possible expander graphs, at least as far as the spectral gap measure of expansion is concerned. 

Ramanujan graphs were explicitly constructed when $d=p+1$ with $p$ a prime power, by Lubotzky, Phillips and Sarnak \cite{lubotzky1988ramanujan} and independently by Margulis \cite{margulis1988explicit}, using deep tools from number theory. More recently, in the breakthrough work \cite{marcus2013interlacing,marcus2018interlacing}, Marcus, Spielman and Srivastava proved that there exist \textit{bipartite Ramanujan graphs} of every degree. It still remains an interesting open problem whether there exist non-bipartite Ramanujan graphs of every degree. However, based on numerical simulation \cite{miller2008distribution}, it was observed a positive portion of \textit{random $d$-regular graphs} are Ramanujan. 

We study the statistics of extreme eigenvalues of random $d$-regular graphs, when the degree $d$ grows with the size of the graph. For random $d$-regular graphs on $N$ vertices, it has been proven in \cite{bauerschmidt2020edge} for $N^{2/9+\fc}\leq d\leq N^{1/3-\fc}$ and in \cite{he2022spectral} for $N^{2/3+\fc}\leq d\leq N/2$, the  fluctuations of extreme eigenvalues are given by  the \textit{Tracy-Widom distribution} from random matrix theory. In this work we generalize these edge universality results to the sparser regime $N^\fc\leq d\leq N^{1/3-\fc}$ for arbitrarily small $\fc>0$. As a consequence, about 69\% of $d$-regular graphs on $N$ vertices with $d^3\ll N\leq d^{1/\fc}$ have all nontrivial eigenvalues bounded in absolute value by $2\sqrt{d-1}$. This gives many Ramanujan graphs.

Universality for the edge statistics of Wigner matrices was first established by the moment method \cite{MR1727234, ruzmaikina2006universality, khorunzhiy2012high} under certain symmetry assumptions on the distribution of the matrix elements. The moment method was further developed in \cite{MR2475670,MR2647136} and \cite{MR2726110}. A different approach to edge universality for Wigner matrices  based on the direct comparison with corresponding Gaussian ensembles was developed in \cite{MR2669449,MR2871147,erdHos2012spectral}. And later,  a necessary and sufficient condition for edge universality of Wigner matrices were discovered in \cite{lee2014necessary}.

For the comparison argument, a key intermediate ingredient is the rigidity of extreme eigenvalues. For edge universality to be true, it is necessary that extreme eigenvalues fluctuate on scale $\OO_\prec(N^{-2/3})$. Main body of this paper is devoted to prove such optimal rigidity estimate. The proof is based on first constructing a \textit{higher order self-consistent equation} for the Stieltjes transform of the empirical eigenvalue distribution, then computing high moments of the self-consistent equation by a recursive moment estimate, and proving it concentrates around zero. This approach was first introduced in \cite{lee2016local} and further developed in \cite{lee2021higher,huang2020transition,huang2022edge,he2021fluctuations} to study sparse Erd{\H o}s-R{\' e}nyi graphs and in  \cite{bauerschmidt2020edge} to study sparse random $d$-regular graphs.

For random $d$-regular graphs, the self-consistent equation has a $1/\sqrt d$ asymptotic expansion. The main challenge for the above approach is to show high moments of the self-consistent equation are close to zero, which requires cancellations of each order in the asymptotic expansion. \cite{bauerschmidt2020edge} was able to explore the cancellation up to $d^{-3/2}$, which implies the fluctuation of extreme eigenvalues is bounded by $\OO_\prec(d^{-3}+N^{-2/3})$. To get the optimal $\OO_\prec(N^{-2/3})$ fluctuation bound, we need to have cancellations up to $d^{-\fa/2}$ for arbitrary large $\fa\geq 0$. To achieve it, in this paper we combine the two steps, and construct  the high order self-consistent equation and compute its high moments in one step. By construction, our high order self-consistent equation concentrates around zero. To identify the the self-consistent equation, we directly show the Stieltjes transform of the spectral measure of the \textit{infinite $d$-regular tree} satisfies the same equation. Thus extreme eigenvalues concentrate around the spectral edge of the infinite $d$-regular tree, with optimal error $\OO_\prec(N^{-2/3})$.

The computation of the high order self-consistent equation boils down to estimate the following quantity (and its generalizations)
\begin{align}\label{e:keyterm}
\bE[A_{\cT} R U], 
\end{align}
where $\cT$ is a sub-tree graph (may not be connected), $A_\cT$ is the indicator function that $\cT$ is contained in the graph, $R$ is a polynomial of the Green's function entries, and $U$ is any polynomial of the Stieltjes transform. To estimate \eqref{e:keyterm}, we resample each edge of the sub-tree $\cT$, and use it as a discrete integration by part. This generalizes the simple switching procedure, introduced in \cite{bauerschmidt2017local} to prove the local law of random $d$-regular graphs.

For \eqref{e:keyterm}, there are three cases:
i) If $R$ contains an off-diagonal Green's function entry $G_{ij}$ such that $i,j$ are in different connected components of $\cT$, we show \eqref{e:keyterm} is negligible. ii) If $R$ contains an off-diagonal Green's function entry $G_{ij}$ such that $i,j$ are in the same connected component of $\cT$, we show \eqref{e:keyterm} is a finite linear combination of terms in the same form, but with some extra factors of $1/\sqrt d$. iii) If $R$ contains only diagonal Green's function entries,  the leading term is obtained by replacing $R$ by a power of the Stieltjes transform of the empirical measure, and the higher order terms are again in the form of \eqref{e:keyterm} with some extra factors of $1/\sqrt d$. By repeating this procedure, we eventually obtain a polynomial of the Stieltjes transform of the empirical measure, which leads to the self-consistent equation. 

%With the optimal bound on the fluctuation of extreme eigenvalues,  the convergence to the Tracy-Widom distribution follows from comparing the adjacency matrix to one with a small Gaussian component. 
To prove the convergence to the Tracy-Widom distribution, we took a dynamical approach, which was original developed  to prove the bulk universality of Wigner matrices \cite{MR2810797,MR2919197, bourgade2017eigenvector,landon2019fixed,landon2017convergence,bourgade2016fixed,MR3699468, erdHos2017universality}.
We interpolate the adjacency matrix with the Gaussian orthogonal ensemble (GOE). The eigenvalues at each time slice have the same law as Dyson Brownian motion (DBM). For sufficiently regular initial data, it has been proven in \cite{landon2017edge}, after short time the eigenvalue statistics at the spectral edge of DBM agree with GOE. With the optimal edge rigidity as input, \cite{landon2017edge} implies the edge universality holds for adjacency matrix plus a small Gaussian component. Our main contribution is to show the edge statistics do not change along the interpolation, thus edge universality holds for the original adjacency matrix. The main challenge is to keep track the edge location $L_t$ along the interpolation. At each time $t$, we can construct a high order self-consistent equation, which governs the spectral of the interpolated matrix. We discover that  the time derivative of $L_t$ is given by the spatial derivative of the corresponding high order self-consistent equation  with a negligible error. This allows us to build a dynamical framework around $L_t$ and show the edge statistics do not change.

It is worth to compare edge results of sparse Erd{\H o}s--R{\'e}nyi graphs and random $d$-regular graphs. Thanks to a series of papers \cite{he2021fluctuations,erdHos2012spectral, lee2016local, lee2021higher,huang2020transition,huang2022edge}, the edge statistics of Erd{\H o}s--R{\'e}nyi graphs $G(N,p)$ with $Np\geq N^\fc$ are now well understood. When the average degree $pN\gg N^{1/3}$, extreme eigenvalues have Tracy-Widom fluctuations \cite{erdHos2012spectral, lee2016local}. However, when $N^{\fc}\leq  pN\ll N^{1/3}$ the extreme eigenvalues have Gaussian fluctuations, given by the fluctuation of the total number of edges \cite{huang2020transition,he2021fluctuations}. The higher order fluctuations are given by a sum of an infinite hierarchy of subgraph counting random variables which are asymptotically joint Gaussian \cite{lee2021higher}. Surprisingly, up to this explicit random shift, the fluctuations of  extreme eigenvalues are still given by the Tracy-Widom distribution \cite{huang2022edge}. 
For random $d$-regular graphs, the fluctuations of those subgraph counting
 quantities are negligible. Extreme eigenvalues concentrate around a deterministic spectral edge and converge to the Tracy-Widom
distribution. In the same regime, bulk universality for Erd{\H o}s--R{\'e}nyi graphs and random $d$-regular graphs are also known \cite{MR3729611,huang2015bulk,erdHos2012spectral}.

For Erd{\H o}s--R{\'e}nyi graphs in the sparser regime $Np=\OO(\ln N)$, there exists a critical value $b_*=1/(\ln 4-1)$ such that  if $Np\geq b_*\ln N/N$, the extreme eigenvalues of the normalized adjacency matrix converge to $\pm 2$ \cite{benaych2019largest, alt2021extremal, tikhomirov2021outliers, benaych2020spectral} and all the eigenvectors are delocalized \cite{alt2022completely, erdHos2013spectral}. For $(\ln  \ln  N)^4\ll Np<b_*\ln N/N$, there are outlier eigenvalues \cite{tikhomirov2021outliers, alt2021extremal}.
The spectrum splits into three phases: a delocalized phase
in the bulk, a fully localized phase near the spectral edge, and a semilocalized phase in between \cite{alt2023poisson,alt2021delocalization}. Moreover, the joint fluctuations of the eigenvalues near the spectral edges form a Poisson point process.
For random $d$-regular graphs with any fixed degree $d\geq 3$, all the eigenvectors are delocalized \cite{bauerschmidt2019local,huang2021spectrum}. The fluctuation of extreme eigenvalues is still conjectured to be the Tracy-Widom distribution \cite{miller2008distribution}. This conjecture is still quite open, but we do know the extreme eigenvalues concentrate around $\pm 2\sqrt{d-1}$ with an error $\OO((\ln \ln N/ \ln N)^2)$ \cite{bordenave2020new, friedman2003proof}, which  recently was  improved by the two authors to a polynomial bound $\OO(N^{-\fc})$ \cite{huang2021spectrum}.

\noindent\textbf{Organization.} We define the model and present the main results in the rest of Section \ref{s:intro}. 
In Section \ref{sec:notation}, we recall the local law for random $d$-regular graphs from \cite{MR3688032}, and state our improved estimates for the Stieltjes transform of the empirical eigenvalue distribution. 
In Section \ref{sec:switchings}, we discuss simple switchings, and use them as a discrete integration by part. In Section \ref{sec:GFE}, we collect some estimates and notations on Green's function. 
In Section \ref{sec:P-construct}, we construct the high order self-consistent equation, and show its moments are close to zero.  In Section \ref{sec:P-identification}, we 
identify the self-consistent equation and prove the optimal edge rigidity. 
In Section \ref{sec:universality}, we prove the edge universality result by interpolating with GOE. 

\noindent\textbf{Notations.} We use $C$ to represent large universal constants, and $c$ small universal constants, which may be different from line by line. Let $Y\geq 0$. %We write that $X=\O_{\prec}(Y)$ or $X\prec Y$, if there exists an \emph{arbitrarily} small exponent $\fc>0$ such that $X\leq  N^{\fc}Y$ for $N\geq N(\fc)$ large enough. 
%We write $X\ll Y$ or $Y\gg X$, if there exists a small exponent $\fc>0$ such that $X\leq  N^{-\fc}Y$ for $N\geq N(\fc)$ large enough. 
We write $X\lesssim Y$ or $Y\gtrsim X$ if there exists a constant $C>0$, such that $X\leq CY$. 
We write $X\asymp Y$ if there exists a constant $C>0$ such that $Y/C\leq X\leq Y/C$.
We use $\bC_+$ to represent the upper half plane. For any index set $\bm m=\{m_1, m_2,\cdots, m_r\}$,
we write
\begin{align}\label{e:starsum}
\sum^*_{\bm m}=\sum_{1\leq m_1, m_2,\cdots, m_r\leq N, \atop \text{distinct}}.
\end{align}

\noindent\textbf{Acknowledgements }
The research of J.H. is supported by  NSF grant DMS-2054835. The research of H-T.Y. is supported by NSF grants DMS-1855509, DMS-2153335 and a Simons Investigator award. 
J.H. wants to thank Yukun He for carefully reading the manuscript of the paper, and pointing some mistakes.

\subsection{Main results}

Let $\bP$ be the uniform probability measure on the set of $d$-regular graphs on $N$ vertices. We identify a graph with its adjacency matrix $A = (A_{ij}) \in \{0,1\}^{N \times N}$, defined as $A_{ij} = 1$ if and only if $i$ and $j$ are adjacent. Thus, $\bP$ is the uniform probability measure on the set of symmetric matrices $A \in \{0,1\}^{N \times N}$ satisfying $\sum_{j = 1}^N A_{ij} = d$  and $A_{ii} = 0$  for all $i = 1,\dots, N$.

Since $A$ is $d$-regular, it is immediate that $A$ has a trivial eigenvalue $d$ with associated eigenvector $(1,1,\dots,1)^*$. Moreover, by the Perron-Frobenius theorem, all other eigenvalues are bounded in absolute value by $d$. For convenience, we shall consider the normalized adjacency matrix
\begin{equation} \label{def_H}
H\deq A/\sqrt{d-1}.
\end{equation}
We denote its eigenvalues by $\lambda_1 = d/\sqrt{d-1} \geq \lambda_2 \geq \cdots \geq \lambda_N \geq -d/\sqrt{d-1}$, and corresponding orthonormal eigenvectors $\bmu_1,\bmu_2,\cdots, \bmu_N$.

Our main result is about the limiting distribution of the extremal eigenvalues. For $1 \ll d \ll N^{1/3}$, extreme eigenvalues of random $d$-regular graphs converge to the Tracy-Widom distribution.

\begin{theorem} \label{thm:Tracy-Widom}
Fix arbitrarily small $\fc>0$, and an integer $k\geq 1$, let $F : \rr^k \to \rr$ be a bounded test function with bounded derivatives. For $N^{\fc} \leq d \leq N^{1/3-\fc}$, there is a universal constant $\fe>0$ so that,
\begin{align}\begin{split} \label{eqn:htedgeb1}
&\phantom{{}={}}\bE_{H}[ F (N^{2/3} ( \lambda_2 - 2), \cdots , N^{2/3} ( \lambda_{k+1}- 2 ) ]\\
& = \bE_{GOE}[ F (N^{2/3} ( \mu_1 - 2 ), \cdots, N^{2/3} ( \mu_k - 2) ) ]+\O\left(N^{-\fe}\right),
\end{split}\end{align}
provided $N$ is large enough. 
The second expectation  is with respect to a GOE matrix with eigenvalues $\mu_1\geq\mu_2\geq\cdots\geq \mu_N$. 
Analogous results hold for the smallest eigenvalues.
\end{theorem}

%In the seminal work \cite{lubotzky1988ramanujan}, Lubotzky, Phillips and Sarnak coined the term \emph{Ramanujan graph}--a $d$-regular graph whose largest nontrivial eigenvalue is bounded in absolute value by $2\sqrt{d-1}$, which are the best possible expanders. Explicit constructions of Ramanujan graphs with $d=p+1$ for some prime and prime powers $p$ were introduced in \cite{lubotzky1988ramanujan,margulis1988explicit} (see also \cite{MR2072849}). 

Since the seminal works \cite{lubotzky1988ramanujan,margulis1988explicit}, it remains an open problem if there exist infinite families of Ramanujan graphs with any degree $d$. For the bipartite case, an affirmative answer was given in the recent breakthrough work \cite{marcus2013interlacing,marcus2018interlacing}; and see \cite{MR3630988} for polynomial time algorithm of this construction. The problem for non-bipartite case is still open. And it is conjectured in \cite{MR2072849,miller2008distribution},  a positive fraction of regular graphs of fixed $d$ is Ramanujan.
As emphasized in \cite{MR2072849},
the Tracy--Widom$_1$ distribution has  about $83\%$ of its mass on the set $\{x:x<0\}$.
Therefore Theorem~\ref{thm:Tracy-Widom} implies $83\%$ of
$d$-regular graphs have the second eigenvalue \emph{less than} $2\sqrt{d-1}$,
provided that $N$ and $d$ obey the conditions of Theorem~\ref{thm:Tracy-Widom}. 
The proof of Theorem~\ref{thm:Tracy-Widom} can be extended to show 
that in the regime $N^{\fc} \leq d \leq N^{1/3-\fc}$ the largest and smallest nontrivial eigenvalues converge in distribution to \emph{independent} Tracy--Widom$_1$ distributions; see Remarks \ref{rem:independence} and \ref{rem:ind2} below.
As a consequence, we have the following result.

\begin{corollary}\label{c:rate}
For $d$ large enough and $d^3\ll N\leq d^{1/\fc}$, about $69\%$\footnote{In \cite{bauerschmidt2020edge}, the percentage is $27\%$, which is a mistake. This percentage $27\%$ is computed as $52\%\times 52\%$, using that the Tracy--Widom$_1$ distribution has  about $52\%$ of its mass on the set $\{x:x<0\}$. The number $52\%$ is wrong, and should be $83\%$. } of $d$-regular graphs on $N$ vertices have all nontrivial eigenvalues bounded in absolute value by $2\sqrt{d-1}$.
\end{corollary}

Although Corollary \ref{c:rate} does not give infinite families of Ramanujan graphs, it does give a  lot Ramanujan graphs.

%We denote the eigenvalues of $H$ by
%$
%\lambda_1\geq \lambda_2\geq\cdots\geq \lambda_N,
%$
%corresponding eigenvectors $u_1, u_2,\cdots, u_N$, and the Green's function of $H$ by
%\begin{align*}
%G(z):=(H-z)^{-1}=\sum_{\al=1}^N\frac{u_\al u_\al^*}{\la_\al-z}.
%\end{align*}
%The Stieltjes transform of the empirical eigenvalue distribution is denoted by
%\begin{align*}
%m_N(z):=\frac{1}{N}\Tr G(z)=\frac{1}{N}\sum_{\al=1}^N\frac{1}{\lambda_\al-z}
%\end{align*}
%
%We fix a large constant $\fb>0$ and define the spectral domain,
%\begin{align}\begin{split}\label{e:domain}
%%\cal D_*&=\{z=E+\ri\eta: \;-\fb\leq E\leq \fb, \; 0< \eta\leq \fb\}, \\
%\cal D&=\{z=E+\ri\eta: \;-\fb\leq E\leq \fb, \; 1/N\ll \eta\leq \fb\}. 
%\end{split}\end{align}
%

Unless stated otherwise, all quantities in this paper depend on the fundamental parameter $N$, and we omit this dependence from our notation. For the following statements, for deterministic $N$-dependent quantities  $X$ and $Y$ \nc we write
$X \ll Y$ if $X \leq N^{-\varepsilon} Y$  for some fixed $\varepsilon>0$ and $N\geq N(\varepsilon)$ large enough.
For $N$-dependent random (or deterministic) variables $X$ and $Y$, we say $Y$ \emph{stochastically dominates} $X$, if for any $\varepsilon>0$ and $D>0$, then
\begin{equation*}
 \bP(X \leq N^\varepsilon Y) \geq 1-N^{-D},
\end{equation*}
for $N\geq N(\varepsilon, D)$ large enough, and we write   $X \prec Y$ or $X=\O_\prec(Y)$. We say an event $\Omega$ holds with \emph{overwhelming probability}, if for any $D>0$, $\bP(\Omega)\geq 1-N^{-D}$ for $N\geq N(D)$ large enough.

%Moreover, for deterministic $N$-dependent quantities $A$, $B$, we write  $A \ll B$, 
%\begin{equation}
% \quad \text{if} \quad A =  O_\varepsilon(N^{-\varepsilon} B) \text{ for some $\varepsilon>0$.}
%\end{equation}

\section{Eigenvalue Rigidity Estimates}\label{sec:notation}

We consider the adjacency matrix $A$ restricted to the subspace orthogonal to the vector ${\bf 1} \deq (1,\dots, 1)^*$.
More precisely,
let $P_\perp \col \bR^N \to \bR^N$ be the orthogonal projection onto ${\bf 1}^{\perp}$,
explicitly given by $P_{\perp}=I-{\bf 1}{\bf 1^*}/N$ where $I$ is the $N\times N$ identity matrix.
Since $H$ is the normalized adjacency matrix of a regular graph, the matrices $H$ and $P_\perp$ commute: $HP_\perp=P_\perp H$.

For a spectral parameter $z \in \bC_+ \deq \{z\in \bC \col \Im[z]>0\}$ we define the \emph{Green's function} by
\begin{equation*}
  G(z) \deq P_\perp (H-z)^{-1}P_{\perp}=\sum_{i=2}^N \frac{\bmu_i \bmu_i^*}{\la_i-z},
\end{equation*}
we remark that the projection $P_{\perp}$ kills the largest eigenvalue.
Thus, $G$ and $(H-z)^{-1}$ agree on the image of $P_\perp$, which is the subspace of $\bR^N$ perpendicular to $\bm 1$.
The Green's function satisfies the relation
\begin{align}\label{e:GHexp}
G(z)H=HG(z)=zG(z)+P_{\perp}=zG(z)+(I-{\bf 1}{\bf 1^*}/N).
\end{align}
Moreover,
\begin{equation} \label{sumG0}
\sum_{i=1}^NG_{ij}(z)=\sum_{j=1}^NG_{ij}(z)=0.
\end{equation}
We denote the normalized trace of $G$, which is also the Stieltjes transform of the empirical spectral measure of $H|_{\bm 1^{\perp}}$, by
\begin{equation} \label{e:m}
  m(z) \deq \frac{1}{N} \sum_i G_{ii}(z)=\sum_{i=2}^N\frac{1}{\la_i-z}.
\end{equation}
 We refer to $m(z)$ simply as the \emph{Stieltjes transform}.
Our goal is to approximate $m(z)$ by $\md(z)$,
the Stieltjes transform of the Kesten--McKay law \cite{MR0109367},
\begin{align*}
\md(z) \deq \int_\bR \frac{\rho_d(x)}{x-z}\rd x , \qquad \rho_d(x)\deq \left(1+\frac{1}{d-1}-\frac{x^2}{d}\right)^{-1}\frac{\sqrt{[4-x^2]_+}}{2\pi}.
\end{align*}
The Kesten--McKay law $\rho_d$ is the spectral measure at any vertex of the infinite $d$-regular tree (see, for example, \cite{MR3364516}
or \cite[Section~5]{MR3962004}). It has support $[-2,2]$ in our normalization. The Stieltjes transform $\md(z)$ is explicitly given by
\begin{equation} \label{e:md}
  \md(z) = -\pa{z+\frac{d}{d-1} \msc(z)}^{-1},
\end{equation}
where $\msc(z)$ is the Stieltjes transform of the Wigner semicircle law,
\begin{equation} \label{def_sc}
\msc(z)\deq\int_\bR\frac{\rhosc(x)}{x-z}\rd x , \qquad \rhosc(x)\deq\frac{\sqrt{[4-x^2]_+}}{2\pi},
\end{equation}
satisfying the self-consistent equation
\begin{equation} \label{e:sc_sce}
1+z\msc(z)+\msc^2(z)=0.
\end{equation}
Later we shall use that, alternatively, $\md(z)$ can be characterized by the self-consistent equation, which follows from rearranging \eqref{e:md}
\begin{align} \label{e:md-sce}
P_\infty(z,\md(z)) = 0, \quad
P_\infty(z,w) = 1 + z w + \frac{d}{d - 1} w^2 + \sum_{k \geq 2} \frac{(-2)^{k - 1} (2k - 3)!!}{k!} \, \frac{d}{(d - 1)^k} \, w^{2k}.
\end{align}
%Indeed, from \eqref{e:md} and \eqref{e:sc_sce} we get
%\begin{equation} \label{e:mdmscquad}
%\frac{1}{\md(z)} = \frac{1}{\msc(z)} - \frac{\msc(z)}{d - 1},
%\end{equation}
%from which we obtain, for large enough $d$,
%\begin{equation*}
%\msc(z) = \frac{d - 1}{2 \md(z)} \biggl( \sqrt{1 + 4 \frac{\md(z)^2}{d - 1}} - 1\biggr) = \md(z) + \frac{d - 1}{2 \md(z)} \sum_{k \geq 2} (-1)^{k - 1} \frac{2^k (2k - 3)!!}{k!} \, \frac{\md(z)^{2k}}{(d - 1)^k}.
%\end{equation*}
%Plugging this into $0 = 1 + z \md(z) + \frac{d}{d - 1} \md(z) \msc(z)$, as follows from \eqref{e:md}, yields \eqref{e:md-sce}.

We fix a large constant $\fK > 0$ and define the spectral domain
\begin{equation} \label{e:D}
  \mathbf D \deq \{z=E+\ri \eta\col -\fK \leq E \leq \fK, 0< \eta \leq \fK, N\eta\sqrt{\min\{|E-2|, |E+2|\}+\eta}\geq N^{1/\fK}\}.
\end{equation}
Here, and throughout the following, we use the notation
$
z = E + \ri \eta
$
for the real and imaginary parts of $z$.
The local semicircle law for random regular graphs \cite{MR3688032}
shows that $m(z)$ is approximated by $\md(z)$ to order $1/\sqrt{d}$ (up to logarithmic corrections),
at least away from the edges $\pm 2$ of the spectral measure.  The next result follows from \cite[Theorem 1.1]{MR3688032}. We remark that \cite[Theorem 1.1]{MR3688032} gives a better estimate for $\Lambdad$ for $z$ away from the edge $\pm 2$, but the following weaker form is enough for our purpose.
\begin{theorem}[{\cite[Theorem 1.1]{MR3688032}}]  \label{thm:rigidity}
Let $z=E+\ri \eta\in \bC_+$, define the deterministic control parameters
\begin{equation} \label{Lambda_weak}
\Lambdao(z) \deq \frac{1}{\sqrt{N\eta}} + \frac{1}{\sqrt{d}} + \frac{d^{3/2}}{N},\quad \Lambdad(z) \deq \left(\frac{1}{\sqrt{N\eta}} + \frac{1}{\sqrt{d}} + \frac{d^{3/2}}{N}\right)^{1/2}.
\end{equation}
Fix small $\fc>0$, we have, for $N^\fc\leq d\leq N^{2/3-\fc}$ and all $z\in \mathbf D$ from \eqref{e:D}
\begin{align} \label{G_estimates}
\max_{i\neq j}|G_{ij}(z)|\prec \Lambdao(z), \quad \max_{i}|G_{ii}(z)-\msc(z)|\prec \Lambdad(z).
\end{align}
\end{theorem}

Averaging over the index $i$, the estimate \eqref{G_estimates} implies an
  estimate on the Stieltjes transform $m(z)$. Using additional cancelations from this average,
  in this paper we shall derive a more precise estimate
  from which we obtain our results about the extremal eigenvalues.

\begin{theorem}\label{t:rigidity}

For $N^\fc\leq d\leq N^{1/3-\fc}$, uniform for $z=E+\ri \eta\in \mathbf D$, the Stieltjes transform of random $d$-regular graphs satisfies
\begin{align}\label{e:mbond}
|m(z)-m_d(z)|\prec 
\left\{\begin{array}{cc}
 \frac{1}{N\eta}+\frac{d}{N\sqrt{\dist(z,\pm 2)}}, & -2\leq E\leq 2,\\
 \frac{1}{\sqrt{\dist(z,\pm 2)}}\left(\frac{1}{N\eta^{1/2}}+\frac{1}{(N\eta)^2}+\frac{d}{N}\right), & |E|\geq 2.
\end{array}
\right.
\end{align}
We denote the classical eigenvalue locations of the Kesten-McKay distribution $\rho_d$ as $\gamma_2>\gamma_3>\cdots>\gamma_N$,
\begin{align*}
    \frac{k-1/2}{N-1}=\int_{\gamma_k}^{2}  \rho_d(x)\rd x,\quad 2\leq k\leq N. 
\end{align*}
Then  we have the following rigidity estimates
 \begin{align}\label{e:largeeig}
 |\lambda_k-\gamma_k|\prec \frac{d}{N}+\frac{1}{N^{2/3}\min\{k-1, N+1-k\}^{1/3}},\quad 2\leq k\leq N.
 \end{align}
\end{theorem}
\begin{remark}
We remark the rigidity estimates \eqref{e:largeeig} is not sharp for bulk eigenvalues because of the extra term $d/N$. However, since $d/N\ll 1/N^{2/3}$, the estimate is sharp for extreme eigenvalues. 
\end{remark}

\section{Simple Switchings} \label{sec:switchings}

Our analysis makes use of switchings for regular graphs and also makes some use of the invariance under the permutation of vertices.
We use ideas related to those introduced in \cite{MR3688032, MR3729611}, to which we also refer for references to other uses of switchings.

\subsection{Switchings}

As in \cite{MR3688032,MR3729611}, we define the signed adjacency matrices
\begin{equation}\label{e:defxi}
(\Delta_{ij})_{ab} \deq \delta_{ia} \delta_{jb} + \delta_{ib} \delta_{ja},\qquad
\xi_{ij}^{k\ell} \deq \Delta_{ij} + \Delta_{k\ell} - \Delta_{ik} - \Delta_{j\ell},
\end{equation}
corresponding to a switching of the edges $ij$ and $k\ell$; see Figure~\ref{fig:switch1}. For any indices $i,j,k,\ell$, we denote the indicator function that the edges $ij$ and $k\ell$ are \emph{switchable}
(i.e.\ the edges $ij$ and $k\ell$ are present and the switching again results in a simple regular graph) by
\begin{align}\label{e:defchi}
\chi_{ij}^{k\ell}(A)=A_{ij}A_{k\ell}(1-A_{ik})(1-A_{jl}).
\end{align} 
\begin{figure}[t]
\begin{center}
\includegraphics[scale=0.4]{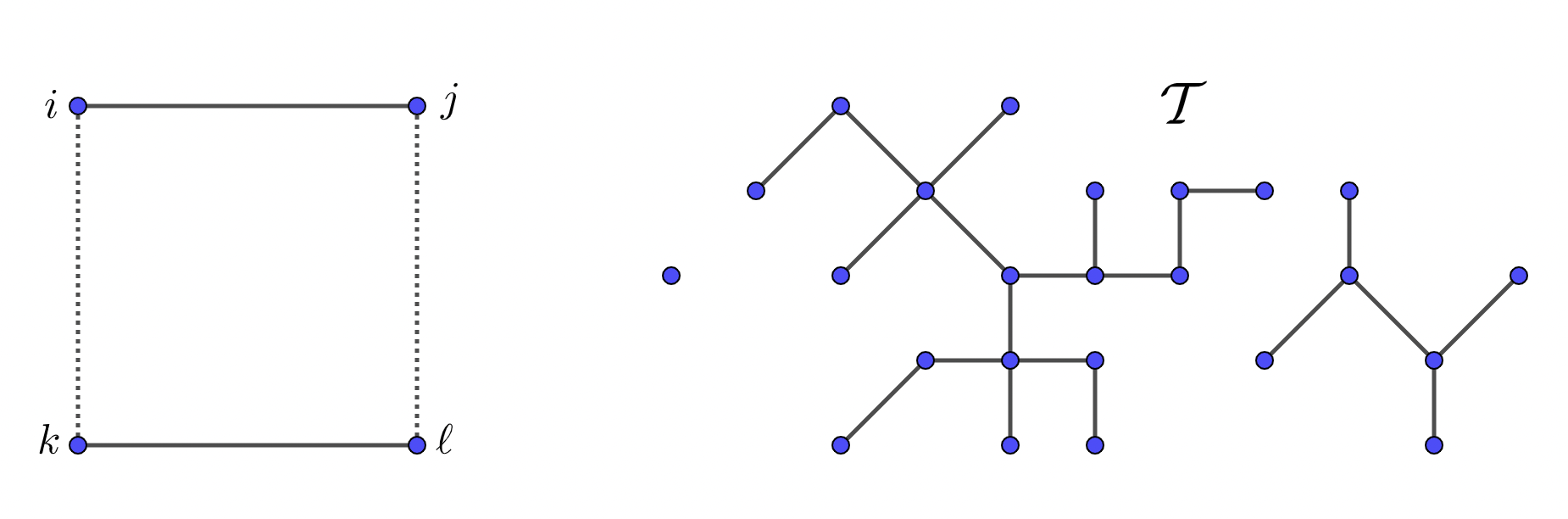}
\end{center}
\caption{Left Panel: A simple switching is given by replacing the solid edges by the dashed edges. Right Panel: An example of forest $\cT$, which consists of three connected components, i.e. $\theta(\cT)=3$.
\label{fig:switch1}}
\end{figure}

The following concept of \emph{forest} will be used throughout the paper.
\begin{definition}(Forest).\label{d:forest}
By a forest we mean a finite simple graph which is a union of trees: $\cT=({\bf i}(\cT), E(\cT))=({\bf i},E)$. Here ${\bf i}\in\qq{N}$ is a finite set of distinct vertices, $E$ is a finite set of edges, and each edge $e\in E$ connects $\{i_e, i'_e\}\in {\bf i}$. We denote the number of connected components of $\cT$ as $\theta(\cT)$. We remark that $\cT$ may contain singletons. See Figure \ref{fig:switch1} for an example.
\end{definition}

Given a forest $\cT=({\bf i},E)$, we denote 
\begin{align*}
A_{\cT}=\prod_{e\in E} A_{i_e i'_e},
\end{align*}
If $\cT$ consists of only singletons, i.e. $E=\emptyset$, then we simply define $A_{\cT}=1$. 
In the rest of this section, we use switchings to estimate terms in the following forms
\begin{align*}
\bE\left[A_{\cT}F(A)\right],\quad \bE\left[A_{\cT}A_{ik}F(A)\right],\quad \bE\left[A_{\cT}A_{jk}F(A)\right],\quad  i\in \bfi, j,k\not\in \bfi
\end{align*}
where $F$ is any function depending on the random graph $A$.
(Later, we shall take $F$ to be a polynomial of the Green's function entries $\{G_{ij}\}_{i,j\in\qq{N}}$ and the Stieltjes transform $m$.)

\begin{proposition}\label{p:bijection}
Given a forest $\cT=({\bf i},E)$, and a function $F$ which depends on the random graph $A$.
For each edge $e\in E$, we associate another edge $\{j_e,j'_e\}\in \qq{N}$, and denote ${\bf j}=\{j_e,j'_e\}_{e\in E}$.
If the indices $\bf i, \bf j$ are distinct, we have the identity
\begin{align} \label{e:bijection}
\bE\left[F(A) \prod_{e\in E}\chi_{i_ei'_e}^{j_e j'_e}(A)\right]
=\bE\left[F\left(A+\sum_{e\in E}\xi_{i_ei'_e}^{j_e j'_e}\right) \prod_{e\in E}\chi_{i_ej_e}^{i'_e j'_e}(A)\right].
\end{align}
If the index $i\in \bfi$ and  the indices ${\bf i}jkm\ell$ are distinct, we have the identities
\begin{align} \begin{split}\label{e:bijection2}
&\bE\left[A_{\cT}\chi_{ik}^{m\ell}(A)F(A)\right]
=\bE\left[A_{\cT}\chi_{im}^{k\ell}(A)F(A+\xi_{ik}^{m\ell})\right],\\
&
\bE\left[A_{\cT}\chi_{jk}^{m\ell}(A)F(A)\right]
=\bE\left[A_{\cT}\chi_{jm}^{k\ell}(A)F(A+\xi_{jk}^{m\ell})\right],
\end{split}\end{align}
where the indicator function $\chi$ is as defined in \eqref{e:defchi}.
\end{proposition}

\begin{proof}
Define the sets of graphs
\begin{align*}\begin{split}
\cG_1=\left\{A:  \prod_{e\in E}\chi_{i_ei'_e}^{j_e j'_e}(A)=1\right\},\quad \cG_2=\left\{A:  \prod_{e\in E}\chi_{i_ej_e}^{i'_e j'_e}(A)=1\right\}.
\end{split}\end{align*}
By our assumption that the indices $\bfi, \bfj$ are distinct, there is a simple bijection between $\cG_1$ and $\cG_2$, namely
\begin{align}\label{e:G1G2bij}
A\in \cG_2\mapsto A+\sum_{e\in E}\xi_{i_ei'_e}^{j_e j'_e}\in \cG_1.
\end{align}
Since $\bP$ is the uniform probability measure on $d$-regular graphs, the claim follows from \eqref{e:G1G2bij}. 

The two identities in \eqref{e:bijection2} can be proven in the same way, we will only give the proof of the first one. We define the sets of graphs
$
\cG_1=\left\{A:  A_{\cT}\chi_{ik}^{m\ell}(A)=1\right\}$ and $\cG_2=\left\{A:  A_{\cT}\chi_{im}^{k\ell}(A)=1\right\}.
$
By our assumption that the indices $\bfi km \ell$ are distinct, there is a simple bijection between $\cG_1$ and $\cG_2$:
$
A\in \cG_2\mapsto A+\xi_{ik}^{m\ell}\in \cG_1,
$
and the claim follows.

\end{proof}
\subsection{Discrete Integration by Part}

We define the discrete derivative operators induced by the simple switchings, which will be used in the rest of this paper.
\begin{definition} \label{d:defD}
For any symmetric matrix $\xi= \sum_{a=1}^b \xi_{i_aj_a}^{k_a\ell_a}$ (recall from \eqref{e:defxi}) for some fixed $b\in \bN$, we denote the vertex set $\bm V(\xi)=\cup_{a=1}^b\{i_a, j_a, k_a,\ell_a\}$ and define the discrete and continuous derivative operator as
\begin{equation} \label{e:D-defxi}
D_{\xi}F(A) \deq F(A+\xi)-F(A),\quad \partial_{\xi} F(A) \deq \sqrt{d-1}\left.\partial_t  F(A + t \xi)\right|_{t = 0}.
\end{equation}
For a sequence of matrices $\xi_1,\xi_2,\cdots, \xi_c$, we denote $\bm\xi=(\xi_1, \xi_2,\cdots, \xi_c)$ and 
\begin{align}\label{e:defbmxi}
\del_{\bm\xi}=\del_{\xi_1}\del_{\xi_2}\cdots \del_{\xi_c},\quad \bm V(\bm\xi)=\cup_{a=1}^c \bm V(\xi_a),\quad |\bm\xi|=c.
\end{align}
\end{definition}
Note that   the operation $D_{\xi}$ switches all edges in $\xi$ at once. The continuous derivative  
 $\partial_{\xi}$ is the directional derivative in the direction $\xi$ of the rescaled variable $H=A/\sqrt{d-1}$. 
For the discrete derivative operator $D_{\xi}$, we have the discrete product rule and the commutative property
\begin{equation} \label{e:D-product}
  D_{\xi}(FG) = (D_{\xi}F)G + F(D_{\xi}G) +  (D_{\xi}F)(D_{\xi}G),\quad D_{\xi_1}D_{\xi_2} F=D_{\xi_2}D_{\xi_1} F,
\end{equation}
If $F$ is continuous differentiable up to $\fb$-th order, the Taylor expansion with remainder gives
\begin{align} \label{e:D-expand}
 D_{\xi} F(A) = \sum_{n=1}^{\fb-1}  \frac{\del_{\xi}^nF(A)}{n!(d-1)^{n/2}} +  \frac{\del_{\xi}^{\fb}F(A+\theta\xi)}{\fb!(d-1)^{\fb/2}} ,
  \end{align}
for some $0\leq \theta\leq 1$.

In the switching in \eqref{e:bijection}, the indicator function $\prod_{e\in E}\chi_{i_ej_e}^{i'_e j'_e}(A)$
enforces that the matrix $A+\sum_{e\in E}\xi_{i_ei'_e}^{j_e j'_e}$ is again the adjacency matrix of a simple graph.
Note that without this indicator function, such as in the following corollaries and elsewhere throughout our proof,
$A+\sum_{e\in E}\xi_{i_ei'_e}^{j_e j'_e}$ is not necessarily the adjacency matrix of a simple graph, just a real symmetric matrix. This does however not affect our arguments, which should be viewed as operating with general symmetric matrices instead of adjacency matrices of simple graphs.

\begin{proposition}\label{c:intbp}
Given a forest $\cT=({\bf i},E)$, for each edge $e\in E$, we associate another edge $\{j_e,j'_e\}\in \qq{N}$, and denote ${\bf j}=\{j_e,j'_e\}_{e\in E}$.
For any function $F_{\bf i}(A)$ depending on the vertices $\bfi$ and the adjacency matrix $A$, we denote
\begin{equation} \label{def_CA}
\cal C(F,A)\deq \max_{\bfi\bfj}
\pa{
\left|F_{\bf i}(A)\right|
+\left|F_{\bf i}\left(A+\sum_{e\in E}\xi_{i_ei'_e}^{j_e j'_e}\right)\right|}.
\end{equation}
Then we have the discrete integration by parts formula (recall the summation $\sum^*$ from \eqref{e:starsum})
\begin{align}\begin{split} \label{e:intbp}
&\phantom{{}={}}\frac{1}{N^{\theta(\cT)}d^{|E|}}\sum^*_{\bfi}\bE\left[A_{\cT}F_{\bfi}(A)\right]=\frac{1}{N^{|\bfi|}}\sum^*_{\bfi}\bE\left[F_{\bfi}(A)\right]\\
&+
\frac{1}{N^{\theta(\cT^+)}d^{|E^+|}}\sum^*_{\bfi \bfj}\bE\left[A_{\cT^+} D_{\cT^+}F_{\bfi}(A)\right]+\OO\left(
\frac{d\bE[\cal C(F, A)]}{N}\right),
\end{split}\end{align}
where $\cT^+=(\bfi\bfj, E^+)$ with $E^+=\cup_{e\in E}\{i_e, j_e\}\cup\{i_e', j_e'\}$, $\xi=\sum_{e\in E}\xi_{i_ei'_e}^{j_e j'_e}$,
and 
$
D_{\cT^+}=D_{\xi} 
$
from \eqref{e:D-defxi}.
\end{proposition}
Notice that the right hand side of \eqref{e:intbp} can be interpreted as the expectation over the switched graph, since $A_{\cT^+}$ is the indicator function that the graph $A$ has edges  $E^+=\cup_{e\in E}\{i_e, j_e\}\cup\{i_e', j_e'\}$. 
The matrix $\xi=\sum_{e\in E}\xi_{i_ei'_e}^{j_e j'_e}$ is the indicator function for the difference of edges of the  unswitched graphs and the switched graph. 

\begin{proposition}\label{c:intbp2}
Given a forest $\cT=({\bfi},E)$, and fix an index $i\in \bfi$. 
For any function $F_{\bfi k}(A)$ depending on the vertices $\bfi k$ and the adjacency matrix $A$,  we have the integration by parts formula
\begin{align}\begin{split} \label{e:intbp2}
&\phantom{{}={}}\frac{1}{N^{\theta(\cT)}d^{|E|+1}}\sum^*_{\bfi k}\bE\left[A_{\cT}A_{ik}F_{\bfi k}(A)\right]=\frac{1}{N^{\theta(\cT)+1}d^{|E|}}\sum^*_{\bfi k}\bE\left[ A_{\cT} F_{\bfi k}(A)\right]\\
&+
\frac{1}{N^{\theta(\cT^+)}d^{|E^+|}}\sum^*_{\bfi k \ell m}\bE\left[A_{\cT^+} D_{\cT^+}F_{\bfi k}(A) \right]+\OO\left(
\frac{d\bE[\cal C(F, A)]}{N}\right),
\end{split}\end{align}
where $\cT^+=(\bfi k m\ell, E^+)$ with $E^+=E\cup\{i,m\}\cup \{k,\ell\}$, $\xi=\xi_{ik}^{m\ell}$, $\cal C(F,A)\deq \max_{\bfi k m\ell}
(
|F_{\bfi k}(A)|
+|F_{\bfi k}(A+\xi_{ik}^{m\ell})|)$, and 
$
D_{\cT^+}=D_{\xi} 
$
from \eqref{e:D-defxi}.

For any function $F_{\bfi jk}(A)$ depending on the vertices $\bfi jk$ and the adjacency matrix $A$,  we have the integration by parts formula
\begin{align}\begin{split} \label{e:intbp3}
&\phantom{{}={}}\frac{1}{N^{\theta(\cT)+1}d^{|E|+1}}\sum^*_{\bfi jk}\bE\left[A_{\cT}A_{jk}F_{\bfi jk}(A)\right]=\frac{1}{N^{\theta(\cT)+2}d^{|E|}}\sum^*_{\bfi jk}\bE\left[ A_{\cT} F_{\bfi jk}(A)\right]\\
&+
\frac{1}{N^{\theta(\cT^+)}d^{|E^+|}}\sum^*_{\bfi jk \ell m}\bE\left[A_{\cT^+} D_{\cT^+}F_{\bfi jk}(A) \right]+\OO\left(
\frac{d\bE[\cal C(F, A)]}{N}\right),
\end{split}\end{align}
where $\cT^+=(\bfi j k m\ell, E^+)$ with $E^+=E\cup\{j,m\}\cup \{k,\ell\}$, $\xi=\xi_{jk}^{m\ell}$, and $\cal C(F,A)\deq \max_{\bfi k m\ell}
(
|F_{\bfi jk}(A)|
+|F_{\bfi jk}(A+\xi_{jk}^{m\ell})|)$,
and 
$
D_{\cT^+}=D_{\xi} 
$
from \eqref{e:D-defxi}.
Here the vertices in $\bfi$ do not participate in the switching in any way.
\end{proposition}

\begin{remark}\label{r:Tstructure}
The following properties of the graph $\cT^+$ constructed in Proposition \ref{c:intbp} will be useful in later part of this paper.
\begin{enumerate}
\item The graph $\cT^+$ consists of $|\bfi|$ connected components, i.e. $\theta(\cT^+)=|\bfi|$. Each vertex $i\in \bfi$ corresponds to one connected component in $\cT^+$. If the vertex $i$ has degree $\deg_{\cT}(i)$ in $\cT$, the corresponding component is a $2$-level tree with $i$ as the root vertex, and $\deg_{\cT}(i)$ children vertices.
\item For any two distinct vertices $i,i'\in \bfi$,  $i$ and $i'$ are in different connected components of $\cT^+$.
\end{enumerate}
\end{remark}

\begin{figure}[t]
\begin{center}
\includegraphics[scale=0.4]{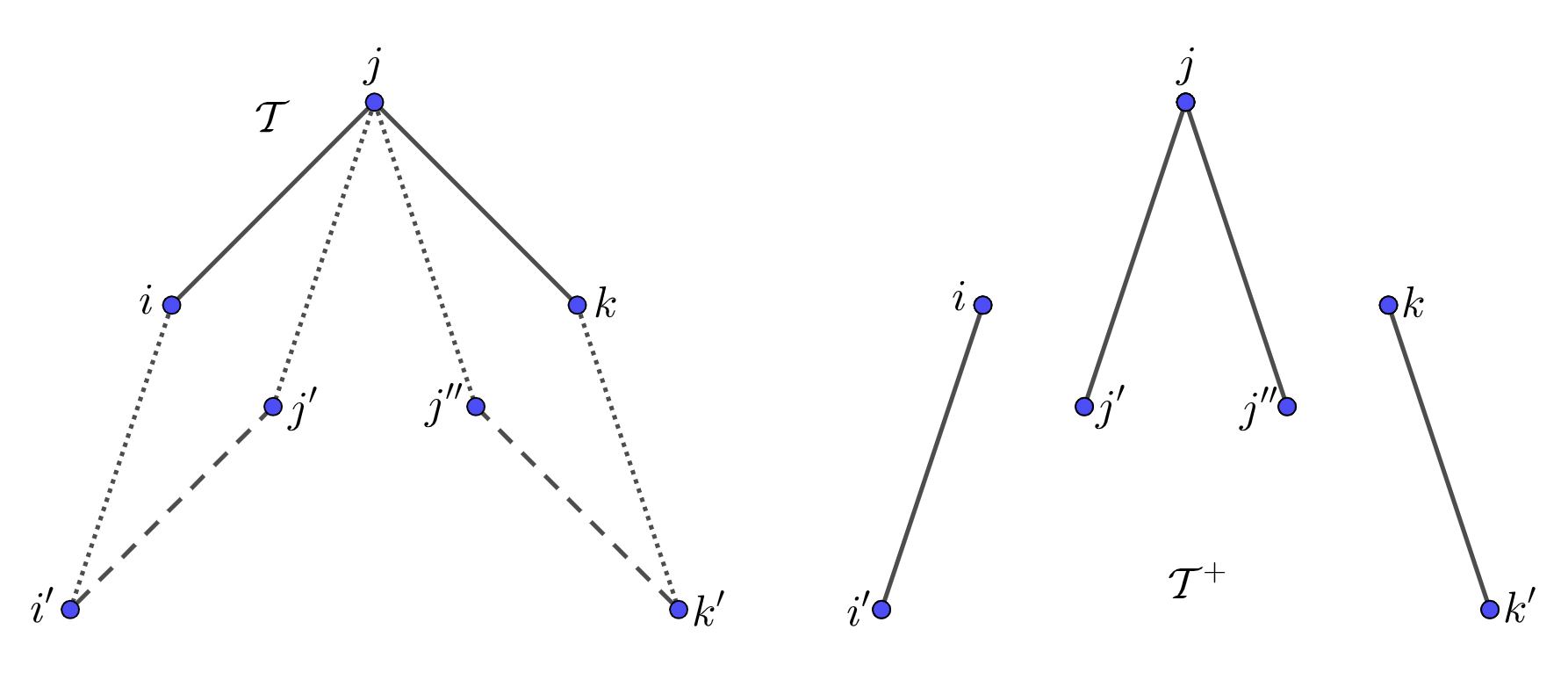}
\end{center}
\caption{Take $\cT=\{\bfi=\{i,j,k\}, E=\{(i,j), (j,k)\}\}$ to be a forest, then $\cT^+$ constructed in Proposition \ref{c:intbp} is given by $\cT^+=\{\bfi^+=(i,i', j,j',j'', k,k'), E^+=\{(i,i'), (j,j'), (j,j''),(k,k')\}\}$.
\label{fig:construct}}
\end{figure}

\begin{example}
Take $\cT=\{\bfi=\{i,j,k\}, E=\{(i,j), (j,k)\}\}$ to be a forest, which is a path of length two (See Figure \ref{fig:construct}). It has only one connected component, i.e. $\theta(\cT)=1$. 
$\cA_\cT=A_{ij}A_{jk}$. The summation 
\begin{align*}
\sum_{ijk}^* A_\cT=\sum_{ijk \text{distinct}}A_{ij}A_{jk},
\end{align*}
which counts the number of ways to embed $\cT$ to $\cG$, where vertices of $\cT$ are mapped to distinct vertices $\qq{N}$ of $\cG$. The forest $\cT^+$ constructed in Proposition \ref{c:intbp} is given by $\cT^+=\{\bfi^+=(i,i', j,j',j'', k,k'), E^+=\{(i,i'), (j,j'), (j,j''),(k,k')\}\}$, which has three connected components; and $i,j,k$ are in different connected components of $\cT^+$. 
\end{example}

\begin{proposition}\label{p:sumA}
For any graph $\cT=(\bfi, E)$, if it is a forest, we have 
\begin{align}\label{e:sumA}
\sum_{\bfi} A_{\cT}= N^{\theta(\cT)}d^{|E|},
\end{align}
and
\begin{align}\label{e:sumA2}\begin{split}
&\frac{1}{N^{\theta(\cT)} d^{|E|}}\sum^*_{\bfi} A_{\cT}=c_{\cT}+\OO_\prec\left(\frac{d}{N}\right),\\
& c_\cT=\frac{\prod_{i\in \bfi: \deg(i)\geq 1}(d-1)(d-2)\cdots (d-\deg(i)+1) }{d^{|E|-\theta(\cT)}}\asymp 1.
\end{split}\end{align}
By definition, the numerator of $c_\cT$ is equal to $1$  if the $\deg(i)=1$. 
If $\cT$ contains exactly one cycle, then we have
\begin{align}\label{e:sumAtriangle}
\sum^*_{\bfi} A_{\cT}\prec N^{\theta(\cT)-1} d^{|E|}.
\end{align}
\end{proposition}

\begin{remark}\label{r:changec}
Given a forest $\cT=(\bfi, E)$. $\cT^+=(\bfi^+, E^+)$ is as constructed in Propositions \ref{c:intbp} or \ref{c:intbp2}. Since $\cT^+$ and $\cT$ have the same degree profile (except for vertices with degree one), we have $c_{\cT}=c_{\cT^+}$ as in \eqref{e:sumA2}. 
\end{remark}

\begin{proof}[Proof of Proposition \ref{c:intbp}]
Since the row and column sums of $A$ equal $d$, by introducing new indices $\{j_e, j_e'\}_{e\in E}$,
we rewrite the left-hand side of \eqref{e:intbp} as
\begin{align}\begin{split}\label{e:newindices}
&\phantom{{}={}}\frac{1}{N^{\theta(\cT)}d^{|E|}}\sum^*_{\bfi}\bE\left[A_{\cT}F_{\bfi}\right]=\frac{1}{N^{\theta(\cT)+|E|}d^{2|E|}}\sum^*_{\bfi}\sum_{\bfj}\bE\left[A_{\cT}\prod_{e\in E} A_{j_ej_e'}F_{\bfi}\right]\\
&=\frac{1}{N^{\theta(\cT)+|E|}d^{2|E|}}\left(\sum^*_{\bfi \bfj}\bE\left[\prod_{e\in E}\chi_{i_ei'_e}^{j_ej'_e}(A)F_{\bf i}\right]+\sum^*_{\bfi}
\sum_{\substack{\bfj: \bfi\bfj \\ \txt{not distinct}}}\bE\left[A_{\cT}\prod_{e\in E} A_{j_ej_e'}F_{\bf i}\right]\right.\\
&
\left.+\sum^*_{\bfi\bfj}\bE\left[\left(A_{\cT}\prod_{e\in E} A_{j_e j'_e}-\prod_{e\in E}\chi_{i_ei'_e}^{j_ej'_e}(A)\right)F_{\bf i}\right]\right).
\end{split}\end{align} 
The second term of the right-hand side in \eqref{e:newindices} can be estimated by
\begin{align}\begin{split} \label{e:avt1}
&\phantom{{}={}}\left|
\frac{1}{N^{\theta(\cT)+|E|}d^{2|E|}}
\sum^*_{\bfi}
\sum_{\substack{\bfj: \bfi\bfj \\ \txt{not distinct}}}A_{\cT}\prod_{e\in E} A_{j_ej_e'}F_{\bf i}(A)
\right|\\
&\leq \frac{1}{N^{\theta(\cT)+|E|}d^{2|E|}}
\sum^*_{\bfi}
\sum_{\substack{\bfj: \bfi\bfj \\ \txt{not distinct}}}A_{\cT}\prod_{e\in E} A_{j_ej_e'}\cal C(F,A)\leq  \frac{1}{N}  \cal C(F,A),
\end{split}\end{align}
where in the last line $\bfi\bfj$ are not distinct, either some $j_e=j_e'$ then $A_{j_ej_e'}=0$, or $\cT\cup_{e\in E}\{j_e,j_e'\}$ has less than $\theta(\cT)+|E|$ connected components. By \eqref{e:sumA}, in this case the sum over $\bfi\bfj$ is bounded by $\OO(N^{\theta(\cT)+|E|-1}d^{2|E|})$.
Similarly, using $ |A_{ij}A_{k\ell}- \chi_{ij}^{k\ell}(A)|=A_{ij}A_{k\ell}[A_{ik}+A_{jl}-A_{ik}A_{jl}]\leq A_{ij}A_{k\ell}[A_{ik}+A_{jl}]$, we also have
\begin{align}\label{e:avt2}
\left|\frac{1}{N^{\theta(\cT)+|E|}d^{2|E|}}\sum^*_{\bf ij}\left[\left(A_{\cT}\prod_{e\in E} A_{j_e j'_e}-\prod_{e\in E}\chi_{i_ei'_e}^{j_ej'_e}(A)\right)F_{\bf i}\right]
\right|\leq \frac{d}{N}\cal C(F,A).
\end{align}
By plugging the estimates \eqref{e:avt1} and \eqref{e:avt2} into \eqref{e:newindices}, we get
\begin{align}\begin{split}\label{e:newindices2}
&\phantom{{}={}}\frac{1}{N^{\theta(\cT)}d^{|E|}}\sum^*_{\bfi}\bE\left[A_{\cT}F_{\bfi}\right]
=\frac{1}{N^{\theta(\cT)+|E|}d^{2|E|}}\sum^*_{\bfi \bfj}\bE\left[\prod_{e\in E}\chi_{i_ei'_e}^{j_ej'_e}(A)F_{\bf i}(A)\right]
+\OO\left(\frac{d\bE[\cal C(F,A)]}{N}\right)\\
&=\frac{1}{N^{\theta(\cT)+|E|}d^{2|E|}}\sum^*_{\bf ij}\bE\left[F\left(A+\sum_{e\in E}\xi_{i_ei'_e}^{j_e j'_e}\right)\prod_{e\in E}\chi_{i_ej_e}^{i'_e j'_e}(A)\right]+\OO\left(\frac{d\bE[\cal C(F,A)]}{N}\right),
\end{split}\end{align}
where we used Proposition \ref{p:bijection} in the last line. 
By Proposition \ref{p:bijection} and estimates analogous to the ones above, we have
\begin{align}\begin{split}\label{e:switchindices}
&\phantom{{}={}}\frac{1}{N^{|\bfi|}}\sum^*_{\bfi}\bE\left[F_{\bfi}\right]
=
\frac{1}{N^{\theta(\cT)+|E|}d^{2|E|}}\sum^*_{\bfi}\sum_{\bfj}\bE\left[\prod_{e\in E} A_{i_e j_e}A_{i'_ej_e'}F_{\bfi}\right]\\
&=\frac{1}{N^{\theta(\cT)+|E|}d^{2|E|}}\sum^*_{\bfi \bfj}\bE\left[\prod_{e\in E}\chi_{i_ej_e}^{i'_ej'_e}(A)F_{\bf i}(A)\right]
+\OO\left(\frac{d\bE[\cal C(F,A)]}{N}\right),
%&=\frac{1}{N^{\theta(\cT)+|E|}d^{2|E|}}\sum^*_{\bf ij}\bE\left[F\left(A+\sum_{e\in E}\xi_{i_ei'_e}^{j_e j'_e}\right)\prod_{e\in E}\chi_{i_ej_e}^{i'_e j'_e}(A)\right]+\OO\left(\frac{d\bE[\cal C(F,A)]}{N}\right).
\end{split}\end{align}
where in the first equality we used \eqref{e:sumA} that $\cT$ is a forest and $\theta(\cT)=|\bfi|-|E|$.
The new graph $\cT^+=(\bfi\bfj, E^+)$ with $E^+=\cup_{e\in E} \{i_e, j_e\}\cup\{i_e', j_e'\}$ has $\theta(\cT^+)=|\bfi|=|E|+\theta(\cT)$ connected components, and $2|E|$ edges. 
The claim now follows from combining \eqref{e:newindices2} and \eqref{e:switchindices}, and replacing $\prod_{e\in E}\chi_{i_ej_e}^{i'_ej'_e}(A)$ to $\prod_{e\in E} A_{i_e j_e}A_{i_e' j_e'}=A_{\cT^+}$.
\end{proof}

\begin{proof}[Proof of Proposition \ref{c:intbp2}]
The two estimates \eqref{e:intbp2} and \eqref{e:intbp3} can be proven in the same way as for Proposition \ref{c:intbp}, by using \eqref{e:bijection2} to do the discrete integration by part. We omit their proofs.  
\end{proof}

\begin{proof}[Proof of Proposition \ref{p:sumA}]
In the special case that $ \cT$ is a tree, for the sum $\sum_\bfi A_{ \cT}$, we can sequentially sum over leaf vertices. Each gives a factor $d$, and the sum over the root vertex gives a factor $N$. Thus
$\sum_{\bfi} A_{ \cT}= N d^{|\bfi|-1}$. In the general case, when $ \cT$ is a forest. We can sum over each of its connected components. There are $\theta(\cT)$ connected components, and $|E|=|\bfi|-\theta(\cT)$, so \eqref{e:sumA} follows.

We assume that $\cT$ contains exactly one cycle $i_1 i_2 \cdots i_\ell$ of length $\ell$, we can sequentially sum over other vertices as in \eqref{e:sumA}
\begin{align}\label{e:sumAtriangle2}
\sum^*_{\bfi} A_{\cT}\leq  N^{\theta(\cT)-1} d^{|E|-\ell}\sum^*_{i_1,i_2, \cdots, i_{\ell}}A_{i_1i_2}A_{i_2i_3}\cdots A_{i_\ell i_1}=N^{\theta(\cT)-1} d^{|E|-\ell}Z,
\end{align}
where $Z$ counts the number of cycles of length $\ell$. To upper bound the righthand side of \eqref{e:sumAtriangle2}, we recall from \cite[Theorem 3]{mckay2004short}: for any finite subgraph $\cJ$ with $|\cJ|$ edges, 
\begin{align}\label{e:ppbb}
\bP(\cJ\subset \cG)=(1+\oo(1))\left(\frac{d}{n}\right)^{|\cJ|}.
\end{align}
Using \eqref{e:ppbb}, for any $k\geq 1$, $\bE[Z^k]\leq C_k d^{k\ell}$. It follows by Markov's inequality
\begin{align*}
\bP(Z\geq d^\ell N^\varepsilon)\leq \frac{C_k}{N^{\varepsilon k}}.
\end{align*}
We conclude that $Z\prec d^\ell$, and the claim \eqref{e:sumAtriangle} follows from combining with \eqref{e:sumAtriangle2}. 

We prove \eqref{e:sumA2} by induction. If $\cT$ consists of singletons, i.e. $E=\emptyset$, then $\sum^*_{\bfi} A_{\cT}/N^{\theta(\cT)}=c_{\cT}=1$. Otherwise, there exists some vertex $i$, such that all of  its neighbors  are leaf vertices, except possibly one. We can decompose $\cT$ in the following way:  $\cT'=(\bfi', E')$, $i\in \bfi'$ with $\deg_{\cT'}(i)\leq 1$ and $\bfi=\bfi'\bfj, \bfj=\{j_1, j_2,\cdots, j_b\}$, $E=E'\cup\{i,j_1\}\cup \cdots \{i,j_b\}$. 
\begin{align}\label{e:sumAj}
\sum^*_{\bfi} A_{\cT}=\sum^*_{\bfi'} A_{\cT'}\sum^*_{\bfj: \bfi'\bfj \text{ distinct}}A_{ij_1}A_{ij_2}\cdots A_{ij_b}.
\end{align}

The two cases that $\deg_{\cT'}(i)=0$ and $\deg_{\cT'}(i)=1$ are similar. We will only study the case when $\deg_{\cT'}(i)=1$. Say $A_{ii'}=1$ with $i'\in \bfi'$, then $\deg_{\cT}(i)=b+1$, and from the definition of $c_\cT$ in \eqref{e:sumA2} we obtain
\begin{align}\label{e:cT}
c_{\cT}=c_{\cT'}\frac{(d-1)(d-2)\cdots (d-b)}{d^{b}}.
\end{align}
We enumerate the neighbors of vertex $i$ in $\cG$ as $i', k_1, k_2, \cdots, k_{d-1}$. If $\{k_1, k_2,\cdots, k_{d-1}\}\cap \bfi'=\emptyset$, then $j_1, j_2,\cdots, j_b$ can take any values in $k_1, k_2, \cdots, k_{d-1}$, and $\sum^*_{\bfj: \bfi'\bfj \text{ distinct}}A_{ij_1}A_{ij_2}\cdots A_{ij_b}=(d-1)(d-2)\cdots (d-b)$. This corresponds to the first term in the following decomposition.
\begin{align}\label{e:sumAj2}
\sum^*_{\bfi} A_{\cT}
=\sum^*_{\bfi} A_{\cT}\prod_{k\in \bfi'\setminus\{i'\}}(1-A_{ik})+\sum^*_{\bfi} A_{\cT}\left(1-\prod_{k\in \bfi'\setminus\{i'\}}(1-A_{ik})\right).
\end{align}
For the second term on the righthand side of \eqref{e:sumAj2}, we notice that $\left|1-\prod_{k\in \bfi'\setminus\{i'\}}(1-A_{ik})\right|\lesssim \sum_{k\in \bfi'\setminus \{i'\}}A_{ik}$. By adding the edge $\{i,k\}$ with $k\in \bfi'\setminus \{i'\}$ to $\cT$, it either forms a cycle or reduces the number of connected component by $1$. Thus from the discussing above, using \eqref{e:sumA} and \eqref{e:sumAtriangle}, we get
\begin{align}\label{e:bb}
\sum^*_{\bfi} A_{\cT}=(d-1)(d-2)\cdots (d-b)\sum^*_{\bfi'} A_{\cT'}+\OO_{\prec}(N^{\theta(\cT)-1}d^{|E|+1}).
\end{align}
The statement \eqref{e:sumA2} follows from combining \eqref{e:cT} and \eqref{e:bb}.
%
%If $\deg_{\cT'}(i)=0$, then $\deg_{\cT}(i)=b$, and
%\begin{align}
%c_{\cT}=c_{\cT'}\frac{d(d-1)(d-2)\cdots (d-b+1)}{d^{b}}.
%\end{align}
%The claim \eqref{e:sumA2} follows from essentially the same argument as for \eqref{e:bb}, except now $\sum^*_{\bfj: \bfi'\bfj \text{ distinct}}A_{ij_1}A_{ij_2}\cdots A_{ij_b}=d(d-1)(d-2)\cdots (d-b+1)$.

% It has been proven \cite[Section 5.2]{gao2020triangles} that as $d=\oo(\sqrt{N})$, the number of triangles in a uniform random $d$-regular graphs is bounded by $\OO_\prec(d^3)$. The claim \eqref{e:sumAtriangle} follows from plugging $\sum^*_{ijk}A_{ij}A_{jk}A_{ki}\prec d^{3}$ into \eqref{e:sumAtriangle2}.

%is not a forest (as in Definition \ref{d:forest}), we denote $\tilde \cT$ the spanning forest of $\cT$. Then $\tilde\cT=(\bfi, \tilde E)$ has the same vertex set as $\cT$, the same number of connected components $\theta(\cT)=\theta(\tilde \cT)$ and $A_{\cT}\leq A_{\tilde \cT}$. In the special case that $\tilde \cT$ is a tree, for the sum $\sum_\bfi A_{\tilde \cT}$, we can sequentially sum over leaf vertices. Each gives a factor $d-1$, and the sum over the root vertex gives a factor. Thus
%$\sum_{\bfi} A_{\tilde \cT}\leq N (d-1)^{|\bfi|-1}$. In the general case, when $\tilde \cT$ is a forest. We can sum over each of its connected components. There are $\theta(\cT)$ connected components, and \eqref{e:sumA} follows.
\end{proof}

\section{Green's Function Estimates}\label{sec:GFE}

In this section we collect some estimates on the Green's function $G$ and the Stieltjes transform of the spectral measure $m$. These will be used repeatedly in the rest of the paper. We also introduce monomials in the Green's function entries, and record some of their basic properties.

By our definition, the Green's function $G=P_\perp (H-z)^{-1}P_\perp $ is symmetric and satisfies \eqref{sumG0}.
The \emph{Ward identity} states that the Green's function $G$ satisfies
\begin{equation} \label{e:WdI}
  \frac{1}{N} \sum_{j} |G_{ij}|^2 = \frac{\Im[G_{ii}]}{N\eta},
  \qquad
  \frac{1}{N} \sum_{ij} |G_{ij}|^2 = \frac{\Im[m]}{\eta};
\end{equation}

We record the following basic result, which follows from Theorem \ref{thm:rigidity} and we will use them repeatedly throughout the rest of the paper.
\begin{lemma} \label{l:basicestimates}
Fix small $\fc>0$, suppose that $N^\fc\leq d\leq N^{2/3-\fc}$. With overwhelming probability, uniformly for $z \in \mathbf D$ from \eqref{e:D}, 
\begin{align}\label{e:Lambound}
\max_{x\neq y}|G_{xy}(z)|\leq \Lambda_o(z),\quad \max_{xx}|G_{xx}(z)|\leq 2,  \quad \max_{x}\Im[G_{xx}(z)]\prec \Im[m(z)],
\end{align}
where $\Lambda_o$ is from Theorem \ref{thm:rigidity}. 
\end{lemma}
\begin{proof}
The first two relations in \eqref{e:Lambound} follow directly from Theorem \ref{thm:rigidity}, for the last relation, we use the delocalization of eigenvectors 
$
  \max_\alpha\max_i |u_\alpha(i)| \prec 1/\sqrt{N},
$
then
\begin{align*}
\Im[G_{xx}]
=\sum_{\al=2}^N \frac{u_\al(x)^2 \eta}{|\la_\al-z|^2}\prec \frac{1}{N}\sum_{\al=2}^N \frac{ \eta}{|\la_\al-z|^2}=\Im[m].
\end{align*}
\end{proof}

\begin{remark} \label{r:basicestimates}
We recall $\xi_{ij}^{k\ell}$ from \eqref{e:defxi}.
As a consequence of Lemma \ref{l:basicestimates}, for any fixed $b \in \bN$ we find using a resolvent expansion that the following holds with overwhelming probability
\begin{equation*}
\max_{xy} \absbb{G_{xy} (A + \xi)} \leq |G_{xy}(A)| + \OO(d^{-1/2}),\quad \xi=\sum_{a=1}^b \xi_{i_aj_a}^{k_a\ell_a}.
\end{equation*}
\end{remark}

In the following proposition we collect some estimates on the Stieltjes transform:
\begin{proposition}\label{p:dermN}
Fix $z\in \mathbf D$ from \eqref{e:D}. 
For any derivative operators $\del_{\xi_1}, \del_{\xi_2},\cdots, \del_{\xi_b}$ as in Definition \ref{d:defD}, with $\cup_{a=1}^b\bm V(\xi_a)=\{i,j\}\cup \bfm$, let $\del_{\bm\xi}=\del_{\xi_1}\del_{\xi_2},\cdots\del_{\xi_b}$. Then for $b\geq 1$
the derivative $\del_{\bm \xi}m (z)$ is a sum of terms in the following form
\begin{align}\label{e:dermN}
    \frac{\Im[m ]}{N\eta}\sum_{\bfk}G^a_{ij}(z) X_{\bfk\bfm i}(z)Y_{\bfk\bfm j}(z),
\end{align}
where $a\geq 0$, the index set $\bfk=\emptyset$ and $|X_{ k \bfm i}(z)|^2, |Y_{k \bfm j}(z)|^2=\OO_\prec(1)$ or $\bfk=\{k\}$ and $\sum_{k}|X_{ k \bfm i}(z)|^2,\sum_{k} |Y_{k \bfm j}(z)|^2=\OO_\prec(1)$.  It follows that 
\begin{align}\label{e:impp}
     |\del_{\bm \xi}m (z)|\prec \frac{\Im[m (z)]}{N\eta}.
\end{align}
\end{proposition}

\begin{proof} 
The derivative $\del_{\bm \xi}m $ is a sum of terms in the following form
\begin{align}\label{e:dermN2}
    \frac{1}{N}\sum_{k}G_{kv_1}G_{v_2 v_3}\cdots G_{v_{2b-2} v_{2b-1}} G_{v_{2b}k},
\end{align}
where $v_1,v_2,\cdots, v_{2b}\in ij\bfm$. For the Green's function entries in \eqref{e:dermN2} we can regroup them depending if they contain indices $i,j$
\begin{align*}
   \frac{1}{N} \sum_{k=1}^N G_{ij}^a \tilde X_{k\bfm i}  \tilde Y_{k\bfm j}, 
\end{align*}
where for each $G_{xy}$ in \eqref{e:dermN2}, if the index set $\{x,y\}$ only contains $i$ we put it in $\tilde X_{k\bfm i}$; if it only contains $j$ we put it in $\tilde Y_{k\bfm  j}$; if it does not contain $i,j$ we still put it in $\tilde X_{k\bfm j}$. 
%$G_{k v_1}$ if $v_1=i$ we put it in $\tilde X_{k\bm m i}$, otherwise we put it in $\tilde Y_{k\bm m i}$; for $G_{v_{2\ell}k }$ if $v_{2\ell}=i$ we put it in $\tilde X_{k\bm m i}$, otherwise we put it in $\tilde Y_{k\bm m i}$; for $G_{v_{s}v_{s+1}}$ if $\{v_s, v_{s+1}\}=\{i,j\}$ we seperate it out, otherwise if $i\in \{v_s, v_{s+1}\}$ we put it in $\tilde X_{k\bm m i}$, if $i\notin \{v_s, v_{s+1}\}$ we put it in $\tilde X_{k\bm m j}$. 

There are two cases. In the first case, each of $\tilde X_{k\bfm i}, \tilde Y_{k\bfm j}$ contains  one of $G_{k v_1}, G_{v_{2b}k}$; in the second case, both $G_{k v_1}, G_{v_{2b}k}$ are in $\tilde X_{k\bfm i}$ or $\tilde Y_{k\bfm j}$.

In the first case, say $G_{k v_1}$ is in $\tilde X_{k\bm m i}$, using  $|G_{v_s v_{s+1}}|\prec 1$ from Lemma \ref{l:basicestimates} and Ward identity \eqref{e:WdI} we have
\begin{align*}
    \sum_{k=1}^N |\tilde X_{k\bfm i}|^2 \prec \sum_{k=1}^N |G_{k v_1}|^2\prec \frac{\Im[m ]}{\eta}, \quad 
    \sum_{k=1}^N |\tilde Y_{k\bfm j}|^2 \prec \sum_{k=1}^N |G_{k v_{2b}}|^2\prec \frac{\Im[m ]}{\eta}.
\end{align*}
The claim \eqref{e:dermN} follows by taking  $\bfk=\{k\}$, $\sqrt{\Im[m ]/\eta}X_{\bfk \bfm i}=\tilde X_{k\bfm i}, \sqrt{\Im[m ]/\eta}Y_{\bfk \bfm j}=\tilde Y_{k\bfm j}$.

In the second case, say both $G_{k v_1}, G_{v_{2b}k}$ are in $\tilde X_{k\bfm i}$. Then $\tilde Y_{k\bfm j}=\tilde Y_{\bfm j}$ does not depend on the index $k$. Moreover, using  $|G_{v_s v_{s+1}}|\prec 1$ from Lemma \ref{l:basicestimates} and Ward identity \eqref{e:WdI}
\begin{align*}
    \left|\sum_k \tilde X_{k\bfm i}\right|\prec \sum_{k=1}^N |G_{k v_1}G_{v_{2b}k}|\leq \frac{1}{2}\sum_{k=1}^N |G_{k v_1}|^2+|G_{v_{2b}k}|^2\prec \frac{\Im[m ]}{\eta}.
\end{align*}
The claim \eqref{e:dermN} follows by taking $\bfk=\emptyset$, $(\Im[m ]/\eta) X_{\bfk\bfm i}=\sum_k \tilde X_{k\bfm i}, Y_{\bfk\bfm j}=\tilde Y_{\bfm j}$.
This finishes the proof of \eqref{e:dermN}. The bound \eqref{e:impp} follows from the decomposition \eqref{e:dermN}. 

%For the bound of $\del^{\bm\beta} m (z)$, we have 
%\begin{align}\label{e:ddmN}
%    \del^{\bm\beta} m (z)=\frac{1}{N}\sum_{i=1}^N\del^{\bm\beta} G_{ii}.
%\end{align}
%The derivative $\del^{\bm\beta}G_{ii}$ is a sum of monomials of Green's function entries in the form
%\begin{align}\label{e:onet}
%G_{ix_1}G_{x_2 x_3}\cdots G_{x_{2b-2} x_{2b-1}}G_{x_{2b} i},
%\end{align}
%where $x_1, x_2,\cdots, x_{2b}\in \bfi$ and $b=|\bm\beta|$. By plugging \eqref{e:onet} into \eqref{e:ddmN}, and use Ward identity \eqref{e:WdI}
%\begin{align}
%\frac{1}{N}\sum_{i}|G_{ix_1}G_{x_2 x_3}\cdots G_{x_{2b-2} x_{2b-1}}G_{x_{2b} i}|
%\prec \frac{1}{N}\sum_{i}|G_{ix_1}||G_{x_{2b} i}|
%\leq \frac{1}{N}\sum_{i}\frac{|G_{ix_1}|^2+|G_{x_{2b} i}|^2}{2}\prec \frac{\Im[m ]}{N\eta}.
%\end{align}

\end{proof}

A central object in our proof is the following notion of a polynomial in the Green's function entries. 
\begin{definition} \label{d:evaluation}
  Let $R = R(\{x_{st}\}_{s,t = 1}^r, y)$ be a monomial in the $r^2+1$ abstract variables $\{x_{st}\}_{s,t = 1}^r, y$. We denote its degree by $\deg(R)$. For ${\bf i} \in \qq{N}^r$, we define its \emph{evaluation} on the Green's function and the Stieltjes transform by
  \begin{equation*}
    R_{\bf i} = R(\{G_{i_s i_t}\}_{s,t = 1}^r, m),
  \end{equation*}
  and say that $R_{\bf i}$ is a \emph{monomial in the Green's function entries} $\{G_{i_s i_t}\}_{s,t = 1}^r$. 
  We denote the number of off-diagonal entries in $R_\bfi$ as  $\chi(R_\bfi)$.
    %  \item
%  Let $F = F(\{x_{st}\}_{s,t = 1}^r)$ be a monomial in $r^2$ variables.
%  Then the number of off-diagonal entries of $F$ is the total degree of variables $x_{st}$ with $s \neq t$. If the number of off-diagonal entries of $F$ is zero then we define $\chi_F = 1$, otherwise we define $\chi_F = 0$.
\end{definition}

%We conclude this section with an elementary result for the operator $\prec$, which we shall use tacitly throughout the following sections.
%
%\begin{lemma} \label{lem:prec_exp}
%Suppose that $A$ and $B$ are nonnegative random variables satisfying $A \leq N^C$ and $B \geq N^{-C}$ for some constant $C > 0$. Then $A \prec B$ implies $\bE [A] \prec \bE [B]$.
%\end{lemma}

\section{Self-Consistent Equation}
\label{sec:P-construct}

In this section we construct the self-consistent equation for the Stieltjes transform $m$, and prove that the self-consistent equation holds with overwhelming probability.

\begin{proposition}\label{p:DSE}
Assume $N^\fc\leq d \leq N^{1/3-\fc}$.  There exists a finite degree polynomial, depending on $d$ but not $N$,
\begin{align} \label{e:P_split}
P(z,u)=1+zu+u^2+Q(u),
\end{align}
where
\begin{align} \label{e:Q_a}
Q(u)= a_2 u^2+a_3 u^3 +a_4 u^4+\cdots+a_n u^n,
\end{align}
is a polynomial with bounded coefficients $a_2,a_3,\dots, a_n=\OO(1/\sqrt d)$ such that, for any $z \in \mathbf D$ from \eqref{e:D}
\begin{align}\label{e:defP}
\bE[|P(z, m(z))|^{2q}]\prec \bE[\Phi_q(z)].
\end{align}
where the control parameter $\Phi_q(z)$ is defined as
\begin{align}\label{e:defPhir}
\begin{split}
\Phi_q(z):&= \frac{d^{3/2}\Lambda_o(z)}{N}|P(z,m(z))|^{2q-1}+\sum_{s= 1}^{2q}\left(\frac{\Im[m(z)]}{N\eta}\right)^s|P( z, m(z))|^{2q-s}\\
&+\sum_{s=1}^{2q-1}\left(\frac{1}{N\eta}+\frac{d^{3/2}\Lambda_o(z)}{N}\right)\left( \frac{\Im[m(z)]|\del_2 P(z,m(z))|}{N\eta}\right)^s |P( z, m(z))|^{2q-s-1}.
\end{split}\end{align}
\end{proposition}

We will show that the estimate \eqref{e:defP} results from the switching invariance of random regular graphs.
% and may hence be viewed as a resulting approximate \emph{Schwinger--Dyson Equation} for the random regular graph ensemble.
 It may be viewed as an approximate \emph{Schwinger--Dyson Equation} for the random regular graph ensemble; in statistical mechanics and field theory, such equations are typically derived by integration by parts. In the rest of the paper, we write $ \del_2 P (z,w) \deq \partial_w  P (z,w)$.

Starting from \eqref{e:GHexp} and noticing $A_{ii}=0$ we obtain
\begin{equation*}
1+zm=\sum^*_{ij}\frac{A_{ij}G_{ij}}{N(d-1)^{1/2}} + \frac{1}{N}.
\end{equation*}
By multiplying both sides by $P^{q-1}\bar P^q$, we get
\begin{align}\begin{split}\label{e:EPP}
\bE[(1+zm)P^{q-1}\bar P^q] 
&=  \frac{1}{N (d - 1)^{1/2}} \sum_{ij}^* \bE [A_{ij}(G_{ij}P^{q-1}\bar P^q)] +\OO_\prec(\bE[ \Phi_q])\\
&=  \frac{d}{(d-1)^{1/2}}\frac{1}{N^{\theta(\cT_0)} d^{|E_0|}} \sum_{ij}^* \bE [A_{\cT_0} R_{\bfi_0}P^{q-1}\bar P^q] +\OO_\prec(\bE[ \Phi_q]),
\end{split}\end{align}
where $\cT_0=(\bfi_0={ij},E_0=\{i,j\})$, $A_{\cT_0}=A_{ij}$ and $R_{\bfi_0}=G_{ij}$.

In Sections  \ref{s:prove} and \ref{s:pre}, we will show that we can use Propositions \ref{c:intbp} and \ref{c:intbp2} to further expand \eqref{e:EPP}. In this way, starting from the graph $\cT_0=(\bfi_0={ij},E_0=\{i,j\})$, we will obtain  a sequence of graphs 
\begin{align*}
\cT_0=(\bfi_0, E_0),\quad  \cT_1=(\bfi_1, E_1), \quad \cT_2=(\bfi_2, E_2), \quad \cdots,
\end{align*}
After we construct graph $\cT_r$, we need to further use Propositions \ref{c:intbp} and \ref{c:intbp2} to expand terms in the following form
\begin{align*}\begin{split} 
\frac{1}{d^{\fh/2}}\frac{1}{N^{\theta(\cT_r)}d^{|E_r|}}\sum^*_{\bfi_r}\bE\left[A_{\cT_r}R_{\bfi_r} U_{\bfi_r}\right], 
\end{split}\end{align*}
where $R_{\bfi_r}$ is a monomial of Green's function entries as in Definition \ref{d:evaluation}, and $U_{\bfi_r}$ is in the following form. 

\begin{definition} \label{d:Ui}
Given an index set $\bfi$, we will use $U_\bfi=U_\bfi(m,\bar m)$ to represent an expression in the following form: for fixed $b\in \bN$ and discrete derivative operators $D_{\xi_1}, D_{\xi_2},\cdots, D_{\xi_b}$ as in Definition \ref{d:defD} with $\bm V(\xi_1), \bm V(\xi_2),\cdots, \bm V(\xi_b)\subset \bfi$,
\begin{align}\label{e:defU}
\left(\prod_{a=1}^{b}(1+D_{\xi_a})^{e_a}(\sqrt{d-1}D_{\xi_a})^{1-e_a}\right)  P ^{q-1}\bar  P ^q, \quad e_1, e_2, \cdots, e_b\in\{0,1\}.
\end{align}
If $b=0$, then the above expression reduces to $P^{q-1}\bar P^q$.
\end{definition}
We remark that from the Taylor expansion \eqref{e:D-expand}, up to negligible error, $U_\bfi$ is a finite linear combination of terms in the form $\del_{\bm\xi}( P ^{q-1}\bar  P ^q)$.

\subsection{Proof of Proposition \ref{p:DSE}}\label{s:prove}

We recall the discussion from last section, that we need to estimate terms in the following form:
\begin{definition} \label{d:order_terms}
For $\fh,\fo\in \bN$, we use the symbol
$\cE_{\fh,\fo}$
to denote a finite linear combination of terms of the form for some $\fh\geq 0$,
\begin{align}\label{e:newterm}
 \frac{1}{d^{\fh/2}}\frac{1}{N^{\theta(\cT)}d^{|E|}}\sum^*_{\bf i}\bE\left[A_{\cT}R_{\bfi} U_\bfi\right],
\end{align}
where the forest $\cT=(\bfi, E)$ is as in Definition \ref{d:forest}, $R_\bfi$ is a monomial of Green's function entries as in Definition \ref{d:evaluation}, such that $\chi(R_\bfi)=\fo$ (recall that $\chi(R_\bfi)$ counts the number of off-diagonal Green's function entries), and $U_\bfi$ is as in \eqref{e:defU}. %We say \eqref{e:newterm} is of order $\fh+\fo$. 
\end{definition}

In the next two propositions, we collect some useful estimates. Their proofs are given by Sections \ref{s:pre} and \ref{s:pre2}.
\begin{proposition}\label{p:someb}
We adopt assumptions in Proposition \ref{p:DSE}, and recall the polynomial $ P $ from \eqref{e:P_split}. For $\del_{\bm\xi}=\del_{\xi_1}\del_{\xi_2},\cdots\del_{\xi_b}$ as in Definition \eqref{d:defD}, then for $b\geq 1$
\begin{align}\label{e:dmN}
\del_{\bm\xi}  P (z,m)=\del_{\bm\xi} m  \del_2 P (z, m)+\O_\prec\left(\frac{\Im[m]^2}{(N\eta)^2} \right)
=\OO_\prec \left(\frac{\Im[m]|\del_2 P(z,m)|}{N\eta}+\frac{\Im[m]^2}{(N\eta)^2} \right),
\end{align}
and
\begin{align}\label{e:onet}
\left(\frac{d^{3/2}\Lambda_o}{N}+ \frac{1}{N\eta}\right)\del_{\bm\xi}( P ^{q-1}(z, m)\bar  P ^q(z,m))\prec \Phi_q.
 \end{align}
As a consequence, if $R_\bfi, U_\bfi$ are as in Definitions \ref{d:evaluation} and \ref{d:Ui}, then if $R_{\bfi}$ contains at least one off-diagonal Green's function entry
\begin{align}\label{e:Uibound}
\left(\frac{d^{3/2}\Lambda_o}{N}+\frac{\Im[m]}{N\eta}\right)\left|R_\bfi U_\bfi\left(A+\sum_{a=1}^b \xi_a\right)\right|\lesssim \Phi_q,
\end{align}
and for $\fo+\fh$ large enough such that $\Lambda_o^{\fo+\fh}\leq 1/N$ (recall $\Lambda_o$ from \eqref{Lambda_weak}), we have
\begin{align}\label{e:Ceod}
\cE_{\fh,\fo}\lesssim \bE[\Phi_q].
\end{align}
\end{proposition}

\begin{proposition}\label{c:one-off}
We assume assumptions in Proposition \ref{p:DSE}. Given a forest $\cT=(\bfi, E)$ as in Definition \ref{d:forest}, $R_\bfi$ is a monomial of Green's function entries as in Definition \ref{d:evaluation}. Let $\del_{\bm\xi}=\del_{\xi_1}\del_{\xi_2},\cdots\del_{\xi_b}$ as in Definition \ref{d:defD} with $\bm V(\bm\xi)\subset \bfi$. If $R_\bfi$ contains at least one off-diagonal Green's function entry, i.e. $R_\bfi=G_{ij} R'_{\bfi}$, and $i,j\in \bfi$ are in different connected components of $\cT$ then
\begin{align}\label{e:AGij}
\frac{1}{N^{\theta(\cT)}d^{|E|}}\sum^*_{\bfi}\bE\left[A_{\cT} R_{\bfi}\del_{\bm\xi}( P ^{q-1}\bar  P ^q)\right]\prec \bE[\Phi_q],
\end{align}
In particular, if we take $\cT=(\bfi, E=\emptyset)$ then $\cA_{\cT}=1$ and 
\begin{align}\label{e:Gij}
\frac{1}{N^{|\bfi|}}\sum^*_{\bfi}\bE\left[R_{\bfi}\del_{\bm\xi}( P ^{q-1}\bar  P ^q)\right]\prec \bE[\Phi_q].
\end{align}
\end{proposition}

%
%
%\begin{align}\begin{split} \label{e:expand}
%\frac{1}{d^{\fh/2}}\frac{1}{N^{\theta(\cT)}d^{|E|}}\sum^*_{\bfi}\bE\left[A_{\cT}R_{\bfi} U\right], 
%\end{split}\end{align}
%where $\cT=(\bfi, E)$ is a forest as in Definition \ref{d:forest}, and $R_\bfi, U_\bfi$ are as in Definitions \ref{d:evaluation} and \ref{d:Ui}. 

The following two propositions state how to further expand the expression \eqref{e:newterm} using Propositions \ref{c:intbp} and \ref{c:intbp2}. There are two cases depending if $R_{\bfi}$ contains off-diagonal Green's function entry or not. Their proofs are given in Sections \ref{s:expand1} and \ref{s:expand2}.
%\begin{align}\label{e:newterm}
% \frac{1}{d^{\fh/2}}\frac{1}{N^{\theta(\cT)}d^{|E|}}\sum^*_{\bf i}\bE\left[A_{\cT}R_{\bfi}(A) U_\bfi(A)\right]\in \cE_{\fh,\fo,\fd}
%\end{align}

\begin{proposition}[off-diagonal discrete integration by parts]\label{c:case1exp}
If $R_{\bfi}$ in \eqref{e:newterm} contains at least one off-diagonal Green's function entries, i.e. $\fo\geq 1$, then up to small error $\OO_\prec(d^{-\fh/2}\bE[\Phi_q])$, \eqref{e:newterm} is a finite sum of terms in the form
\begin{align}\label{e:nextod1}
\frac{1}{d^{\tilde \fh /2}}\frac{1}{N^{\theta(\cT^+)}d^{|E^+|}}\sum^*_{\bfi^+}\bE\left[A_{\cT^+ } R_{\bfi^+ } U_{\bfi^+}\right]\in \cE_{\tilde \fh,\tilde \fo},\quad U_{\bfi^+}=(1+D_{\cT^+})U_\bfi,
\end{align}
where $\cT^+=(\bfi^+, E^+)$ with $\bfi^+=\bfi \bfj$, $E^+=\cup_{e\in E}\{i_e, j_e\}\cup\{i_e', j_e'\}$, $\tilde \fh\geq \fh+\fo$.
\end{proposition}

\begin{proposition}[diagonal discrete integration by parts]\label{c:case2exp}
If $R_{\bfi}$ in \eqref{e:newterm} is a monomial of diagonal Green's function entries and $m$ ($\fo=0$), and contains at least one diagonal Green's function entry, say $R_{\bfi}=G_{ii} R'_{\bfi}$, then up to small error $\OO_\prec(d^{-\fh/2}\bE[\Phi_q])$, \eqref{e:newterm} equals
\begin{align*}
\frac{1}{d^{\fh/2}}\frac{1}{N^{\theta(\cT)}d^{|E|}}\sum^*_{\bfi}\bE\left[A_{\cT} (m R'_{\bfi}) U_\bfi\right],
\end{align*}
plus a finite sum of terms in the form
\begin{align}\label{e:nextod2}
\frac{1}{d^{\tilde \fh/2}}\frac{1}{N^{\theta(\cT^+)}d^{|E^+|}}\sum^*_{\bfi^+ }\bE\left[A_{\cT^+ } R_{\bfi^+ } U_{\bfi^+}\right]\in \cE_{\tilde \fh,\tilde \fo},
\end{align}
where there are three cases
\begin{enumerate}
\item $\cT^+ =(jk\ell m\bfi, E\cup \{j,\ell\}\cup\{k,m\})$, $U_{\bfi^+}=(1+D_{jk}^{\ell m})U_{\bfi}$ or $U_{\bfi^+}=\sqrt{d-1}D_{jk}^{\ell m}U_{\bfi}$; then  either $\tilde \fh\geq \fh+1$ or $\tilde \fh=\fh$ and $\tilde \fo\geq 1$.
\item $\cT^+ =(k\ell m\bfi, E\cup \{i,\ell\}\cup\{k,m\})$, where $i\in \bfi$, $U_{\bfi^+}=(1+D_{ik}^{\ell m})U_{\bfi}$ or $U_{\bfi^+}=\sqrt{d-1}D_{ik}^{\ell m}U_{\bfi}$; 
then  either $\tilde \fh\geq \fh+1$ or $\tilde \fh=\fh$ and $\tilde \fo\geq 1$.
\item $\cT^+ =\cT$, $U_{\bfi^+}=U_{\bfi}$; then either   $\tilde \fh\geq \fh+1$.
\end{enumerate}
In all three cases $\tilde \fh+\tilde \fo\geq \fh+\fo+1$.
\end{proposition}

Starting from the graph $\cT_0=(\bfi_0={ij},E_0=\{i,j\})$ as in \eqref{e:EPP},  using Proposition \ref{c:intbp} with $F_{\bfi}=G_{ij}P^{q-1}\bar P^{q-1}$,  we get $\cT_1=(\bfi_1=ijk\ell, E_1=\{i,k\}\cup\{j,\ell\})$. Then for $i\neq j$, $|F_\bfi|\prec \Lambda_o|P|^{2q-1}$,  $d^{3/2}\cC(F,A)/N\prec \Phi_q$, and \eqref{e:EPP} equals  
\begin{align}\begin{split}\label{e:T1t}
& \phantom{{}={}}\frac{d}{(d-1)^{1/2}}\frac{1}{N^{2}} \sum_{ij}^* \bE [G_{ij}P^{q-1}\bar P^q]
 +\frac{d}{(d-1)^{1/2}}\frac{1}{N^{\theta(\cT_1)} d^{|E_1|}} \sum_{ijk\ell}^* \bE [A_{\cT_1}D_{ij}^{k\ell}(G_{ij} P^{q-1}\bar P^q)],
\end{split}\end{align}
with an error bounded by $\OO_\prec(\bE[\Phi_q])$. For the first term in \eqref{e:T1t}, we have $\sum_{ij} G_{ij}=0$ and $\sum^*_{ij}G_{ij}\prec N$. Thus it is bounded by $\bE[\Phi_q]$. For the second term in \eqref{e:T1t}, by the discrete chain rule, we can rewrite it as
\begin{align}\begin{split}\label{e:T1t2}
& 
 \phantom{{}={}}\frac{d}{d-1}\frac{1}{N^{\theta(\cT_1)} d^{|E_1|}} \sum_{ijk\ell}^* \bE [A_{\cT_1}G_{ij} \sqrt{d-1}D_{ij}^{k\ell}(P^{q-1}\bar P^q)]\\
 &+\frac{d}{(d-1)^{1/2}}\frac{1}{N^{\theta(\cT_1)} d^{|E_1|}} \sum_{ijk\ell}^* \bE [A_{\cT_1} D_{ij}^{k\ell}(G_{ij})(1+D_{ij}^{k\ell})P^{q-1}\bar P^q)]
+\OO_\prec(\bE[\Phi_q]).
\end{split}\end{align}
For the first term in \eqref{e:T1t2}, we recall $\cT_1=(\{i,j,k,\ell\},\{\{i,k\}, \{j,\ell\}\})$. We notice that $i,j$ are in different connected components of $\cT_1$. We can Taylor expand $\sqrt{d-1}D_{ij}^{k\ell}(P^{q-1}\bar P^q)$ as in \eqref{e:D-expand}, then it is a linear combination of terms in the form $\del_{\bm\xi}(P^{q-1}\bar P^q)$.  We conclude from \eqref{e:AGij} that the first term  in \eqref{e:T1t2}  is bounded by $\bE[\Phi_q]$. For the second term in \eqref{e:T1t2}, we can similarly Taylor expand $D_{ij}^{k\ell}G_{ij}$,
\begin{align}\label{e:expandhi}
 D_{ij}^{k\ell}G_{ij} = \sum_{n=1}^{\fb-1}  \frac{\del_{\xi}^nG_{ij}}{n!(d-1)^{n/2}} +  \frac{\del_{\xi}^{\fb}G_{ij}(A+\theta\xi)}{\fb!(d-1)^{\fb/2}} ,\quad \xi=\xi_{ij}^{k\ell}.
\end{align}
We notice that $\del_{\xi}^nG_{ij}$ is a linear combination of monomials of Green's function entries, and is stochastically bounded, i.e. $\OO_\prec(1)$. We can take $\fb$ large enough such that the remainder term is bounded by $\OO_\prec(1/N)$. In this way, up to error $\bE[\Phi_q]$,  the second term in \eqref{e:T1t2} can be rewritten as a linear combination of 
\begin{align}\label{e:rec1a}
 \frac{1}{d^{\fh/2}}\frac{1}{N^{\theta(\cT_1)}d^{|E_1|}}\sum^*_{\bfi_1}\bE\left[A_{\cT_1}R_{\bfi_1}U_{\bfi_1}\right],\quad U_{{\bfi}_1}=P^{q-1}\bar P^q \text{ or } U_{\bfi_1}=\sqrt{d-1}D_{ij}^{k\ell}P^{q-1}\bar P^q,
\end{align}
with $\fh\geq 0$. The terms corresponding to $\fh=0$ are from taking $n=1$ in \eqref{e:expandhi}, and $U_{\bfi_1}=P^{q-1}\bar P^q$:
\begin{align}\begin{split}\label{e:exp}
&\phantom{{}={}}\frac{d}{d-1}\frac{1}{N^{\theta(\cT_1)} d^{|E_1|}} \sum_{ijk\ell}^* \bE [A_{ik}A_{j\ell} \del_{ij}^{k\ell}(G_{ij})P^{q-1}\bar P^q)]\\
&=-\frac{d}{d-1}\frac{1}{N^{\theta(\cT_1)} d^{|E_1|}} \sum_{ijk\ell}^* \bE [A_{ik}A_{j\ell} (G_{ii}G_{jj}+G_{ik}G_{j\ell}+\cI)P^{q-1}\bar P^q)]\\
&=-\frac{d}{d-1}\left(\bE[m^2P^{q-1}\bar P^q]+\frac{1}{N^{\theta(\cT_1)} d^{|E_1|}} \sum_{ijk\ell}^* \bE [A_{ik}A_{j\ell}G_{ik}G_{j\ell}P^{q-1}\bar P^q)]\right)+\OO_\prec(\bE[\Phi_q]),
\end{split}\end{align}
where
\begin{align}\label{e:cI}
\cI=G^2_{ij}+G_{il}G_{kj}+G_{ii}G_{kj}+G_{ik}G_{ij}+G_{ij}G_{lj}+G_{il}G_{jj}.
\end{align}
To get the last line in \eqref{e:exp}, we use the following property, i.e.,  since $\cI$ contains Green's function entries with indices in different connected components of $\cT_1$,  the corresponding terms are bounded by $\bE[\Phi_q]$ thanks to Proposition \ref{c:one-off}. 

For terms in \eqref{e:rec1a} with $\fh\geq 1$ and the second term on the righthand side of \eqref{e:exp}, we can further use Propositions \ref{c:case1exp} and \ref{c:case2exp} to expand them. In this way we will obtain  a sequence of forests 
\begin{align}\label{e:constructT}
\cT_0=(\bfi_0, E_0),\quad  \cT_1=(\bfi_1, E_1), \quad \cT_2=(\bfi_2, E_2), \quad \cT_3=(\bfi_3, E_3),\quad \cdots.
\end{align}
More precisely, if we have a term corresponding to the graph $\cT_t=(\bfi_t, E_t)$
\begin{align}\begin{split} \label{e:Trterm}
\frac{1}{d^{\fh/2}}\frac{1}{N^{\theta(\cT_t)}d^{|E_t|}}\sum^*_{\bfi_t}\bE\left[A_{\cT_t}R_{\bfi_t} U_{\bfi_t}\right]\in \cE_{\fh, \fo}.
\end{split}\end{align}
There are two cases:
if $\fo\geq 1$, we can use Proposition \ref{c:case1exp} to expand \eqref{e:Trterm} as a sum of terms in $\cE_{\tilde \fh, \tilde\fo}$ with $\tilde \fh\geq \fh+\fo$ corresponding to new graphs $\cT_{t+1}=(\bfi_{t+1}, E_{t+1})$;  if $\fo=0$, and $R_{\bfi_t}$ contains diagonal Green's function entries, we use  Proposition \ref{c:case2exp} to replace the diagonal Green's function entries in $R_{\bfi_t}$ to $m$,
\begin{align}\label{e:Trterm2}
\frac{1}{d^{\fh/2}}\frac{1}{N^{\theta(\cT_t)}d^{|E_t|}}\sum^*_{\bfi_t}\bE\left[A_{\cT_t}m^{\fd} U_{\bfi_t}\right]\in \cE_{\fh, 0},\quad \fd=\deg(R_{\bfi_t}),
\end{align}
and the remaining terms are in $\cE_{\tilde \fh, \tilde\fo}$ with $\tilde \fh\geq \fh$ and $\tilde \fh+\tilde \fo\geq \fh+1$ corresponding to new graphs $\cT_{t+1}=(\bfi_{t+1}, E_{t+1})$. It follows by induction that in \eqref{e:Trterm} $\fh+\fo\geq \floor{t/2}$. Thus for $t$ large enough, \eqref{e:Ceod} implies that \eqref{e:Trterm} is bounded by $\bE[\Phi_q]$.

The next Proposition states that  we can rewrite \eqref{e:Trterm2} back as an expectation of $m^\fd  P^q \bar P^{q-1}$. Its proof will be given in Section \ref{s:rewrite}.

\begin{proposition}\label{p:rewrite}
If $\cT_t$ is constructed as in \eqref{e:constructT} with $\fh\geq 1$, there exists some constant $C\asymp 1$ depending on the sequence of graphs $\cT_1,\cT_2,\cdots,\cT_t$, the following holds
\begin{align}\label{e:back}
\frac{1}{d^{\fh/2}}\frac{1}{N^{\theta(\cT_t )}d^{|E_t |}}\sum^*_{\bfi_t }\bE[A_{\cT_t } m^\fd U_{\bfi_t}]
=\frac{C}{d^{\fh/2}} \bE[ m^\fd  P^q \bar P^{q-1}]+\OO_\prec(\bE[\Phi_q]).
\end{align}
\end{proposition}

%{\cob
%We terminate this construction, when either $t$ is sufficiently large such that 
%\begin{align}\begin{split} 
%\frac{1}{d^{\fe/2}}\frac{1}{N^{\theta(\cT_t)}d^{|E_t|}}\sum^*_{\bfi_t}\bE\left[A_{\cT_t}R_{\bfi_t} U_{\bfi_t}\right]\prec \bE[\Phi_q].
%\end{split}\end{align}
%Or $R_{\bfi_t}=m^{\fd}$, which does not contain any Green's function entries. 
%
%We notice that $\fe=0$ in \eqref{e:rec1a} only if $R_{\bfi_1}$ is from 
%\begin{align}
%\del_{ij}^{k\ell} G_{ij}=-G_{ii}G_{jj}-G_{ij}G_{ij}-G_{ik}G_{\ell j}-G_{i\ell}G_{kj}-G_{ii}G_{kj}-G_{ik}G_{ij}-G_{ij}G_{\ell j}-G_{i\ell}G_{jj}.
%\end{align}
%From Proposition \ref{c:one-off}, in this case, \eqref{e:rec1a} is bounded by $\bE[\Phi_q]$ unless $R_{\bfi_1}=-G_{ii}G_{jj}$ or $G_{ik}G_{\ell j}$. If $R_{\bfi_1}=G_{ik}G_{\ell j}$, we can further use Proposition \ref{c:case1exp} to expand \eqref{e:rec1a}, and the new terms are in $\cE_{\fe, \fo}$ with $\fe\geq 4$. If $R_{\bfi_1}=G_{ii}G_{jj}$, we can use Proposition \eqref{c:case2exp} to replace the diagonal Green's function entries by $m^2$, 
%\begin{align}
%\frac{1}{N^{\theta(\cT_1)}d^{|E_1|}}\sum^*_{\bf i}\bE\left[A_{\cT_1}m^2U_{\bfi_1}\right],
%\end{align}
%and the remaining terms are of order at least $1$. 
%
%
%}

 All these terms as in \eqref{e:Trterm2} and the first term in \eqref{e:exp}  lead to a polynomial $\wt Q(m)$, such that 
\begin{align}\begin{split}\label{e:EPP2}
\bE[(1+zm)P^{q-1}\bar P^q] 
&=  -\bE [\wt Q(m)P^{q-1}\bar P^q] +\OO_\prec( \Phi_q).
\end{split}\end{align}
The first term on the righthand side of \eqref{e:exp} gives the term $du^2/(d-1)$ in $\wt Q(u)$. All other terms as in \eqref{e:Trterm2} have $\fh\geq 1$, they lead to terms in $\wt Q(u)$ with coefficients bounded by $\OO(1/\sqrt d)$. Therefore $\wt Q(u)=u^2+Q(u)$ where $Q(u)$ has coefficients bounded by $\OO(1/\sqrt d)$ as in \eqref{e:Q_a}.

%
%Next we show $\wt Q(u)=u^2+Q(u)$, where $Q(u)$ has coefficients bounded by $\OO(1/\sqrt d)$ as in \eqref{e:Q_a}. If $\fh\geq 1$ in \eqref{e:Trterm2}, it gives a term in $Q(u)$ with coefficient bounded by $\OO(1/\sqrt d)$. 
%If $\fh=0$ in \eqref{e:Trterm2}, then $\fh+\fo=0 \geq \floor{t/2}$, and $t=1$. If $\fh=\fo=0$ in \eqref{e:rec1a}, from the discussion above \eqref{e:rec1a}, we must have $R_{\bfi_1}=-G_{ii}G_{jj}$. After replacing $R_{\bfi_1}$ by $m^2$, and using Proposition \ref{p:rewrite}, we get the term $\bE[dm^2 P^{q-1}\bar P^q/(d-1)]$. The claim $\wt Q(u)=u^2+Q(u)$ follows.
%

%\begin{proposition}
%We 
%\begin{align}\label{e:dmN}
%    &\del^{\bm \beta}  P (z,m(z))=\del^{\bm \beta} m (z) \del_2 P (z, m(z))+\O_\prec\left(\frac{\Im[m (z)]^2}{(N\eta)^2} \right).
%\end{align}
%Moreover, we have the following estimate
%\begin{align}\label{e:onet}
%    \frac{\Im[m ]}{N\eta}\del^{\bm\beta}( P ^{q-1}(z, m(z))\bar  P ^q(z,m(z)))\prec \Phi_q.
%\end{align}
%\end{proposition}

\subsection{Proof of Propositions \ref{p:someb}}\label{s:pre}

\begin{proof}[Proof of Proposition \ref{p:someb}]
For any set $J\subset \qq{1,b}$, we denote $\del_{J}=\prod_{j\in J}\del_{\xi_j}$.
The derivative $\del_{\bm\xi}  P (z,m)$ is a linear combination of terms in the form
\begin{align}\label{e:decom}
\del_{{J_1}}m \del_{{J_2}}m\cdots \del_{J_r} m \del_2^rP,\quad J_1\cup J_2\cup\cdots \cup J_r=\qq{1,b}, \quad |J_1|,|J_2|,\cdots, |J_r|\geq 1.
\end{align}
From \eqref{e:impp}, we have $|\del_{J}m|\prec \Im[m]/N\eta$ for $|J|\geq 1$. Thus \eqref{e:decom} is bounded by $(\Im[m]/N\eta)^2$, except for $r=1$. In this case, \eqref{e:decom} simplifies to $\del_{\bm\xi}m \del_2 P\prec \Im[m]|\del_2 P|/N\eta$. This gives \eqref{e:dmN}. 

For \eqref{e:onet}, $\del_{\bm\xi} (P^{q-1}\bar P^q)$ can be bounded by a finite sum of terms in the form
\begin{align*}
|\del_{J_1}P| |\del_{{J_2}}P|\cdots |\del_{J_r}P| |P|^{2q-r-1}\prec \left(\frac{\Im[m]|\del_2 P(z,m)|}{N\eta}+\frac{\Im[m]^2}{(N\eta)^2} \right)^r |P|^{2q-r-1},
\end{align*} 
where $J_1\cup J_2\cup\cdots \cup J_r=\qq{1,b}$ and $|J_1|,|J_2|,\cdots, |J_r|\geq 1$. The claim follows from the definition of $\Phi_q$ in \eqref{e:defPhir}

For \eqref{e:Uibound}, let $\xi=\sum_{a=1}^b \xi_a$, from Remark \ref{r:basicestimates}, we have that $R_\bfi(A+\xi)\prec 1$. For $U_\bfi(A+\xi)$, we estimate it by the sum of $U_\bfi(A)$ and $D_{\xi} U_\bfi=U_\bfi(A+\xi)-U_\bfi(A)$. From the definition \eqref{e:defU}, and the Taylor expansion \eqref{e:D-defxi}, $U_\bfi$ and $D_{\xi} U_\bfi$ are both  finite sums of terms in the form
\begin{align}\label{e:dP}
\del_{\xi_{i_1}}\del_{\xi_{i_2}}\cdots \del_{\xi_{i_r}}(P^{q-1}\bar P^{q}),\quad 1\leq i_1,i_2,\cdots, i_r\leq b. 
\end{align}
If $r=0$, then \eqref{e:dP} is bounded by $|P|^{2q-1}$; if $r\geq 1$, we have the bound \eqref{e:onet}. The claim \eqref{e:Uibound} follows.

Finally for \eqref{e:Ceod}, we recall from \eqref{G_estimates} that for off-diagonal Green's function entry $|G_{ij}|\prec \Lambda_o$. For any term in $\cE_{\fh,\fo}$, we can bound $|R_\bfi|\prec \Lambda_o^\fo$ and use the trivial bound $|U_\bfi|\prec N\Phi_q$ from \eqref{e:onet}, then
\begin{align*}
 \frac{1}{d^{\fh/2}}\frac{1}{N^{\theta(\cT)}d^{|E|}}\sum^*_{\bf i}\bE\left[A_{\cT}R_{\bfi} U_\bfi \right]
 \prec  \frac{\Lambda_o^{\fo} }{d^{\fh/2}}\frac{1}{N^{\theta(\cT)}d^{|E|}}\sum^*_{\bf i}\bE\left[A_{\cT} N\Phi_q \right]\leq \frac{\Lambda_o^{\fo} }{d^{\fh/2}}N\bE[\Phi_q]\leq \bE[\Phi_q],
\end{align*}
provided that $\Lambda_o^{\fo}/d^{\fh/2}\leq \Lambda_o^{\fo+\fh}\leq 1/N$,  where we used \eqref{e:sumA} for the summation over $A_\cT$. 
\end{proof}

\begin{proposition}\label{p:DFU}
Let $\del_{\bm\xi}=\del_{\xi_1}\del_{\xi_2},\cdots\del_{\xi_b}$ as in Definition \ref{d:defD} with $\bm V(\bm\xi)\subset \bfi$.
Let $R_\bfi$ be as in Definition \ref{d:evaluation}, assume it contains an off-diagonal Green's function entry $G_{ii'}$ such that $i$ and $i'$ are in different connected components of $\cT$. 
$\cT^+$ is constructed from $\cT$ as in Proposition \ref{c:case1exp}. Then up to error $\Phi_q$, $D_{\cT^+} (R_{\bfi}\del_{\bm\xi}(P^{q-1}\bar P^q))$ is a finite linear combination of $d^{-\fh/2}R_{\bfi^+}\del_{\bm\xi'}(P^{q-1}\bar P^q)$ with $\fh\geq 1$, where $R_{\bfi^+}$
contains at least one off-diagonal Green's function entry $G_{jj'}$ such that $j,j'$ are in different connected components of $\cT^+$; $\bm V(\bm\xi')\subset \bfi^+$ and $|\bm\xi'|\geq b$.  
\end{proposition}
\begin{proof}
Thanks to the discrete chain rule \eqref{e:D-defxi}, we can rewrite $D_{\cT^+} (R_{\bfi}\del_{\bm\xi}(P^{q-1}\bar P^q))$ in the form
\begin{align}\label{e:DFU2}
R_{\bfi}  D_{\cT^+}\del_{\bm\xi}(P^{q-1}\bar P^q)
+ (D_{\cT^+}R_{\bfi}) (1+D_{\cT^+})\del_{\bm\xi}(P^{q-1}\bar P^q).
\end{align}
We recall the construction of $\cT^+$ from $\cT$. Say $\cT=({\bfi},E)$, for each edge $e=\{i_e, i_e'\}\in E$, we associate another edge $\{j_e,j'_e\}\in \qq{N}$. Then $\cT^+=(\bfi\bfj,E^+)$ consists of $2|E|$ disjoint edges, with $E^+=\cup_{e\in E}\{i_e, j_e\}\cup\{i_e', j_e'\}$. For the first term in \eqref{e:DFU2}, using Remark \ref{r:Tstructure},  if $i,i'\in \bfi$ are distinct, they are in different connected component of $\cT^+$.

By the Taylor expansion \eqref{e:D-expand}, up to error $\bE[\Phi_q]$, $ D_{\cT^+}\del_{\bm\xi}(P^{q-1}\bar P^q)$ is a finite linear combination of $d^{-\fh/2}\del_{\bm\xi'}(P^{q-1}\bar P^q)$ with $\fh\geq 1$, $\bm V(\bm\xi')\subset \bfi^+$ and $|\bm\xi'|\geq b$. 

Next we study the second term in \eqref{e:DFU2}, and show $D_{\cT^+}R_{\bfi}$ contains at least one off-diagonal Green's function entry $G_{jj'}$ such that $j,j'$ are in different connected components of $\cT^+$. 
We need to compute 
\begin{align}\label{e:difG}
D_{\cT^{+}}(G_{ii'}R_{\bfi})
=\sum_{n=1}^{\fb-1} \frac{1}{n!(d-1)^{n/2}}\left(\sum_{e\in \cT}\del_{i_e j_e}^{i_e' j_e'}\right)^n (G_{ii'}R_{\bfi})+\OO_\prec \left(d^{-\fb/2}\right).
\end{align}
If the derivative does not hit $G_{ii'}$, then $R_{\bfi^+}$ contains $G_{ii'}$ where $i,i'$ are in different connected components of $\cT^+$. 
The derivative of $G_{ii'}$ in \eqref{e:difG} is a sum of terms in the form
\begin{align*}
G_{i x_1}G_{x_2 x_3} \cdots G_{x_{2a-2}x_{2a-1}}G_{x_{2a}i'},
\end{align*}
where for $1\leq b\leq a$, $\{x_{2b-1}, x_{2b}\}\in \{i_{e_b}, j_{e_b}, i_{e_b}', j_{e_b}'\}$ for some ${e_b}\in E$. 
%Since $j,j'\in \bfi_t$ are in different connected components of $\cT_t$, say $j\in \{i_{e_0}, i'_{e_0}\}$ and $j'\in \{i_{e_{a+1}}, i'_{e_{a+1}}\}$. 
%Then $e_0$ and $e_{a+1}$ are in different connected component of $\cT_t$. 
For the sequence of vertices and edges, $i, e_1, \cdots, e_{a}, i'$, since $i,i'\in \bfi$ are in different connected components of $\cT$, there are two cases:
i) $i$ and $e_1$ are in different connected component of $\cT$. From the construction of $\cT^+$, $i$ and $\{i_{e_{1}}, j_{e_{1}}, i_{e_{1}}', j_{e_{1}}'\}$ are disconnected in $\cT^+$. Thus $i$ and $x_{1}$ are in different connected components of $\cT^+$. If $x_{2a}$ and $i'$ are in different connected components of $\cT$, by the same argument, $x_{2a}$ and $i'$ are in different connected components of $\cT^{+}$. 
ii) There is some $1\leq b\leq a-1$ such that $e_b$ and $e_{b+1}$ are in different connected components of $\cT$. From the construction of $\cT^+$, $\{i_{e_b}, j_{e_b}, i_{e_b}', j_{e_b}'\}$ and $\{i_{e_{b+1}}, j_{e_{b+1}}, i_{e_{b+1}}', j_{e_{b+1}}'\}$ are disconnected in $\cT^+$. Thus $x_{2b}$ and $x_{2b+1}$ are in different connected components of $\cT^+$.  This finishes the proof of Proposition \ref{p:DFU}. 
\end{proof}

\subsection{Proof of Proposition \ref{c:one-off}}\label{s:pre2}
\begin{proposition}\label{p:one-off}
Proposition \ref{c:one-off} holds if $b\geq 1$. 
\end{proposition}

\begin{proof}[Proof of Proposition \ref{p:one-off} with $E= \emptyset$]
If $R_\bfi$ contains two off-diagonal Green's function entries $G_{ij}, G_{i'j'}$, then by the AM-GM inequality, and $|G_{xy}|\prec 1$
\begin{align}\begin{split}\label{e:two-off}
 &\phantom{{}={}} \left| \frac{1}{N^{|\bfi|}}\sum_{\bfi}^*\bE\left[ R_\bfi \del_{\bm\xi}(P^{q-1}\bar P^{q}) \right]\right| \prec 
   \frac{1}{N^{|\bfi|}}\sum_{\bfi}^*\bE\left[ (|G_{ij}|^2+|G_{i'j'}|^2)\left|  \del_{\bm\xi}(P^{q-1}\bar P^{q})\right| \right]\\
  &\prec 
   \frac{1}{N^{|\bfi|}}\sum_{\bfi}^*\bE\left[ (|G_{ij}|^2+|G_{i'j'}|^2) (N\eta)\Phi_q \right]\prec \bE[\Phi_q],
\end{split}\end{align} 
where we used \eqref{e:onet} for the second inequality, and Ward identity \eqref{e:WdI} for the last inequality. 

In the rest, we assume that $R_\bfi$ contains exactly one off-diagonal term. We denote $\bfi=\{i,j, m_1, m_2,\cdots, m_r\}$, $|\bfi|=r+2$, and $R_\bfi=G_{ij}R'_{\bfi}$.
We recall the derivatives of $P$ from \eqref{e:dmN},
\begin{align}\begin{split}\label{e:DPP}
    \del_{\bm\xi'} P=\del_{\bm\xi'} m (z)\del_2 P+\O_\prec\left(\frac{\Im[m ]^2}{(N\eta)^2} \right).
\end{split} \end{align}
We can rewrite \eqref{e:Gij} as a sum of terms in the form
\begin{align}\label{e:PII}
    \frac{1}{N^{r+2}}\sum_{ij\bfm} \bE[G_{ij}R'_\bfi I_1 I_2\cdots I_t P^{r-1-t_1}\bar P^{r-t_2}],\quad t_1+t_2=t,
\end{align}
where $I_s$ for $1\leq s\leq t$ corresponds to terms in \eqref{e:DPP} defined in the following: 
\begin{enumerate}
   \item\label{e:1} From \eqref{e:dermN}, the first term in \eqref{e:DPP} $\del_{\bm\xi'}m (z) \del_2 P$ is a sum of terms in the form:
   \begin{align}\label{e:x11}
    \frac{\Im[m ]\del_2 P}{N\eta}\sum_{\bfk_s}G_{ij}^{a_s} X_{\bfk_s \bfm i} Y_{\bfk_s \bfm j}.
\end{align}
   where $a_s\geq 0$ and $\bfk_s$ is some index set, and $\sum_{\bfk_s}|X_{\bfk_s \bfm i}^2|,\sum_{\bfk_s} |Y_{\bfk_s\bfm j}^2|=\O_\prec(1)$.  We take $I_s$ to be \eqref{e:x11} or its complex conjugate. In this case $|I_s|\prec \Im[m]|\del_2 P|/N\eta$.
   \item \label{e:3} $I_s$ is bounded by $\O_\prec\left(\frac{\Im[m ]^2}{(N\eta)^2} \right)$ corresponding to the last term in \eqref{e:DPP}.
\end{enumerate}

From the construction, we have $|I_s|\prec \Im[m ]/N\eta$ for all $1\leq s\leq t$. 
If there is one $I_s$ corresponding to Item \ref{e:1} with $a_s\geq 1$, then
\begin{align*}
    |I_s|\prec |G_{ij}|\frac{\Im[m ] |\del_2 P|}{N\eta}.
\end{align*}
And noticing $|\del_2 P|\prec 1$, we have 
\begin{align*}
 \frac{1}{N^{r+2}}\sum^*_{ij\bfm} \bE[G_{ij}R'_\bfi I_1 I_2\cdots I_t P^{q-t_1-1}\bar P^{q-t_2}]
    &\prec \frac{1}{N^{2}}\sum^*_{ij} \bE\left[|G_{ij}|^{2}\left(\frac{\Im[m ]}{N\eta}\right)^{t}|P^{2q-t-1}|\right]\\
    &\leq  \bE\left[\left(\frac{\Im[m ]}{N\eta}\right)^{t+1}|P^{2q-t-1}|\right]\prec \bE[\Phi_q].
\end{align*}
where we used Ward identity \eqref{e:WdI}.

If there is one $I_s=\O_\prec\left(\frac{\Im[m ]^2}{(N\eta)^2} \right)$ as in Item \ref{e:3}, we have
\begin{align*}
   &\phantom{{}={}}\frac{1}{N^{r+2}}\sum^*_{ij\bfm} \bE[G_{ij}R'_\bfi I_1 I_2\cdots I_t P^{q-t_1-1}\bar P^{q-t_2}]\\
   &\prec \bE\left[\frac{\Im[m ]^2}{(N\eta)^2}\left(\frac{\Im[m ]}{N\eta}\right)^{t-1}|P^{2q-t-1}|\right]\prec \bE[\Phi_q].
\end{align*}

In the rest, we can assume that each $I_s$ corresponds to Item \ref{e:1} with $a_s=0$. Then we have
\begin{align}\begin{split}\label{e:GijFI}
    &\phantom{{}={}}\frac{1}{N^{r+2}}\sum^*_{ij\bfm} \bE[G_{ij}R'_\bfi I_1 I_2\cdots I_t P^{q-t_1-1}\bar P^{q-t_2}]\\
    &=    \frac{1}{N^{r+2}}\sum_{\bfm}^*\bE\left[\sum^*_{ij \not\in \bfm}G_{ij}\sum_{
    \bfk}X_{\bfk \bfm i}Y_{\bfk \bfm j}\left(\frac{\Im[m ]\del_2 P}{N\eta}\right)^{t}P^{q-t_1-1}\bar P^{q-t_2}\right].
\end{split}\end{align}
where $\sum^*_{ij \not\in \bfm}$ is a sum over indices $i\neq j\in \qq{N}\setminus \bfm$, $\bfk=\cup_s \bfk_s$ and
\begin{align*}
    X_{\bfk \bfm i}=\prod_{s} X_{\bfk_s \bfm i},\quad  Y_{\bfk \bfm j}=\prod_{s} Y_{\bfk_s \bfm j}.
\end{align*}
They are bounded 
\begin{align*}
    \sum_{\bfk}|X^2_{\bfk \bfm i}|=\sum_{\bfk}\prod_{s} |X^2_{\bfk_s \bfm j}|
    =\sum_s \sum_{\bfk_s} |X^2_{\bfk_s \bfm j}|\prec  1, \quad \sum_{\bfk}|Y^2_{\bfk\bfm j}|\prec 1.
\end{align*}
We can rewrite the sum $\sum^*_{ij \not\in \bfm}$ as the difference of the sum over $i,j\not\in \bfm$ and $i=j\not\in \bfm$.
\begin{align}\begin{split}\label{e:L2est0}
    \frac{1}{N^2}\sum^*_{ij \not\in \bfm}\sum_{
    \bfk}G_{i j}X_{\bfk \bfm i}Y_{\bfk \bfm j}
     =
     \frac{1}{N^2}\sum_{ij \not\in \bfm}\sum_{
    \bfk }G_{i j}X_{\bfk \bfm i}Y_{\bfk \bfm j}
    -
     \frac{1}{N^2}\sum_{i \not\in \bfm}\sum_{
    \bfk }G_{ii}X_{\bfk \bfm i}Y_{\bfk \bfm i}
   \end{split}\end{align}
For the first term on the righthand side of \eqref{e:L2est0}, we use the norm of $G$ is bounded by $1/\eta$, and Cauchy-Schwartz inequality
\begin{align}\begin{split}\label{e:L2est}
    \left|\frac{1}{N^2}\sum_{ij \not\in \bfm}\sum_{
    \bfk}G_{i j}X_{\bfk \bfm i}Y_{\bfk \bfm j}\right|
    &\leq  \frac{1}{N^2\eta}\sum_\bfk \sqrt{\sum_{i}|X_{\bfk \bfm i}|^2 \sum_{i}|Y_{\bfk \bfm i}|^2}\\
    &\leq \frac{1}{N^2\eta}\sqrt{\sum_\bfk \sum_{i}|X_{\bfk \bfm i}|^2 \sum_\bfk \sum_{i}|Y_{\bfk \bfm i}|^2}\prec \frac{1}{N\eta}.
\end{split}\end{align}
For the second term on the righthand side of \eqref{e:L2est0} we have 
\begin{align}\begin{split}\label{e:L2est2}
    \left|\frac{1}{N^2}\sum_{i \not\in \bfm}\sum_{
    \bfk}G_{i i}X_{\bfk \bfm i}Y_{\bfk \bfm j}\right|
    &\leq  \frac{1}{N^2}\sum_\bfk \sqrt{\sum_{i}|X_{\bfk \bfm i}|^2 \sum_{i}|Y_{\bfk \bfm i}|^2}\prec \frac{1}{N}.
    %&\leq \frac{1}{N^2}\sqrt{\sum_\bfk \sum_{i}|X_{\bfk \bfm i}|^2 \sum_\bfk \sum_{i}|Y_{\bfk \bfm i}|^2}\prec \frac{1}{N}.
\end{split}\end{align}

By plugging \eqref{e:L2est} and \eqref{e:L2est2} into \eqref{e:GijFI}, and recall $|\bfm|=r$, we get
\begin{align*}
     \frac{1}{N^{r+2}}\sum^*_{ij\bfm} \bE[G_{ij}R'_\bfi I_1 I_2\cdots I_t P^{q-t_1-1}\bar P^{q-t_2}]\prec \bE\left[\frac{1}{N\eta}\left(\frac{\Im[m ] \del_2 P}{N\eta}\right)^t |P|^{2q-t-1}\right]\prec \bE[\Phi_q].
\end{align*}
This finishes the proof of \eqref{e:Gij} for $b\geq 1$ and $E=\emptyset$. 
%This finishes the proof of the first case that $|\bm\beta|\geq 1$. 
\end{proof}

\begin{proof}[Proof of Proposition \ref{p:one-off} with $E\neq \emptyset$]
To prove \eqref{e:AGij}, we need to introduce some new notation. 
For $\fh\in \bN$, we use the symbol
$\wt\cE_{\fh}$
to denote a finite linear combination of terms of the form 
\begin{align}\label{e:tcE}
 \frac{1}{d^{\fh/2}}\frac{1}{N^{\theta(\cT)}d^{|E|}}\sum^*_{\bf i}\bE\left[A_{\cT}R_{\bfi} \del_{\bm\xi}(P^{q-1}\bar P^q)\right].
\end{align}
where the forest $\cT=(\bfi, E)$ is as in Definition \ref{d:forest}, $R_\bfi$ is a monomial of Green's function entries as in Definition \ref{d:evaluation}, and $\del_{\bm\xi}=\del_{\xi_1}\del_{\xi_2}\cdots \del_{\xi_b}$ with $b\geq 1$ and $ {\bf V}(\bm \xi)\subset \bfi$. By our assumption, $R_\bfi$ contains an off-diagonal Green's function entry $G_{ij}$ such that $i,j$ are in different connected components of $\cT$.

Next we show that any term in $\wt\cE_{\fh}$ (as in \eqref{e:tcE}) with $E\neq \emptyset$ can be rewritten as terms in $\wt\cE_{\tilde \fh}$ with $\tilde \fh\geq \fh+1$, up to a negligible error $\OO_\prec(d^{-\fh/2}\bE[\Phi_q])$:
\begin{align}\label{e:AGTc2}
 \frac{1}{d^{\fh/2}}\frac{1}{N^{\theta(\cT)}d^{|E|}}\sum^*_{\bf i}\bE\left[A_{\cT}R_{\bfi} \del_{\bm\xi}(P^{q-1}\bar P^q)\right]\in \wt\cE_{\tilde \fh}+\OO_\prec(d^{-\fh/2}\bE[\Phi_q]), \quad \tilde \fh\geq \fh+1.
\end{align}
The same as in \eqref{e:Ceod}, if $\fh$ is large enough, $|\wt\cE_{\fh}|\leq \bE[\Phi_q]$. 
Proposition \ref{p:one-off} follows from repeatedly using of \eqref{e:AGTc2}.

We prove the statement \eqref{e:AGTc2} for the case $\fh=0$, the general case follows from 
multiplying $d^{-\fh/2}$ on both sides.
For any term in $\wt\cE_{\fh}$ in the form \eqref{e:tcE}, we can use Proposition \ref{c:intbp} to expand, 
\begin{align}\begin{split}\label{e:AGTc3}
\phantom{{}={}}&\frac{1}{N^{\theta(\cT)}d^{|E|}}\sum^*_{\bf i}\bE\left[A_{\cT} R_{\bfi} \del_{\bm\xi}(P^{q-1}\bar P^q)\right]
= \frac{1}{N^{|\bfi|}}\sum^*_{\bf i}\bE\left[R_{\bfi} \del_{\bm\xi}(P^{q-1}\bar P^q)\right]\\
+&
\frac{1}{N^{\theta(\cT^+)}d^{|E^+|}}\sum^*_{\bfi \bfj}\bE\left[A_{\cT^+} D_{\cT^+}(R_\bfi \del_{\bm\xi}(P^{q-1}\bar P^q)) \right]+\OO_\prec\left(
\bE[\Phi_q]\right),
%=\cE_{\fh+1}+\OO_\prec(\bE[\Phi_q]).
\end{split}\end{align}
where $\cT^+=(\bfi^+=\bfi\bfj, E^+)$ with $E^+=\cup_{e\in E}\{i_e, j_e\}\cup\{i_e', j_e'\}$.
%and 
%\begin{align}
%D_{\cT^+}F_\bfi:=\sqrt{d-1}\left(F_{\bf i}\left(A+\sum_{e\in E}\xi_{i_ei'_e}^{j_e j'_e}\right)-F_{\bfi}(A)\right).
%\end{align}

For the first term on the righthand side of \eqref{e:AGTc3},  it is bounded by $\OO_\prec(\bE[\Phi_q])$, thanks to Proposition \eqref{p:one-off} with $E=\emptyset$. For the second term on the righthand side of \eqref{e:AGTc3}, thanks to Proposition \ref{p:DFU}, $D_{\cT^+}(R_\bfi \del_{\bm\xi}(P^{q-1}\bar P^q))$ is a linear combination of terms in the form $d^{-\tilde \fh/2}R_{\bfi^+}\del_{\bm\xi'}(P^{q-1}\bar P^q)$ with $\tilde \fh\geq 1$, where $R_{\bfi^+}$ contains an off-diagonal term $G_{k\ell}$ such that $k,\ell$ are in different connected components of $\cT^+$. Moreover, $|\bm\xi'|\geq b\geq 1$. Thus we conclude that the second term on the righthand side of \eqref{e:AGTc3} is in $\wt\cE_{\tilde \fh}$ with $\tilde \fh\geq 1$. This finishes the proof of \eqref{e:AGTc2}.
\end{proof}

\begin{proposition}\label{p:one-off2}
Proposition \ref{c:one-off} holds if $b= 0$. 
\end{proposition}

\begin{proof}[Proof of Proposition \ref{p:one-off2}]
For $\fh,\fo \in \bN$, we use the symbol
$\wt\cE_{\fh,\fo}$
to denote a finite linear combination of terms of the form 
\begin{align}\label{e:tcE2}
  \frac{1}{d^{\fh/2}}\frac{1}{N^{\theta(\cT)}d^{|E|}}\sum^*_{\bf i}\bE\left[A_{\cT}R_{\bfi}U\right],\quad U=P^{q-1}\bar P^{q},
\end{align}
where the forest $\cT=(\bfi, E)$ is as in Definition \ref{d:forest}, $R_\bfi$ is a monomial of Green's function entries as in Definition \ref{d:evaluation}, and $U_\bfi$ is as in \eqref{e:defU}. $R_\bfi$ contains an off-diagonal Green's function entry $G_{ij}$, such that $i,j$ are in different connected components of $\cT$,  and $\chi(R_\bfi)= \fo\geq 1$ (recall that $\chi(R_\bfi)$ counts the number of off-diagonal Green's function entries). 

We remark that the only  difference between $\wt\cE_{\fh,\fo}$ and $\cE_{\fh,\fo}$ from \eqref{e:newterm} is that $U=P^{q-1}\bar P^{q}$ which is independent of the indices $\bfi$.

%
%
%
%By the discrete chain-rule \eqref{}, we have
%\begin{align}
%D_{\cT^+}(R_\bfi U_\bfi)
%=D_{\cT^+}R_\bfi(1+D_{\cT^+}) U_\bfi+R_\bfi D_{\cT^+} U_\bfi
%\end{align}
%
%
%We can use Corollary \ref{c:intbp} to expand \eqref{e:AGTc} as in \eqref{e:rec}. Using the same notation, $\cT_0=(ijk\ell \bfm, \{k,\ell\})$, $\theta(\cT_0)=r+3$, $|E_0|=1$
%\begin{align}\begin{split} \label{e:rec}
%&\phantom{{}={}}\frac{1}{N^{r+3}\sqrt d}\sum^*_{\bfi}\bE\left[A_{\cT_0}(G_{\ell k}G_{ij} F)\right]
%=\sum_{s=0}^{r-1}\frac{\sqrt d}{N^{|\bfi_s|}}\sum^*_{\bfi_s}\bE\left[D_{\cT_{s}}D_{\cT_{s-1}}\cdots D_{\cT_{1}} (G_{\ell k}G_{ij} F)\right]\\
%&+
%\frac{\sqrt d}{N^{\theta(\cT_r)}d^{|E_r|}}\sum^*_{\bfi_r}\bE\left[A_{\cT_r}D_{\cT_r}D_{\cT_{r-1}}\cdots D_{\cT_1} (G_{\ell k}G_{ij} F)\right]+\OO_\prec\left(\frac{d\Phi}{N}\right).
%\end{split}\end{align}
%
%We can use Corollary \ref{c:intbp} to expand \eqref{e:AGTc} as in \eqref{e:rec}. Using the same notation, $\cT_0=(ijk\ell \bfm, \{k,\ell\})$, $\theta(\cT_0)=r+3$, $|E_0|=1$
%\begin{align}\begin{split} \label{e:rec}
%&\phantom{{}={}}\frac{1}{N^{r+3}\sqrt d}\sum^*_{\bfi}\bE\left[A_{\cT_0}(G_{\ell k}G_{ij} F)\right]
%=\sum_{s=0}^{r-1}\frac{\sqrt d}{N^{|\bfi_s|}}\sum^*_{\bfi_s}\bE\left[D_{\cT_{s}}D_{\cT_{s-1}}\cdots D_{\cT_{1}} (G_{\ell k}G_{ij} F)\right]\\
%&+
%\frac{\sqrt d}{N^{\theta(\cT_r)}d^{|E_r|}}\sum^*_{\bfi_r}\bE\left[A_{\cT_r}D_{\cT_r}D_{\cT_{r-1}}\cdots D_{\cT_1} (G_{\ell k}G_{ij} F)\right]+\OO_\prec\left(\frac{d\Phi}{N}\right).
%\end{split}\end{align}
%
%
%
Next we show that any term in $\wt\cE_{\fh,\fo}$  can be rewritten as terms in $\wt\cE_{\tilde \fh, \tilde \fo}$ with $\tilde \fh\geq \fh$ and $\tilde \fh+\tilde \fo\geq \fh+\fo+1$ up to a negligible error $\OO_\prec(d^{-\fh/2}\bE[\Phi_q])$:
\begin{align}\label{e:AGTck}
 \frac{1}{d^{\fh/2}}\frac{1}{N^{\theta(\cT)}d^{|E|}}\sum^*_{\bf i}\bE\left[A_{\cT}R_{\bfi} \del_{\bm\xi}(P^{q-1}\bar P^q)\right]\in \wt\cE_{\tilde \fh,\tilde \fo}+\OO_\prec(d^{-\fh/2}\bE[\Phi_q]). 
\end{align}
Thanks to \eqref{e:Ceod}, if $\fh,\fo$ is large enough, $|\wt\cE_{\fh,\fo}|\leq \bE[\Phi_q]$. 
Proposition \ref{p:one-off2} follows from  \eqref{e:AGTck}. 
\end{proof}

\begin{proof}[Proof of \eqref{e:AGTck} with $E=\emptyset$]
Without loss of generality we can take $\fh=0$.
If $R_\bfi$ contains two off-diagonal Green's function entries, then by the same argument as in \eqref{e:two-off}, we can show it is bounded by $\OO_\prec(\bE[(\Im[m]/N\eta)|P|^{2q-1}]=\OO_\prec(\bE[\Phi_q])$. If $R_\bfi$ contains exact one off-diagonal Green's function entry, let $R_\bfi=G_{ij}R'_\bfi$. 
By the definition of the Green's function \eqref{e:GHexp}, we have
\begin{align}
\label{e:GHijterm}&-\frac{1}{N}=-zG_{ij}+\frac{(AG)_{ij}}{\sqrt{d-1}},\quad \left(1-\frac{1}{N}\right)=-zG_{kk}+\frac{(AG)_{kk}}{\sqrt{d-1}}.
\end{align}
Multiplying the first relation in \eqref{e:GHijterm} by $G_{kk}$ and the second relation by $G_{ij}$, averaging over the index $k$, and then taking the difference, we get 
\begin{align}\label{e:Gij-m}
G_{ij}=\frac{1}{N(d-1)^{1/2}}\sum_{k\in \qq{N}\setminus \bfi}(AG)_{kk}G_{ij}-\frac{1}{\sqrt{d-1}}(AG)_{ij}m +\OO_\prec\left(\frac{1}{N}\right).
\end{align}
By plugging \eqref{e:Gij-m} into \eqref{e:Gij},
\begin{align}\begin{split}\label{e:tdiff2}
\frac{1}{N^{|\bfi|}}\sum^*_{\bfi}\bE\left[G_{ij} R'_{\bfi}U\right]
&=\frac{1}{N^{|\bfi|+1}\sqrt{d-1}}
\sum^*_{\bfi k}\bE\left[ (AG)_{kk} R_{\bfi}U\right]\\ 
&-\frac{1}{N^{|\bfi|}\sqrt{d-1}}
\sum^*_{\bfi }\bE\left[(AG)_{ij}mR'_{\bfi}U\right] +\OO_\prec(\bE[\Phi_q]).
\end{split}\end{align}

The claim \eqref{e:AGTck} follows from the following two relations
\begin{align}
\label{e:pp1}\frac{1}{N^{|\bfi|+1}\sqrt{d-1}}
\sum^*_{\bfi k}\bE\left[ (AG)_{kk} R_{\bfi}U\right]=-\frac{d}{(d-1)}\frac{1}{N^{|\bfi|}}\sum^*_{\bfi}\bE\left[m^2R_{\bfi}U\right]
+\wt \cE_{\tilde \fh, \tilde \fo}+O_\prec[\Phi_q],\\
\label{e:pp2}\frac{1}{N^{|\bfi|}\sqrt{d-1}}
\sum^*_{\bfi }\bE\left[(AG)_{ij}mR'_{\bfi}U\right]=-\frac{d}{(d-1)}\frac{1}{N^{|\bfi|}}\sum^*_{\bfi}\bE\left[m^2R_{\bfi}U\right]
+\wt \cE_{\tilde \fh, \tilde \fo}+O_\prec[\Phi_q],
\end{align}
where $\tilde \fh+\tilde \fo\geq 2$.
\end{proof}

\begin{proof}[Proof of \eqref{e:pp1}]
For \eqref{e:pp1}, we have
\begin{align*}
\frac{1}{N^{|\bfi|+1}\sqrt{d-1}}
\sum^*_{\bfi k}\bE\left[ (AG)_{kk} R_{\bfi}U\right]=\frac{1}{N^{|\bfi|+1}\sqrt{d-1}}\sum^*_{k\ell \bfi}\bE\left[A_{k\ell}G_{k\ell}R_{\bfi} U\right]+\OO_\prec(\bE[\Phi_q]).
\end{align*}
We have used $|G_{k\ell}|\prec 1$ and bounded the terms corresponding to  $\ell\in \bfi\setminus \{k\}$ by $\sqrt{d}\bE[|U|
]/N\prec \bE[\Phi_q]$. 

By using Proposition \ref{c:intbp2} with $\cT^+=(k\ell uv\bfi, E^+=\{k,u\}\cup\{\ell,v\})$ and $F_{k\ell \bfi}=G_{k\ell}R_\bfi U$. Then $|F_{k\ell \bfi}|\prec \Lambda_o |P|^{2q-1}$ and it follows from the definition \ref{e:defPhir}, $d^{3/2}\cC(F,A)/N\prec \Phi_q$, and  we have
\begin{align}\begin{split} \label{e:case1off}
&\phantom{{}={}}\frac{1}{N^{|\bfi|+1}\sqrt{d-1}}\sum^*_{k\ell \bfi}\bE\left[A_{k\ell}G_{k\ell}R_{\bfi} U\right]=\frac{d}{(d-1)^{1/2}}\frac{1}{N^{|\bfi|+2}}\sum^*_{k\ell\bfi}\bE\left[G_{k\ell}R_{\bfi} U\right]\\
&+
\frac{d}{\sqrt{d-1}}\frac{1}{N^{\theta(\cT^+)}d^{|E^+|}}\sum^*_{k\ell u v\bfi}\bE\left[A_{\cT^+}D_{k\ell}^{uv}(G_{k\ell}R_{\bfi} U)\right]+\OO_\prec\left(\bE[\Phi_q]\right).
\end{split}\end{align}
Since $\sum_{k\ell}G_{k\ell}=0$, and $\sum^*_{k\ell}G_{k\ell}\prec N$, the first term on the righthand side is bounded by $\bE[\Phi_q]$. For the second term on the righthand side of \eqref{e:case1off}, we can rewrite it as a sum
\begin{align}\begin{split}\label{e:decompcase1off2}
&\phantom{{}={}}\frac{d}{d-1}\frac{1}{N^{\theta(\cT^+)}d^{|E^+|}}\sum^*_{k\ell u v\bfi}\bE\left[A_{\cT^+}G_{k\ell}R_{\bfi} \sqrt{d-1}D_{k\ell}^{uv}U\right]\\
&+\frac{d}{\sqrt{d-1}}\frac{1}{N^{\theta(\cT^+)}d^{|E^+|}}\sum^*_{k\ell u v\bfi}\bE\left[A_{\cT^+}D_{k\ell}^{uv}(G_{k\ell}R_{\bfi})(1+D_{k\ell}^{uv}) U\right].
\end{split}\end{align}
For the first term in \eqref{e:decompcase1off2}, since $k,\ell$ are in different connected components of $\cT^+$ and the Taylor expansion of $\sqrt{d-1}D_{k\ell}^{uv}U$ is a linear combination of terms in the form $\del_{\bm\xi}(P^{q-1}\bar P^q)$, it is bounded by $\bE[\Phi_q]$ thanks to Proposition \ref{p:one-off}. For the second term in \eqref{e:decompcase1off2}, we have
\begin{align}\label{e:difG2off2}
D^{uv}_{k\ell}(G_{k\ell}R_{\bfi})
=\sum_{n=1}^{\fb} \frac{(\del^{uv}_{k\ell})^n (G_{k\ell}R_{\bfi})}{n!(d-1)^{n/2}}+\OO_\prec \left(d^{-\fb/2}\right).
\end{align}

If $\del^{uv}_{k\ell}$ hits $m$ in $R_\bfi$, which can be bounded by $\Im[m]/N\eta$ thanks to Proposition \ref{p:dermN}, and the resulting term is bounded by $\bE[\Phi_q]$. 
The other terms in 
 $(\del^{uv}_{k\ell})^n  (G_{k\ell}R_{\bfi})$ form a polynomial in Green's function entries. By plugging \eqref{e:difG2off2} into the second term in \eqref{e:decompcase1off2}, the term in \eqref{e:difG2off2} corresponding to $n$ results in a linear combination of 
\begin{align}\label{e:DRoff2}
\frac{d}{(d-1)^{(n+1)/2}}\frac{1}{N^{\theta(\cT^+)}d^{|E^+|}}\sum^*_{\bfi^+}\bE\left[A_{\cT^+}R_{\bfi^+}U\right],\quad \bfi^+=k\ell uv \bfi.
\end{align}
Thanks to Proposition \ref{p:DFU}, $R_{\bfi^+}$ contains an off-diagonal Green's function entry with indices in different connected components of $\cT^+$. 
If $n\geq 2$, \eqref{e:DRoff2} is in $\wt\cE_{\tilde \fh,\tilde \fo}$ with $\tilde \fh,\tilde \fo\geq 1$. Next we study the terms in \eqref{e:difG2off2} with $n=1$, 
\begin{align*}
\frac{d}{(d-1)}\frac{1}{N^{\theta(\cT^+)}d^{|E^+|}}\sum^*_{\bfi^+}\bE\left[A_{ku}A_{\ell v}\del_{k\ell}^{uv} (G_{k\ell}R_{\bfi})U\right],\quad \bfi^+=k\ell uv \bfi.
\end{align*}
There are two cases: i) if the derivative hits $R_{\bfi}$, then $R_{\bfi^+}$ contains the Green's function entry $G_{k\ell}$, and $k,\ell$ are in different connected component of $\cT^+$. Moreover, $\del_{k\ell}^{uv}  R_{\bfi}$ contains at least one off-diagonal Green's function entry. The corresponding term is in $\wt\cE_{0,2}$.
 ii) if the derivative hits $G_{k\ell}$, 
we get $R_{\bfi^+}=-G_{kx}G_{y\ell}G_{ij}R'_{\bfi}$, where $\{x,y\}\in \{\{k,\ell\}, \{u,v\}, \{k,u\}, \{y,v\}\}$. $R_{i^+}$ contains $G_{ij}$, and $i,j$ are in different connected component of $\cT^+$. Unless $(x,y)=(k,\ell)$, for all other cases, $R_{i^+}$ contains at least two off-diagonal terms. The corresponding terms are in $\wt\cE_{0,2}$. For the case $(x,y)=(k,\ell)$, we get
\begin{align}\label{e:lead1}
-\frac{d}{(d-1)}\frac{1}{N^{\theta(\cT^+)}d^{|E^+|}}\sum^*_{\bfi k\ell u v}\bE\left[A_{ku}A_{\ell v}G_{kk}G_{\ell\ell}R_{\bfi}U\right]
=-\frac{d}{(d-1)}\frac{1}{N^{|\bfi|}}\sum^*_{\bfi}\bE\left[m^2R_{\bfi}U\right]
+O_\prec[\Phi_q].
\end{align}
This gives \eqref{e:pp1}
\end{proof}

\begin{proof}[Proof of \eqref{e:pp2}]
We can rewrite \eqref{e:pp2} as
\begin{align}\label{e:dd}
\frac{1}{N^{|\bfi|}\sqrt{d-1}}\left(\sum^*_{k \bfi}\bE\left[A_{ik}G_{jk}mR'_{\bfi} U\right]
+\sum^*_{\bfi}\sum_{k\in \bfi}\bE\left[A_{ik}G_{jk}mR'_{\bfi} U\right]\right),
\end{align}
For the second term on the righthand of \eqref{e:dd}, we can simply bound it using $|G_{jk}m R_\bfi'|\prec 1$
\begin{align*}
\frac{1}{N^{|\bfi|}\sqrt{d-1}}\sum^*_{\bfi}\sum_{k\in \bfi}\bE\left[A_{ik}G_{jk}mR'_{\bfi} U\right]
\prec \frac{1}{N^{|\bfi|}\sqrt{d-1}}\sum^*_{\bfi}\sum_{k\in \bfi}
\bE\left[A_{ik}|P|^{2q-1}\right]\prec \bE[\Phi_q].
\end{align*}

For the first term on the righthand of \eqref{e:dd}, by using Proposition \ref{c:intbp2} with $\cT^+=(kuv\bfi, E^+=\{i,u\}\cup\{k,v\})$ and $F_{k \bfi}=G_{jk}m R'_\bfi U$. Then $|F_{k\bfi}|\prec \Lambda_o|P|^{2q-1}$ and we have
\begin{align}\begin{split} \label{e:case1off2}
&\phantom{{}={}}\frac{1}{N^{|\bfi|}\sqrt{d-1}}\sum^*_{k\ell \bfi}\bE\left[A_{ik}G_{jk}mR'_\bfi U\right]=\frac{d}{(d-1)^{1/2}}\frac{1}{N^{|\bfi|+1}}\sum^*_{k\bfi}\bE\left[G_{jk}mR'_\bfi U\right]\\
&+
\frac{d}{\sqrt{d-1}}\frac{1}{N^{\theta(\cT^+)}d^{|E^+|}}\sum^*_{k u v\bfi}\bE\left[A_{\cT^+}D_{ik}^{uv}(G_{jk}mR'_\bfi U)\right]+\OO_\prec\left(\bE[\Phi_q]\right).
\end{split}\end{align}
Since $\sum_{k}G_{jk}=0$, and $\sum^*_{k\notin \bfi}G_{jk}\prec 1$, the first term on the righthand side is bounded by $\bE[\Phi_q]$. For the second term on the righthand side of \eqref{e:case1off2}, we can rewrite it as a sum
\begin{align}\label{e:decompcase1off}
\frac{d}{\sqrt{d-1}}\frac{1}{N^{\theta(\cT^+)}d^{|E^+|}}\left(\sum^*_{k u v\bfi}\bE\left[A_{\cT^+}G_{jk}mR'_\bfi D_{ik}^{uv}U\right]+\sum^*_{kuv\bfi}\bE\left[A_{\cT^+}D_{ik}^{uv}(G_{jk}mR'_\bfi) (1+D_{ik}^{uv})U\right]\right).
\end{align}
For the first term in \eqref{e:decompcase1off}, since $j,k$ are in different connected components of $\cT^+$, it is bounded by $\bE[\Phi_q]$ thanks to Proposition \ref{p:one-off}. For the second term in \eqref{e:decompcase1off}, we have
\begin{align}\label{e:difG2off}
D^{uv}_{ik}(G_{jk}mR'_\bfi )
=\sum_{n=1}^{\fb-1} \frac{(\del^{uv}_{ik})^n (G_{jk}mR'_\bfi )}{n!(d-1)^{n/2}}+\OO_\prec \left(d^{-\fb/2}\right).
\end{align}
By plugging \eqref{e:difG2off} into the second term in \eqref{e:decompcase1off}, up to error $\bE[\Phi_q]$, the term in \eqref{e:difG2off} corresponding to $n$ results in a linear combination of 
\begin{align}\label{e:DRoff}
\frac{d}{(d-1)^{(n+1)/2}}\frac{1}{N^{\theta(\cT^+)}d^{|E^+|}}\sum^*_{\bfi^+}\bE\left[A_{\cT^+}R_{\bfi^+}U\right],\quad \bfi^+=k\ell uv \bfi.
\end{align}
Thanks to Proposition \ref{p:DFU}, $R_{\bfi^+}$ contains an off-diagonal Green's function entry with indices in different connected components of $\cT^+$. 
If $n\geq 2$, \eqref{e:DRoff} is in $\wt\cE_{\tilde \fh,\tilde \fo}$ with $\tilde \fh,\tilde  \fo\geq 1$. Next we study the terms in \eqref{e:difG2off} corresponding to $n=1$, 
\begin{align*}
\frac{d}{(d-1)}\frac{1}{N^{\theta(\cT^+)}d^{|E^+|}}\sum^*_{\bfi^+}\bE\left[A_{iu}A_{k v}\del_{ik}^{uv} (G_{jk}mR'_\bfi )U\right],\quad \bfi^+=k uv \bfi.
\end{align*}
There are two cases: i) if the derivative hits $mR'_\bfi $, then $R_{\bfi^+}$ contains the Green's function entry $G_{jk}$, and $j,k$ are in different connected component of $\cT^+$. Moreover, $\del_{ik}^{uv}  (mR'_\bfi)$ contains at least one off-diagonal Green's function entry. The corresponding term is in $\wt\cE_{0,2}$.
 ii) if the derivative hits $G_{jk}$, 
we get $R_{\bfi^+}=-G_{jx}G_{yk} mR'_\bfi $, where $\{x,y\}\in \{\{i,k\}, \{u,v\}, \{i,u\}, \{k,v\}\}$. 
For all cases, $G_{jx}$ is an off-diagonal term, and $j,x$ are in different connected components of $\cT^+$. Except when $(x,y)=(i,k), (v,k)$, $R_{\bfi^+}$ contains two off-diagonal Green's function term. The corresponding terms are in $\wt\cE_{0,2}$. When $(x,y)=(i,k), (v,k)$, we get
\begin{align}\begin{split}\label{e:dTjk2}
&\phantom{{}={}}-\frac{d}{(d-1)}\frac{1}{N^{\theta(\cT^+)}d^{|E^+|}}\sum^*_{\bfi^+}\bE\left[A_{iu}A_{k v}(G_{ji}G_{kk}+G_{jv}G_{kk})mR'_\bfi U\right]\\
&=-\frac{d}{(d-1)}\frac{1}{N^{|\bfi|}}\sum^*_{k\bfi^+}\bE\left[m^2 R_\bfi U\right]
-\frac{d}{(d-1)}\frac{1}{N^{|\bfi|+1}d}\sum^*_{\bfi^+}\bE\left[A_{k v}G_{jv}G_{kk}mR'_\bfi U\right].
\end{split}\end{align}
The first term on the righthand side of \eqref{e:dTjk2} is the main term in  \eqref{e:lead1}. For the second term on the righthand side of \eqref{e:dTjk2}, using Proposition \ref{c:intbp} with $\cT=(kuv\bfi , \{k,v\})$, $\cT^+=(xykuv\bfi, \{k,x\}\cup\{v,y\})$ and $F_{kuv\bfi}=G_{jv}G_{kk}mR_\bfi'U$, 
\begin{align}\begin{split}\label{e:lead2}
&\phantom{{}={}}\frac{1}{N^{|\bfi|+1}d}\sum^*_{\bfi^+}\bE\left[A_{k v}G_{jv}G_{kk}mR'_\bfi U\right]
=\frac{1}{N^{|\bfi|+2}}\sum^*_{\bfi^+}\bE\left[G_{jv}G_{kk}mR'_\bfi U\right]\\
&+\frac{1}{N^{|\bfi|+2}d^2}\sum^*_{xy\bfi^+}\bE\left[A_{kx}A_{vy}D_{kv}^{xy}(G_{jv}G_{kk}mR'_\bfi U)\right]+\OO_\prec(\bE[\Phi_q]).
\end{split}\end{align}
For the first term on the righthand side of \eqref{e:lead2}, since $\sum_{v}G_{iv}=0$, we have $\sum_{v\notin k\bfi}G_{iv}\prec 1$. It is bounded by $\OO_\prec(\bE[\Phi_q])$.
The second term on the righthand side of \eqref{e:lead2} is a linear combination of terms in $\wt\cE_{\tilde \fh,\tilde \fo}$ with $\tilde \fh,\tilde  \fo\geq 1$.
\end{proof}

%$R_{i^+}$ contains $G_{ij}$, and $i,j$ are in different connected component of $\cT^+$. Unless $(x,y)=(k,\ell)$, for all other cases, $R_{i^+}$ contains at least two off-diagonal terms. The corresponding term is in $\wt\cE_{0,2}$. For the case $(x,y)=(k,\ell)$, we get
%\begin{align}\label{e:lead1}
%-\frac{d}{(d-1)}\frac{1}{N^{\theta(\cT^+)}d^{|E^+|}}\sum^*_{\bfi k\ell u v}\bE\left[A_{ku}A_{\ell v}G_{kk}G_{\ell\ell}mR'_\bfi )U\right]
%=-\frac{d}{(d-1)}\frac{1}{N^{|\bfi|}}\sum^*_{\bfi}\bE\left[m^2R'_\bfi U\right]
%\end{align}
%
%
%
%
%The same argument as in \eqref{e:dG},
%Both pairs $j,x$ and $y,k$ are in the same connected component of $\cT_{r+1}$ unless $(x,y)=(j,k), (\ell, m)$. Therefore,  thanks to Proposition  \ref{c:one-off} the corresponding term \eqref{e:DR} is bounded by $\Phi_q$, unless $(x,y)=(j,k), (\ell, m)$. They correspond to these two terms
%\begin{align}\label{e:lasttwo}
%&-\frac{d}{d-1}\frac{1}{N^{\theta(\cT_{r+1})}d^{|E_{r+1}|}}\sum^*_{\bfi_{r+1}}\bE\left[A_{\cT_{r+1}}G_{jj}G_{kk}R_{\bfi_{r}}(1+D_{jk}^{\ell m}) U\right]\\
%&-\frac{d}{d-1}\frac{1}{N^{\theta(\cT_{r+1})}d^{|E_{r+1}|}}\sum^*_{\bfi_{r+1}}\bE\left[A_{\cT_{r+1}}G_{jj'}G_{kk'}R_{\bfi_{r}}(1+D_{jk}^{\ell m}) U\right],
%\end{align}
%where the second term in \eqref{e:lasttwo} is in $\wt\cE_{0,2, \fd+2}$.
%

\begin{proof}[Proof of \eqref{e:AGTck} with $E\neq \emptyset$]
We can use Proposition \ref{c:intbp} to expand, 
\begin{align}\begin{split}\label{e:AGTc5}
\phantom{{}={}}&\frac{1}{N^{\theta(\cT)}d^{|E|}}\sum^*_{\bf i}\bE\left[A_{\cT} R_{\bfi}U\right]
= \frac{1}{N^{|\bfi|}}\sum^*_{\bf i}\bE\left[G_{ij} R_\bfi' U\right]\\
+&
\frac{1}{N^{\theta(\cT^+)}d^{|E^+|}}\sum^*_{\bfi \bfj}\bE\left[A_{\cT^+} D_{\cT^+}(R_\bfi U) \right]+\OO\left(
\bE[\Phi_q]\right),
%=\cE_{\fh+1}+\OO_\prec(\bE[\Phi_q]).
\end{split}\end{align}
where $\cT^+=(\bfi\bfj, E^+)$ with $E^+=\cup_{e\in E}\{i_e, j_e\}\cup\{i_e', j_e'\}$.
For the first term on the righthand side of \eqref{e:AGTc5}, it is in $\wt\cE_{\fh,\fo}$ with $E=\emptyset$. From the discussion above, we know it can be rewritten as terms in $\wt\cE_{\tilde \fh,\tilde \fo}$ with $\tilde \fh+\tilde \fo\geq \fh+\fo+1$. For the second term in \eqref{e:AGTc5}, by the same argument as for the second term on the righthand side of \eqref{e:AGTc3}, up to an error $\OO_\prec(\bE[\Phi_q])$, it can be rewritten as terms in $\wt\cE_{\tilde \fh,\tilde \fo}$ with $\tilde \fh, \tilde \fo\geq 1$.
\end{proof}

\subsection{Proof of Proposition \ref{c:case1exp}}\label{s:expand1}

\begin{proof}[Proof of Proposition \ref{c:case1exp}]
We prove the statement for $\fh=0$, the general statement follows by multiplying both sides by $1/d^{\fh/2}$. By using Proposition \ref{c:intbp} we have
\begin{align}\begin{split} \label{e:rec2}
\frac{1}{N^{\theta(\cT)}d^{|E|}}\sum^*_{\bfi}\bE\left[A_{\cT}R_{\bfi} U\right]&=\frac{1}{N^{|\bfi|}}\sum^*_{\bfi}\bE\left[R_{\bfi} U_\bfi\right]\\
&+
\frac{1}{N^{\theta(\cT^+ )}d^{|E^+ |}}\sum^*_{\bfi^+}\bE\left[A_{\cT^+ }D_{\cT^+ }(R_{\bfi} U_\bfi)\right]+\OO_\prec\left(\bE[\Phi_q]\right),
\end{split}\end{align}
where $\cT^+ =(\bfi^+ , E^+ )$ with $\bfi^+=\bfi \bfj$ and $E^+ =\cup_{e\in E_r}\{i_e, j_e\}\cup\{i_e', j_e'\}$; Here we used Proposition \ref{p:someb} to bound the error term by $\bE[\Phi_q]$. 

Since $R_{\bfi}$ contains at least one off-diagonal Green's function entries, the first term on the righthand side is bounded by $\bE[\Phi_q]$ thanks to Proposition \ref{c:one-off}. For the second term on the righthand side of \eqref{e:rec2}, we can rewrite it as a sum
\begin{align}\label{e:decomp2}
\frac{1}{N^{\theta(\cT^+ )}d^{|E^+ |} }\left(\sum^*_{\bfi_1}\bE\left[A_{\cT^+ } R_{\bfi} D_{\cT^+ }U_\bfi\right]
+\sum^*_{\bfi}\bE\left[A_{\cT^+ } D_{\cT^+ }R_{\bfi} (1+D_{\cT^+ })U_\bfi\right]\right).
\end{align}
For the first term in \eqref{e:decomp2}, say $R_{\bfi}$ contains the off-diagonal Green's function entry $G_{ij}$. Then from our construction of $\cT^+ $, $i,j$ are in different connected components of $\cT^+ $, it is bounded by $\bE[\Phi_q]$ thanks to Proposition \ref{c:one-off}. 

For the second term in \eqref{e:decomp2}, by Taylor expansion \eqref{e:D-expand}
\begin{align}\label{e:difG2}
D_{\cT^+}R_{\bfi}
=\sum_{n=1}^{\fb-1} \frac{1}{n!(d-1)^{n/2}}\left(\sum_{e\in \cT}\del_{i_e j_e}^{i_e' j_e'}\right)^n R_{\bfi}+\OO_\prec \left(d^{-\fb/2}\right).
\end{align}
If the derivatives in \eqref{e:difG2} hit $m$, which can be bounded by $\Im[m]/N\eta$ thanks to Proposition \ref{p:dermN}, and the resulting term is bounded by $\bE[\Phi_q]$. 
Therefore up to error $\OO_\prec(\bE[\Phi_q])$ the second term in \eqref{e:decomp2} is a finite sum of terms in the form
\begin{align}\label{e:R++}
\frac{1}{d^{n/2}}\frac{1}{N^{\theta(\cT^+)}d^{|E^+|}}\sum^*_{\bfi^+}\bE\left[A_{\cT^+} R_{\bfi^+} (1+D_{\cT^+})U_\bfi\right],\quad n\geq 1.
\end{align}
If some off-diagonal Green's function entry $G_{ij}$ in $R_{\bfi}$ is not hit by a derivative in \eqref{e:difG2}, then $G_{ij}$ is in $R_{\bfi^+}$ and $i,j$ are in different connected components of $\cT^+ $. The corresponding term \eqref{e:R++} is bounded by $\bE[\Phi_q]$. The remaining terms correspond to the case that  in \eqref{e:difG2} each off-diagonal entry is hit by a derivative, and we have  $n\geq \fo$. 
The claim \eqref{e:nextod1} follows.
\end{proof}

\subsection{Proof of Propositions \ref{c:case2exp} }\label{s:expand2}
\begin{proof}[Proof of proposition \ref{c:case2exp}]
We prove the statement for $\fh=0$, the general statement follows by multiplying both sides by $1/d^{\fh/2}$. By the definition of the Green's function \eqref{e:GHexp}, we have
\begin{align}
\label{e:GHiterm}&\left(1-\frac{1}{N}\right)=-zG_{jj}+\frac{(AG)_{jj}}{\sqrt{d-1}},\quad \left(1-\frac{1}{N}\right)=-zG_{ii}+\frac{(AG)_{ii}}{\sqrt{d-1}}.
\end{align}
Multiplying the first relation in \eqref{e:GHiterm} by $G_{ii}$ and the second relation by $G_{jj}$, averaging over the indices, and then taking the difference, we get 
\begin{align}\label{e:Gii-m}
G_{ii}-m=\frac{1}{N(d-1)^{1/2}}\sum_{j:j\notin \bfi} (AG)_{jj}G_{ii}-\frac{1}{\sqrt{d-1}}(AG)_{ii}m +\OO_\prec\left(\frac{1}{N}\right).
\end{align}
By plugging \eqref{e:Gii-m} into \eqref{e:newterm}, we get
\begin{align}\begin{split}\label{e:replaceGii}
&\phantom{{}={}}\frac{1}{N^{\theta(\cT)}d^{|E |}}\sum^*_{\bfi }\bE\left[A_{\cT}R_{\bfi } U_\bfi\right]
=\frac{1}{N^{\theta(\cT)}d^{|E |}}\sum^*_{\bfi }\bE\left[A_{\cT} mR'_{\bfi } U_\bfi\right]\\
&+\frac{1}{N^{\theta(\cT)+1}d^{|E |}(d-1)^{1/2}}\sum^*_{j\bfi }\bE\left[A_{\cT} (AG)_{jj}G_{ii}R'_{\bfi } U_\bfi\right]\\
&-\frac{1}{N^{\theta(\cT)}d^{|E |}(d-1)^{1/2}}\sum^*_{\bfi }\bE\left[A_{\cT} (AG)_{ii}mR'_{\bfi } U_\bfi\right]+\OO_\prec(\bE[\Phi_q]).
\end{split}\end{align}
The claim of Proposition \ref{c:case2exp} follows from the next two estimates
\begin{align}
\label{e:guanxi1}\frac{\sum^*_{j\bfi }\bE\left[A_{\cT} (AG)_{jj}R_{\bfi } U_\bfi\right]}{N^{\theta(\cT)+1}d^{|E |}(d-1)^{1/2}}
&=-\frac{d}{d-1}\frac{\sum^*_{jk\ell m\bfi}\bE\left[A_{\cT}A_{j\ell}A_{km}G_{jj}G_{kk}R_{\bfi}U_\bfi\right]}{N^{\theta(\cT)+2}d^{|E|+2}}+\cE_{\tilde \fh, \tilde \fo}+\OO_\prec(\bE[\Phi_q]),\\
\label{e:guanxi2}
\frac{\sum^*_{\bfi }\bE\left[A_{\cT} (AG)_{ii}mR'_{\bfi } U_\bfi\right]}{N^{\theta(\cT)}d^{|E |}(d-1)^{1/2}}
&=-\frac{d}{d-1}\frac{\sum^*_{k\ell m\bfi}\bE\left[A_{\cT}A_{i\ell}A_{km}mG_{kk} R_{\bfi}U_\bfi\right]}{N^{\theta(\cT)+1}d^{|E|+2}}+\cE_{\tilde \fh, \tilde \fo}+\OO_\prec(\bE[\Phi_q]),
\end{align}
where $\tilde \fh\geq 0$ and $\tilde \fh+\tilde \fo\geq 1$, and the higher order terms $\cE_{\tilde \fh, \tilde \fo}$ are as in the first two cases listed in Proposition \ref{c:case2exp}.
Then the difference of \eqref{e:guanxi1} and \eqref{e:guanxi2} is given 
\begin{align*}
\frac{d}{d-1}\frac{1}{N^{\theta(\cT)+1}d^{|E|+2}}\left(\frac{1}{N}\sum^*_{jk\ell m\bfi}\bE\left[A_{\cT}A_{j\ell}A_{km}G_{jj}G_{kk}R_{\bfi}U_\bfi\right]
-\sum^*_{k\ell m\bfi}\bE\left[A_{\cT}A_{i\ell}A_{km}mG_{kk} R_{\bfi}U_\bfi\right]\right)\\
=\frac{\deg_\cT(i)}{d-1}\frac{1}{N^{\theta(\cT)}d^{|E|}}
\sum^*_{\bfi}\bE\left[A_{\cT}m^2 R_{\bfi}U_\bfi\right]+\OO_{\prec}(\bE[\Phi_q])
\in \cE_{2,0}+\OO_{\prec}(\bE[\Phi_q]),
\end{align*}
which is in the form of the last case listed in Proposition \ref{c:case2exp}.
\end{proof}

\begin{proof}[Proof of \eqref{e:guanxi1}]
We can rewrite \eqref{e:guanxi1} as
\begin{align}\begin{split}\label{e:ARU}
\frac{d}{(d-1)^{1/2}}\frac{1}{N^{\theta(\cT)+1}d^{|E |+1}}\sum^*_{jk\bfi }\bE\left[A_{\cT}A_{jk}G_{jk}R_{\bfi } U_\bfi\right]+\OO_\prec(\Phi_q),
\end{split}\end{align}
In \eqref{e:guanxi1} for the terms when $k\in \bfi$, we have bounded it by using \eqref{e:Uibound} and \eqref{e:sumA}:
\begin{align}
\frac{d}{(d-1)^{1/2}}\frac{1}{N^{\theta(\cT)+1}d^{|E |+1}}\sum_{j\bfi}^*\sum_{k\in \bfi }\bE\left[A_{\cT}A_{jk}|G_{jk}|\frac{N}{d}\Phi_q\right]\prec \bE[\Phi_q].
\end{align}

For \eqref{e:ARU}, by using Proposition \ref{c:intbp2} with $\cT^+=(\bfi^+, E^+)=(jk\ell m\bfi, E\cup\{j,\ell\}\cup\{k,m\})$, and $F_{jk\bfi}=G_{jk}R_\bfi U_\bfi$, we have
\begin{align}\begin{split} \label{e:case1exp}
&\phantom{{}={}}\frac{d}{(d-1)^{1/2}}\frac{1}{N^{\theta(\cT )+1}d^{|E |+1}}\sum^*_{jk\bfi }\bE\left[A_{\cT }A_{jk}G_{jk}R_{\bfi } U_\bfi\right]\\
&=\frac{d}{(d-1)^{1/2}}\frac{1}{N^{\theta(\cT )+2}d^{|E |}}\sum^*_{jk\bfi }\bE\left[A_{\cT }G_{jk}R_{\bfi } U_\bfi\right]\\
&+
\frac{d}{(d-1)^{1/2}}\frac{1}{N^{\theta(\cT^+)}d^{|E^+|}}\sum^*_{\bfi^+}\bE\left[A_{\cT}D_{jk}^{\ell m}(G_{jk}R_{\bfi } U_\bfi)\right]+\OO_\prec\left(\bE[\Phi_q]\right).
\end{split}\end{align}

%{\cor
%By using Corollary \ref{c:intbp} with $\cT_{r+1}=(\cT \cup\{j,k\})^+=(\bfi_{r+1}, E_{r+1})$ we have
%\begin{align}\begin{split} \label{e:case1exp}
%&\frac{d}{(d-1)^{1/2}}\frac{1}{N^{\theta(\cT )+1}d^{|E_r|+1}}\sum^*_{jk\bfi_r}\bE\left[A_{\cT_r}A_{jk}G_{jk}R_{\bfi_r} U_\bfi\right]\\
%&=\frac{d}{(d-1)^{1/2}}\frac{1}{N^{\theta(\cT_r)+1}d^{|E_r|+1}}\sum^*_{jk\bfi_r}\bE\left[G_{jk}R_{\bfi_r} U_\bfi\right]\\
%&+
%\frac{d}{(d-1)^{1/2}}\frac{1}{N^{\theta(\cT_{r+1})}d^{|E_{r+1}|}}\sum^*_{\bfi_{r+1}}\bE\left[A_{\cT_{r+1}}D_{\cT_{r+1}}(G_{jk}R_{\bfi_r} U)\right]+\OO_\prec\left(\frac{d\Phi}{N}\right)\\
%\end{split}\end{align}
%}

Since $\sum_{jk}G_{jk}=0$, and $\sum^*_{jk}G_{jk}\prec N$, the first term on the righthand side is bounded by $\bE[\Phi_q]$. For the second term on the righthand side of \eqref{e:case1exp}, we can rewrite it as a sum
\begin{align}\begin{split}\label{e:decompcase1}
&\phantom{{}={}}\frac{d}{d-1}\frac{1}{N^{\theta(\cT^+)}d^{|E^+|}}\sum^*_{\bfi^+}\bE\left[A_{\cT^+}G_{jk}R_{\bfi} \sqrt{d-1}D_{jk}^{\ell m}(U_\bfi)\right]\\
&+\frac{d}{(d-1)^{1/2}}\frac{1}{N^{\theta(\cT^+)}d^{|E^+|}}\sum^*_{\bfi^+}\bE\left[A_{\cT^+}D_{jk}^{\ell m}(G_{jk}R_{\bfi})(1+D_{jk}^{\ell m}) U_\bfi\right].
\end{split}\end{align}
For the first term in \eqref{e:decompcase1}, since $j,k$ are in different connected components of $\cT^+$, it is bounded by $\bE[\Phi_q]$ thanks to Proposition \ref{c:one-off}. For the second term in \eqref{e:decompcase1}, we have
\begin{align}\label{e:difG21}
D_{jk}^{\ell m}(G_{jk}R_{\bfi})
=\sum_{n=1}^{\fb-1} \frac{(\del_{jk}^{\ell m})^n (G_{jk}R_{\bfi})}{n!(d-1)^{n/2}}+\OO_\prec \left(d^{-\fb/2}\right).
\end{align}
By plugging \eqref{e:difG21} into the second term in \eqref{e:decompcase1}, up to error $\bE[\Phi_q]$, the term in \eqref{e:difG21} corresponding to $n$ results in a linear combination of 
\begin{align}\label{e:DR}
\frac{d}{(d-1)^{(n+1)/2}}\frac{1}{N^{\theta(\cT^+)}d^{|E^+|}}\sum^*_{\bfi^+}\bE\left[A_{\cT^+}R_{\bfi^+}(1+D_{jk}^{\ell m}) U_\bfi\right].
\end{align}
If $n\geq 2$, \eqref{e:DR} is in $\cE_{\tilde \fh, \tilde \fo}$ with $\tilde \fh\geq 1$. For the term in \eqref{e:difG21} corresponding to $n=1$, \begin{align}\label{e:dTjk}
\del_{jk}^{\ell m} (G_{jk}R_{\bfi})
\end{align}
is a linear combination of monomials $R_{\bfi^+}$ of Green's function entries. There are two cases: i) if the derivative hits $R_{\bfi}$, then $R_{\bfi^+}$ contains the Green's function entry $G_{jk}$. Since $j,k$ are in different connected component of $\cT^+$, the corresponding term \eqref{e:DR} is bounded by $\bE[\Phi_q]$, thanks to Proposition  \ref{c:one-off}. ii) if the derivative hits $G_{jk}$, 
we get $R_{\bfi^+}=-G_{jx}G_{yk}R_{\bfi}$, where $\{x,y\}\in \{\{j, k\}, \{\ell,m\}, \{j,\ell\}, \{k,m\}\}$. 
One of the pairs $j,x$ and $y,k$ are in different connected components of $\cT^+$ unless $(x,y)=(j,k), (\ell, m)$. Therefore,  thanks to Proposition  \ref{c:one-off} the corresponding terms \eqref{e:DR} are bounded by $\bE[\Phi_q]$, unless $(x,y)=(j,k), (\ell, m)$. They correspond to these two terms
\begin{align}\label{e:lasttwo}\begin{split}
&-\frac{d}{d-1}\frac{1}{N^{\theta(\cT ^+)}d^{|E ^+|}}\sum^*_{\bfi ^+}\bE\left[A_{\cT ^+}G_{jj}G_{kk}R_{\bfi}(1+D_{jk}^{\ell m}) U_\bfi\right]\\
&-\frac{d}{d-1}\frac{1}{N^{\theta(\cT ^+)}d^{|E ^+|}}\sum^*_{\bfi ^+}\bE\left[A_{\cT ^+}G_{j\ell}G_{km}R_{\bfi}(1+D_{jk}^{\ell m}) U_\bfi\right],
\end{split}\end{align}
where the first term gives \eqref{e:guanxi1} and the second term in \eqref{e:lasttwo} is in $\cE_{0,2}$.
\end{proof}

\begin{proof}[Proof of \eqref{e:guanxi2}]
For \eqref{e:guanxi2}, we separate it into two terms corresponding to $k\in \qq{N}\setminus \bfi$ or $k\in \bfi\setminus \{i\}$:
\begin{align}\label{e:rewrite}\begin{split}
&\phantom{{}={}}\frac{1}{N^{\theta(\cT )}d^{|E |}(d-1)^{1/2}}\sum^*_{\bfi }\bE\left[A_{\cT } (AG)_{ii}mR'_{\bfi } U_\bfi\right]\\
&=\frac{1}{N^{\theta(\cT )}d^{|E |}(d-1)^{1/2}}\sum^*_{k\bfi }\bE\left[A_{\cT } A_{i k}G_{ik}mR'_{\bfi } U_\bfi\right]\\
&+\frac{1}{N^{\theta(\cT )}d^{|E |}(d-1)^{1/2}}\sum^*_{\bfi ,k\in \bfi \setminus \{i\}}\bE\left[A_{\cT } A_{ik}G_{ik}mR'_{\bfi } U_\bfi\right].
\end{split}\end{align}

For the first term on the righthand side of \eqref{e:rewrite}, 
By using Proposition \ref{c:intbp2} with $\cT ^+=(k\ell m \bfi , E \cup\{i,\ell\}\cup\{k, m\})$ and $F_{k\bfi}=G_{ik}m R_\bfi' U$, we have $d^{3/2}\cC(F,A)/N\prec \Phi_q$ and
\begin{align}\begin{split} \label{e:case1exp2}
&\phantom{{}={}}\frac{d}{(d-1)^{1/2}}\frac{1}{N^{\theta(\cT )}d^{|E |+1}}\sum^*_{k\bfi }\bE\left[A_{\cT } A_{i k}G_{ik}mR'_{\bfi } U\right]\\
&=\frac{d}{(d-1)^{1/2}}\frac{1}{N^{\theta(\cT )+1}d^{|E |}}\sum^*_{k\bfi }\bE\left[A_{\cT }G_{ik}mR'_{\bfi } U\right]\\
&+
\frac{d}{(d-1)^{1/2}}\frac{1}{N^{\theta(\cT ^+)}d^{|E ^+|}}\sum^*_{\bfi ^+}\bE\left[A_{\cT ^+}D_{ik}^{\ell m}(G_{ik}mR'_{\bfi } U)\right]+\OO_\prec\left(\bE[\Phi_q]\right).
\end{split}\end{align}
Since $\sum_{k}G_{ik}=0$, and $\sum_{k\in \qq{N}\setminus\bfi }G_{ik}\prec 1$, the first term on the righthand side is bounded by $\bE[\Phi_q]$. For the second term on the righthand side of \eqref{e:case1exp2}, we can rewrite it as a sum
\begin{align}\label{e:decompcase2}\begin{split}
&\frac{d}{d-1}\frac{1}{N^{\theta(\cT ^+)}d^{|E ^+|}}\sum^*_{\bfi ^+}\bE\left[A_{\cT ^+}G_{ik}mR'_{\bfi } \sqrt{d-1}D_{ik}^{\ell m}(U)\right]\\
&+\frac{d}{(d-1)^{1/2}}\frac{1}{N^{\theta(\cT ^+)}d^{|E ^+|}}\sum^*_{\bfi ^+}\bE\left[A_{\cT ^+}D_{ik}^{\ell m}(G_{ik}mR'_{\bfi })(1+D_{ik}^{\ell m}) U\right].
\end{split}\end{align}
For the first term on the righthand side of \eqref{e:decompcase2}, since $i,k$ are in different connected components of $\cT ^+$, it is bounded by $\bE[\Phi_q]$ thanks to Proposition \ref{c:one-off}. For the second term in \eqref{e:decompcase2}, we have
\begin{align}\begin{split}\label{e:difG3} 
D_{ik}^{\ell m}(G_{ik}mR'_{\bfi })
&=\sum_{n=1}^{\fb-1} \frac{1}{n!(d-1)^{n/2}}\left(\del_{ik}^{\ell m}\right)^n (G_{ik}mR'_{\bfi})+\OO_\prec \left(d^{-\fb/2}\right).
%&=\sum_{n=1}^{\fb-1} \frac{m}{n!(d-1)^{n/2}}\left(\del_{ik}^{\ell m}\right)^n (G_{ik}R'_{\bfi})+\OO_\prec \left(d^{-\fb/2}+\frac{\Im[m]}{N\eta}\right),
\end{split}\end{align}
If the derivative $\del_{ik}^{\ell m}$ hits $m$, which is bounded by $\Im[m]/(N\eta)$ thanks to \eqref{e:impp}, and the resulting terms are bounded by $\bE[\Phi_q]$. The remaining terms in $(\del_{ik}^{\ell m})^n (G_{ik} m R_{\bfi})$ form a polynomial in Green's function entries. By plugging \eqref{e:difG3} into the second term in \eqref{e:decompcase2}, the term in \eqref{e:difG3} corresponding to $n$ results in a linear combination of 
\begin{align}\label{e:DR2}
\frac{d}{(d-1)^{(n+1)/2}}\frac{1}{N^{\theta(\cT ^+)}d^{|E ^+|}}\sum^*_{\bfi ^+}\bE\left[A_{\cT ^+}R_{\bfi ^+}(1+D_{ik}^{\ell m}) U_\bfi\right].
\end{align}
If $n\geq 2$, \eqref{e:DR2} is in $\cE_{\tilde \fh, \tilde \fo}$ with $\tilde \fh\geq 1$. For the term in \eqref{e:difG3} corresponding to $n=1$, 
\begin{align*}
\del_{ik}^{\ell m}(G_{ik} mR'_{\bfi})
\end{align*}
is a linear combination of monomials $R_{\bfi ^+}$ of Green's function entries. By the same argument as for \eqref{e:dTjk}, either \eqref{e:DR2} is given by
\begin{align}\begin{split}\label{e:lasttwo2}
&-\frac{d}{d-1}\frac{1}{N^{\theta(\cT ^+)}d^{|E ^+|}}\sum^*_{\bfi ^+}\bE\left[A_{\cT ^+}mG_{ii}G_{kk}R'_{\bfi}(1+D_{ik}^{\ell m}) U_\bfi\right]\\
&-\frac{d}{d-1}\frac{1}{N^{\theta(\cT ^+)}d^{|E ^+|}}\sum^*_{\bfi ^+}\bE\left[A_{\cT ^+}m G_{i\ell}G_{km}R'_{\bfi}(1+D_{ik}^{\ell m}) U_\bfi\right],
\end{split}\end{align}
or bounded by $\bE[\Phi_q]$.The first term gives \eqref{e:guanxi1} and the second term is in $\cE_{0, 2}$.

%There are two cases: i) if the derivative hits $R'_{\bfi_r}$, then $R'_{\bfi_{r+1}}$ contains the Green's function entry $G_{ik}$. Since $i,k$ are in different connected component of $\cT_{r+1}$, the corresponding term \eqref{e:DR2} is bounded by $\Phi_q$, thanks to Proposition  \ref{c:one-off}. ii) if the derivative hits $G_{ik}$, we get $R_{\bfi_{r+1}}=G_{ix}G_{yk}R_{\bfi_r}$, where $\{x,y\}\in \{\{i_e, j_e\}, \{i'_e, j'_e\}, \{i_e, i'_e\}, \{j_e, j'_e\}\}$ for $e\in \cT_r\cup\{j,k\}$. $j,x$ are in different connected component of $\cT_{r+1}$ unless $x=j$ or $x=j'$. And $y,k$ are in different connected component of $\cT_{r+1}$ unless $y=k$ or $y=k'$. Therefore,  thanks to Proposition  \ref{c:one-off} the corresponding term \eqref{e:DR} is bounded by $\Phi_q$, unless $\{x,y\}=\{j,k\}$ or $\{j',k'\}$. They correspond to these two terms
%\begin{align}\label{e:lasttwo}
%&-\frac{d}{d-1}\frac{1}{N^{\theta(\cT_{r+1})}d^{|E_{r+1}|}}\sum^*_{\bfi_{r+1}}\bE\left[A_{\cT_{r+1}}G_{jj}G_{kk}R_{\bfi_{r}}(1+D_{\cT_{r+1}}) U\right]\\
%&-\frac{d}{d-1}\frac{1}{N^{\theta(\cT_{r+1})}d^{|E_{r+1}|}}\sum^*_{\bfi_{r+1}}\bE\left[A_{\cT_{r+1}}G_{jj'}G_{kk'}R_{\bfi_{r}}(1+D_{\cT_{r+1}}) U\right],
%\end{align}
%where the second term in \eqref{e:lasttwo} is in $\cE_{0,2, \fd+2}$.
%
%For $n=1$, 
%\begin{align}
%\frac{1}{(d-1)^{1/2}}\left(\sum_{e\in \cT_r\cup\{j,k\}}\del_{i_e j_e}^{i_e' j_e'}\right) (G_{ik}mR'_{\bfi_r})
%=\frac{1}{(d-1)^{1/2}} \left(-mG_{kk}R_{\bfi_r}+G_{ik}m\left(\sum_{e\in \cT_r\cup\{j,k\}}\del_{i_e j_e}^{i_e' j_e'}\right)R'_{\bfi_r} +\right)
%\end{align}
%

For the second term on the righthand side of \eqref{e:rewrite}, there are two cases either the edge $\{i,k\}\in E$, or $k\in \bfi\setminus\{i\}$ and $\{i,k\}\not\in E$. For the case $\{i,k\}\in E $, then $A_{\cT }A_{ik}=A_{\cT }$, we can rewrite the last term in \eqref{e:rewrite} as
\begin{align*}
\frac{1}{N^{\theta(\cT )}d^{|E |}(d-1)^{1/2}}\sum^*_{\bfi }\bE\left[A_{\cT } G_{ik}mR'_{\bfi } U\right],\end{align*}
which is in $\cE_{1,1}$.

If $k\in \bfi \setminus\{i\}$ and $\{i,k\}\not\in E $, there are two cases: i) $\{i,k\}$ are in different connected component of $\cT $.  ii) $\{i,k\}$ are in the same connected components of $\cT $.
In case i), $\cT \cup \{i,k\}$ is still a forest with $\theta(\cT )-1$ connected components, and $|E |+1$ edges. Using \eqref{e:sumA},
\begin{align}\label{e:ssN}\begin{split}
&\frac{1}{N^{\theta(\cT )}d^{|E |}(d-1)^{1/2}}\sum^*_{\bfi }\bE\left[A_{\cT } A_{ik}G_{ik}mR'_{\bfi } U_\bfi\right]\\
&\prec \frac{\Lambdao}{N^{\theta(\cT )}d^{|E |}(d-1)^{1/2}}\sum^*_{\bfi }\bE\left[A_{\cT } A_{ik}| U_\bfi|\right]
\prec \frac{\sqrt{d}\Lambdao}{N}\bE[U_\bfi]\leq \bE[\Phi_q].
\end{split}\end{align}
In case ii), $i,k$ belong to the same connected component. Since $\{i,k\}\not\in \cT $, $\cT \cup \{i,k\}$ contains exactly one cycle. Proposition \ref{p:sumA} implies that
$
(1/N^{\theta(\cT )}d^{|E |})\sum^*_{\bfi }A_{\cT }A_{ik}\prec 1/N,
$
and the same estimate \eqref{e:ssN} holds in this case. 
\end{proof}

\subsection{Proof of Proposition \ref{p:rewrite}}\label{s:rewrite}

\begin{proof}[Proof of Proposition \ref{p:rewrite}]
We will prove the following relation for any forest $\cT=(\bfi, E)$, 
\begin{align}\label{e:induction}
\frac{1}{N^{\theta(\cT^+)}d^{|E^+|}}\sum^*_{\bfi^+}\bE[A_{\cT^+} m^\fd U_{\bfi^+}]
=\frac{C}{N^{\theta(\cT )}d^{|E |}}\sum^*_{\bfi }\bE[A_{\cT } m^\fd U_\bfi]+\OO_\prec(\sqrt{d}\bE[\Phi_q]),
\end{align}
where the graphs $ \cT^+$ is as constructed in Propositions \ref{c:case1exp} and \ref{c:case2exp}. Then by sequentially using the relation \eqref{e:induction} for $\cT=\cT_{t-1},\cdots, \cT_1$, we get
\begin{align}\label{e:finalexp}
\frac{1}{d^{\fh/2}}\frac{1}{N^{\theta(\cT_t)}d^{|E_t |}}\sum^*_{\bfi_t }\bE[A_{\cT_t } m^\fd U_{\bfi_t}]
=\frac{C}{d^{\fh/2}}\sum^*_{ijk\ell}\frac{1}{N^2d^2}\bE[A_{ik}A_{j\ell} m^\fd U_{\bfi_1}]+\OO_\prec(\bE[\Phi_q]),
\end{align}
where  $U_{\bfi_1}=P^q \bar P^{q-1}$ or $U_{\bfi_1}=\sqrt{d-1}D_{ij}^{k\ell}P^q \bar P^{q-1}$. We remark that $d^{-\fh/2}$ cancels the $\sqrt{d}$ factor in the error term of \eqref{e:induction}.

To prove \eqref{e:induction}, we first notice that 
\begin{align}\begin{split}\label{e:TEexp}
&\phantom{{}={}}\frac{1}{N^{\theta(\cT^+)}d^{|E^+|}}\sum^*_{\bfi^+}\bE[A_{\cT^+} m^\fd (e+D_{\cT^+}) U_{\bfi}]\\
&=\frac{1}{N^{\theta(\cT^+)}d^{|E^+|}}\sum^*_{\bfi^+}\bE[A_{\cT^+} (e+D_{\cT^+})(m^\fd  U_\bfi)]+\OO_\prec(\bE[\Phi_q]),\quad e\in\{1,0\}.
\end{split}\end{align}

If $\cT ^+$ is constructed from $\cT$ in Proposition \ref{c:case1exp}, then $U_{\bfi^+}=(1+D_{\cT^+})U_{\bfi}$. Using \eqref{e:TEexp}, we can first rewrite the lefthand side of \eqref{e:induction} as the sum of two terms
\begin{align}\begin{split}\label{e:twott}
&\phantom{{}={}}\frac{1}{N^{\theta(\cT^+)}d^{|E^+|}}\sum^*_{\bfi ^+}\bE[A_{\cT ^+} (1+D_{\cT ^+}) (m^\fd  U_\bfi)]\\
&=\frac{1}{N^{\theta(\cT^+)}d^{|E^+|}}\sum^*_{\bfi^+}\bE[A_{\cT^+}m^\fd  U_\bfi]+\frac{1}{N^{\theta(\cT^+)}d^{|E^+|}}\sum^*_{\bfi ^+}\bE[A_{\cT ^+} D_{\cT ^+} (m^\fd  U_\bfi)].
\end{split}\end{align}

We recall from Remark \ref{r:Tstructure} that $\theta(\cT^+)=|\bfi|$ and $|E^+|=2|E|$. We can rewrite the first term on the righthand  side of \eqref{e:twott} as
\begin{align}\begin{split}\label{e:sumAT}
&\phantom{{}={}}\frac{1}{N^{\theta(\cT^+)}d^{|E^+|}}\sum^*_{\bfi^+}\bE[A_{\cT^+}m^\fd  U_\bfi]
=\frac{1}{N^{|\bfi|}d^{2|E|}}\sum^*_{\bfi \cup_{e\in E}\{j_e, j_e'\}}\bE\left[\prod_{e\in E}A_{i_ej_e}A_{i_e'j_e'}m^\fd  U_\bfi\right]\\
&=\frac{1}{N^{|\bfi|}}\sum^*_{\bfi}\bE[m^\fd  U_\bfi]+\OO_\prec(\bE[\Phi_q]),
\end{split}\end{align}
where to get the second line, we sum over the indices $j_e, j_e'$, which gives a factor $d^{2|E|}$. 
Using  \eqref{e:TEexp} and \eqref{e:twott}, \eqref{e:induction} follows from Proposition \ref{c:intbp}, i.e., 
\begin{align*}
&\phantom{{}={}}\frac{1}{N^{\theta(\cT^+)}d^{|E^+|}}\sum^*_{\bfi ^+}\bE[A_{\cT ^+} (1+D_{\cT ^+}) (m^\fd  U_\bfi)]\\
&=\frac{1}{N^{|\bfi|}}\sum^*_{\bfi}\bE[m^\fd  U_\bfi]+\frac{1}{N^{\theta(\cT^+)}d^{|E^+|}}\sum^*_{\bfi ^+}\bE[A_{\cT ^+} D_{\cT ^+} (m^\fd  U_\bfi)]+\OO_\prec(\bE[\Phi_q])\\
&=\frac{1}{N^{\theta(\cT )}d^{|E |}}\sum^*_{\bfi }\bE[A_{\cT } m^\fd U_\bfi]+\OO_\prec(\bE[\Phi_q]).
\end{align*}

If $\cT ^+$ is constructed from $\cT  $ as in Proposition \ref{c:case2exp}, then there are three cases. 
In the first case $\cT ^+=(jk\ell m\bfi, E\cup \{j,\ell\}\cup\{k,m\})$, we have $\theta(\cT^+)=\theta(\cT)+2$ and $|E^+|=|E|+2$. If $U_{\bfi^+}=(1+D_{jk}^{\ell m})U_{\bfi}$, thanks to Proposition \ref{c:intbp2} (with the  identities used  in the reverse direction), we have 
\begin{align}\begin{split}\label{e:back1}
&\phantom{{}={}}\frac{1}{N^{\theta(\cT ^+)}d^{|E ^+|}}\sum^*_{\bfi ^+}\bE[A_{\cT}A_{j\ell}A_{km} (1+D_{jk}^{\ell m})( m^\fd  U_{\bfi})]\\
&=\frac{1}{N^{\theta(\cT  )+2}d^{|E  |}}\sum^*_{\bfi jk }\bE[A_{\cT  } m^\fd U_{\bfi}]+\frac{1}{N^{\theta(\cT ^+)}d^{|E ^+|}}\sum^*_{ \bfi jk\ell m}\bE[A_{\cT}A_{j\ell}A_{km} D_{jk}^{\ell m}( m^\fd  U_{\bfi})]+\OO_\prec(\bE[\Phi_q])\\
&=\frac{1}{N^{\theta(\cT  )+1}d^{|E  |+1}}\sum^*_{\bfi jk  }\bE[A_{\cT  }A_{jk} m^\fd U_{\bfi}]+\OO_\prec(\bE[\Phi_q])\\
&=\frac{1}{N^{\theta(\cT  )}d^{|E  |}}\sum^*_{\bfi  }\bE[A_{\cT  } m^\fd U_{\bfi}]+\OO_\prec(\bE[\Phi_q]).
\end{split}\end{align}
If $U_{\bfi^+}=\sqrt{d-1}D_{jk}^{\ell m}U_{\bfi}$, we notice that 
\begin{align}\label{e:back2}
\frac{1}{N^{\theta(\cT ^+)}d^{|E ^+|}}\sum^*_{\bfi ^+}\bE[A_{\cT}A_{j\ell}A_{km}  m^\fd  U_{\bfi}]
&=\frac{1}{N^{\theta(\cT  )}d^{|E  |}}\sum^*_{\bfi  }\bE[A_{\cT  } m^\fd U_{\bfi}]+\OO_\prec(\bE[\Phi_q]).
\end{align}
Then by taking difference of \eqref{e:back1} and \eqref{e:back2}, we get
\begin{align}\label{e:back3}
\frac{1}{N^{\theta(\cT ^+)}d^{|E ^+|}}\sum^*_{\bfi ^+}\bE[A_{\cT}A_{j\ell}A_{km} D_{jk}^{\ell m} m^\fd  U_{\bfi}]
=\OO_\prec(\bE[\Phi_q]),
\end{align}
and \eqref{e:induction} holds with $C=0$. 
%In the first case $\cT ^+=(\cT  \cup\{j,k\})^+$, where $j,k\not\in \bfi_r$.
%\begin{align}
%\frac{1}{N^{\theta(\cT ^+)}d^{|E ^+|}}\sum^*_{\bfi_{r+1}}\bE[A_{\cT_{r+1}} (1+D_{\cT_{r+1}}) m^\fd  U]
%&=\frac{1}{N^{\theta(\cT_r)+1}d^{|E_r|+1}}\sum^*_{jk\bfi_r}\bE[A_{\cT_r}A_{jk} m^\fd U]+\OO_\prec(\Phi_q)\\
%&=\frac{1}{N^{\theta(\cT_r)}d^{|E_r|}}\sum^*_{\bfi_r}\bE[A_{\cT_r} m^\fd U]+\OO_\prec(\Phi_q)
%\end{align}

In the second case $\cT^+=(k\ell m \bfi, E\cup\{i,\ell\}\cup\{k,m\})$, where $i\in \bfi$. 
We have that $\theta(\cT^+)=\theta(\cT)+1$ and $|E^+|=|E|+2$. 
If $U_{\bfi^+}=(1+D_{ik}^{\ell m})U_{\bfi}$, thanks to Proposition \ref{c:intbp2} (with the  identities used  in the reverse direction)
\begin{align}\label{e:back4}\begin{split}
&\phantom{{}={}}\frac{1}{N^{\theta(\cT^+ )}d^{|E^+ |}}\sum^*_{\bfi^+ }\bE[A_{\cT}A_{i\ell}A_{km} (1+D_{ik}^{\ell m}) m^\fd  U_{\bfi}]\\
&=\frac{1}{N^{\theta(\cT )+1}d^{|E |}}\sum^*_{\bfi k}\bE[A_{\cT } m^\fd U_{\bfi}]+\frac{1}{N^{\theta(\cT^+ )}d^{|E^+ |}}\sum^*_{k\ell m \bfi }\bE[A_{\cT}A_{i\ell}A_{km} D_{ik}^{\ell m} m^\fd  U_{\bfi}]+\OO_\prec(\bE[\Phi_q])\\
&=\frac{1}{N^{\theta(\cT )}d^{|E |+1}}\sum^*_{k\bfi }\bE[A_{\cT }A_{ik} m^\fd U_{\bfi}]+\OO_\prec(\bE[\Phi_q])\\
&=\left(1-\frac{\deg_\cT(i)}{d}\right)\frac{1}{N^{\theta(\cT )}d^{|E |}}\sum^*_{\bfi }\bE[A_{\cT } m^\fd U_{\bfi}]+\OO_\prec(\bE[\Phi_q]).
\end{split}\end{align}
If $U_{\bfi^+}=\sqrt{d-1}D_{ik}^{\ell m} U_{\bfi}$, we notice that 
\begin{align}\label{e:back5}\begin{split}
&\phantom{{}={}}\frac{1}{N^{\theta(\cT^+ )}d^{|E^+ |}}\sum^*_{\bfi^+ }\bE[A_{\cT}A_{i\ell}A_{km}  m^\fd  U_{\bfi}]\\
&=\left(1-\frac{\deg_\cT(i)}{d}\right)\frac{1}{N^{\theta(\cT )}d^{|E |}}\sum^*_{\bfi }\bE[A_{\cT } m^\fd U_{\bfi}]+\OO_\prec(\bE[\Phi_q]).
\end{split}\end{align}

Then by taking difference of \eqref{e:back4} and \eqref{e:back5}, we get
\begin{align}\label{e:back3}
\frac{1}{N^{\theta(\cT^+ )}d^{|E^+ |}}\sum^*_{\bfi^+ }\bE[A_{\cT}A_{i\ell}A_{km} D_{ik}^{\ell m} m^\fd  U_{\bfi}]
=\OO_\prec(\bE[\Phi_q]),
\end{align}
and \eqref{e:induction} holds with $C=0$. 

In the last case $\cT^+=\cT $, and $U_{\bfi^+}=U_{\bfi}$. There is nothing to prove. 

For the righthand side of \eqref{e:finalexp}, by the same argument as in \eqref{e:back1} and \eqref{e:back2}, we can rewrite it as $C\bE[m^\fd P^{q-1}\bar P^q]+\OO_{\prec}(\bE[\Phi_q])$ if $U_{\bfi_1}=P^q \bar P^{q-1}$; and $\OO_{\prec}(\bE[\Phi_q])$ if $U_{\bfi_1}=\sqrt{d-1}D_{ij}^{k\ell}P^q \bar P^{q-1}$. The claim \eqref{e:back} follows. 
\end{proof}

\section{Identification of the Self-consistent Equation}
\label{sec:P-identification}

The procedure that generates the polynomial $P$ from Proposition \ref{p:DSE} is explicit but quite complicated, so that explicitly tracking the resulting coefficients of $ P $ is challenging beyond the first few orders. In this section we show these coefficients (asymptotically)
are close to those of the power series $P_\infty(z,w)$ from \eqref{e:md-sce}, which characterizes the Stieltjes transform $\md$ of the Kesten--McKay law.

\begin{proposition}\label{p:idfP}
Uniformly in $z\in \bC_+$,
the polynomial $ P (z,u) = 1 + zu +u^2+ Q(u)$ constructed in Proposition~\ref{p:DSE} satisfies
\begin{equation*}
 P (z, \md(z))= \OO(d/N),
\end{equation*}
where $\md$ is the Stieltjes transform of the Kesten--McKay law, given by \eqref{e:md}.
\end{proposition}

The following corollary follows from Proposition \ref{p:idfP}. Their proofs are essentially the same as \cite[Corollary 6.2, Corollary 6.4]{bauerschmidt2020edge}, so we omit them. 
\begin{corollary}\label{c:Pproperty}
Let $ P $ be the polynomial constructed in Proposition \ref{p:DSE}. Then
$P(z,w)-P_{\infty}(z,w)  = Q(w)-Q_\infty(w)$
is a power series in $w$, which converges on the whole complex plane. Each of its coefficients is of order $\OO(d/N)$.
\end{corollary}

\begin{corollary}\label{c:mdproperty}
The polynomial $ P $ constructed in Proposition \ref{p:DSE} satisfies
\begin{align}\label{e:derPa}
  | \del_2P (z,\md)|\asymp \sqrt{|\kappa|+\eta}+ \OO(d/N),
  \quad
   \del_2^2P (z,\md)=2+ \OO ( d^{-1/2}),
  %\quad
  %\del_2 P''(z,\md) = \OO(1),
\end{align}
where $z=2+\kappa+\ri \eta$ or $z=-2-\kappa+\ri\eta$ for $\eta\leq \fK, -3/2\leq \kappa\leq \fK$, where the constant $\fK$ is from \eqref{e:D}. 
\end{corollary}

\subsection{Proof of Proposition~\ref{p:idfP}}
We recall that the Green's function of the infinite $d$-regular tree is given by
\begin{equation}\label{e:treeG}
G^{\text{tree}}_{ij}(z) = m_d(z) \pbb{- \frac{\msc(z)}{\sqrt{d-1}}}^{\dist(i,j)}.
\end{equation}
In particular if $\dist(i,j)=1$, we have $G^{\text{tree}}_{ij}=-m_d m_{sc}/\sqrt{d-1}$. We will also need the following identity between $m_{sc}$ and $m_d$,
which is a reorganization of \eqref{e:md}
\begin{align}\label{e:multiHG}
1=-z m_d+\sum_{j = 1}^N\frac{A_{ij}}{\sqrt{d-1}} \left(- \frac{\md \msc}{\sqrt{d - 1}}\right).
\end{align}

To prove Proposition \ref{p:idfP}, we will follow the same procedure as in Section \ref{sec:P-construct}. We will start by computing the quantity $(1+zm_d)P(z,m_d)^{q-1}\bar P(z,m_d)^q$.
\begin{proposition}\label{p:minfty}
The polynomial $ P (z,u) = 1 + zu + u^2+Q(u)$ constructed in Proposition~\ref{p:DSE} satisfies
\begin{align}\label{e:selfconsist}
(1+zm_d)P^{q-1}\bar P^q=-(m_d^2+Q(m_d))P^{q-1}\bar P^q +O_\prec(\wh\Phi_q),
\quad \wh\Phi_q=\frac{d}{N}|P|^{2q-1}.
\end{align} 
\end{proposition}
We show that all the estimates in Section \ref{sec:P-construct} still hold if we do the following substitution
\begin{align}\label{e:replace}
\frac{1}{N^{\theta(\bfi)}d^{|E|}}\sum^*_{\bfi}\bE[A_\cT R_\bfi U_\bfi]
\Rightarrow
\frac{1}{N^{\theta(\bfi)}d^{|E|}}\sum^*_{\bfi}A_\cT \wh R^\cT_\bfi \wh U^\cT_\bfi,
\end{align}
where $\cT=(\bfi, E)$ is any forest (recall from \eqref{d:forest})
\begin{align}\label{e:replace2}
\wh G^\cT_{ij}= m_d \pbb{- \frac{\msc}{\sqrt{d-1}}}^{\dist_\cT(i,j)}, \quad 
\wh R^\cT_\bfi=R_\bfi(\{\wh G^\cT_{ij}\}_{ij\in\bfi}, m_d),
\quad \wh U^\cT_\bfi=U_\bfi(m_d, \bar m_d).
\end{align}

We remark that in \eqref{e:replace} and \eqref{e:replace2}, we have replaced the Green's function entry $G_{ij}$ by the Green's function $\wh G^\cT_{ij}$ of the $d$-regular tree \eqref{e:treeG}, and the Stieltjes transform $m$ to $m_d$. If $i,j$ are in different connected component of $\cT$, then $\dist_{\cT}(i,j)=\infty$, and $\wh G^\cT_{ij}=0$. The forest $\cT$ consists of $\theta(\cT)$ connected components, each component is a tree. We can embed $\cT$ into $\theta(\cT)$ copies of infinite $d$-regular trees, i.e. each connected component is embedded into an infinite $d$-regular tree. 
Then $\wh G^\cT_{ij}$ can be viewed as the Green's function of the collection of  infinite $d$-regular trees.  

The new expression on the righthand side of \eqref{e:replace} depends mainly on the forest $\cT$. In fact $\wh R^\cT_\bfi$ and $\wh U^\cT_\bfi$ depends only on the forest $\cT$ and is independent of the choice of the index set $\bfi$. We can rewrite \eqref{e:replace} as
%\hty{ $\wh R^\cT_\bfi$ actually depends on some strange process which generates the forest $\cT$. }
\begin{align}
\frac{1}{N^{\theta(\bfi)}d^{|E|}}\sum^*_{\bfi}A_\cT \wh R^\cT_\bfi \wh U^\cT_\bfi
=\left(\frac{1}{N^{\theta(\bfi)}d^{|E|}}\sum^*_{\bfi}A_\cT\right) \wh R^\cT_\bfi \wh U^\cT_\bfi=c_\cT \wh R^\cT_\bfi \wh U^\cT_\bfi +\OO_{\prec}(\wh\Phi_q).
\end{align}
where the constant $c_\cT$ is from \eqref{e:sumA2}.

\begin{example}
Take $\cT=\{\bfi=\{i,j,k,\ell\}, E=\{(i,j), (j,k)\}\}$ to be a forest, which is a path of length two (See Figure \ref{fig:construct}) and a singleton. Then 
\begin{align*}
&\wh G^\cT_{xx}=m_d, \quad x\in \{i,j,k,\ell\},\quad \wh G^\cT_{x\ell}=0, \quad x\in \{i,j,k\},\\
&\wh G^\cT_{ij}=\wh G^\cT_{jk}= - \frac{\msc m_d  }{\sqrt{d-1}},\quad \wh G^\cT_{ik}=  \frac{\msc^2 m_d  }{d-1}.
\end{align*}
For example, if $R_\bfi=G_{ii}G_{ij}G_{ik}G_{\ell\ell}$, then $\wh R_\bfi=m_d^4(-\msc/\sqrt{d-1})^3$. The value of $\wh R_\bfi$ depends only on the forest $\cT$, but does not depend on the choice of the index set $\bfi\subset \qq{N}$.  If $R_\bfi=G_{ii}G_{ij}G_{ik}G_{j\ell}$, then $\wh R_\bfi=0$, because $j,\ell$ are in different connected components of $\cT$ and $\wh G^{\cT}_{j\ell}=0$. If $U_\bfi= (1+D_{ij}^{k\ell})(P^{q-1}\bar P^q)$, then $\wh U^{\cT}_\bfi= (1+D_{ij}^{k\ell})(P^{q-1}(m_d)\bar P^q(m_d))=P^{q-1}(m_d)\bar P^q(m_d)$. 
In the last equality, we used  that $D_{ij}^{k\ell}(P^{q-1}(m_d)\bar P^q(m_d))=0$.
Again, $\wh U^{\cT}_\bfi$ does not depend on the choice of the index set $\bfi\subset \qq{N}$.
\end{example}

We can view $\cT=(\bfi, E)$ as a subgraph of $\theta(\cT)$ copies of infinite $d$-regular trees. Then we can also do a simple switching as in Section \ref{sec:switchings}: for each edge $e=\{i_e,i'_e\}\in E$, we introduce a new infinite $d$-regular tree with an edge $\{j_e,j_e'\}$. At the end we will have $\theta(\cT)+|E|$ copies of infinite $d$-regular trees. We denote it as $\cD$ with adjacency matrix $D$. We can view this as an embedding of $\cT$ into $\theta(\cT)+|E|$ copies of infinite $d$-regular trees.
Since $\{i_e,i_e'\}$ and $\{j_e,j_e'\}$ belong to different infinite $d$-regular trees, after switching,  we still have two $d$-regular trees. If we switch all the pairs $\{i_e,i_e'\}$ and $\{j_e,j_e'\}$ for $e\in E$, we get the graph $\cT^+$ as constructed in Proposition \ref{c:intbp}. And it also gives us an embedding of $\cT^+$ into $\theta(\cT)+|E|$ copies of infinite $d$-regular trees. We denote this new graph as $\widetilde \cD$ with adjacency matrix $\widetilde D$, then $\widetilde D=D-\sum_{e\in E}\xi_{i_ei_e'}^{j_ej_e'}$.
See Figure \ref{fig:treeswitch}. 

%we can pick another edge $\{j_e,j_e'\}$ in $\cD$, such that $\chi_{i_ei_e'}^{j_ej_e'}(D)=1$. Moreover, by possibly changing $j_e$ and $j_e'$, we can assume that after switching the edges $\{i_e,i_e'\}$ and $\{j_e,j_e'\}$ we still get a $d$-regular tree. Namely $D-\xi_{i_ei_e'}^{j_ej_e'}$ is still an infinite $d$-regular tree. This is the case if in the three components graph $\cD\setminus\{\{i_e, i_e'\}, \{j_e,j_e'\}\}$, $i_e, j_e$ are in different connected components, and so are $i_e',j_e'$. 
%If we switch all the pairs $\{i_e,i_e'\}$ and $\{j_e,j_e'\}$ for $e\in E$, we get the graph $\cT^+$ as constructed in Proposition \ref{c:intbp}. And it also gives us an embedding of $\cT^+$ into $\theta(\cT)$ $d$-regular tree.
%See Figure \ref{fig:treeswitch}. 

\begin{figure}[t]
\begin{center}
\includegraphics[scale=0.38]{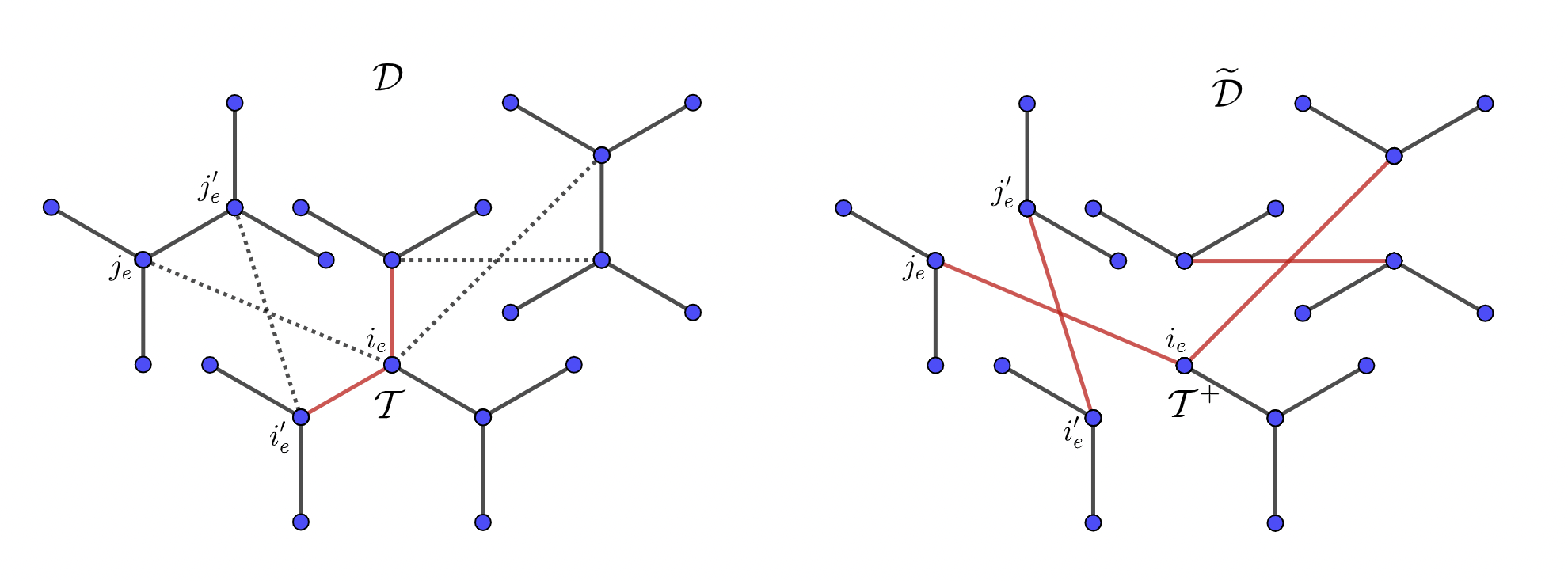}
\end{center}
\caption{We can view $\cT=(\bfi, E)$ as a subgraph of $\theta(\cT)+|E|$ copies of infinite $d$-regular trees $\cD$ on the left. After simple switching, we get an embedding of $\cT^+$ into $\theta(\cT)+|E|$ copies of infinite $d$-regular trees $\widetilde \cD$ on the right.
\label{fig:treeswitch}}
\end{figure}

Similarly to \eqref{e:EPP}, thanks to the relation \eqref{e:multiHG}, we have
\begin{align}\begin{split}\label{e:hatEPP}
(1+zm_d)P^{q-1}\bar P^q
&=  \frac{1}{N (d - 1)^{1/2}} \sum_{ij}^* A_{ij}\left(- \frac{\md \msc}{\sqrt{d - 1}}\right)P^{q-1}\bar P^q \\
&=  \frac{d}{(d-1)^{1/2}}\frac{1}{N^{\theta(\cT_0)} d^{|E_0|}} \sum_{ij}^* A_{\cT_0}\wh R^{\cT_0}_{\bfi_0}P^{q-1}\bar P^q,
\end{split}\end{align}
where $\cT_0=(\bfi_0={ij},E_0=\{i,j\})$ and $\wh R^{\cT_0}_{\bfi_0}=\wh G^{\cT_0}_{ij}=-m_d \msc/\sqrt{d-1}$.

\begin{proposition}\label{p:replaceR}
Propositions \ref{c:one-off}, \ref{c:case1exp}, \ref{c:case2exp}, and \ref{p:rewrite} hold if we replace $\Phi_q$ to $\wh\Phi_q$ and $R, U$ to $\wh R, \wh U$ as in \eqref{e:replace}. 
\end{proposition}

 Fix a forest $\cT=(\bfi, E)$. Let $\cT^+=(\bfi^+, E^+),\xi$ be as constructed in Propositions \ref{c:intbp} or \ref{c:intbp2}. 
 From the discussion above Figure \ref{fig:treeswitch}, we have two embeddings: the embedding of $\cT$  into $\theta(\cT^+)$ copies infinite $d$-regular trees $\cD$, and the embedding of $\cT^+$  into $\theta(\cT^+)$ copies infinite $d$-regular trees $\widetilde \cD$. The simple switchings from $\cT$ to $\cT^+$ also make $\cD$ into $\widetilde \cD$. We denote the adjacency matrices of $\cD$ and $\widetilde \cD$ as $D$ and $\widetilde D$ respectively.  The Green's function $\widehat G^\cT$ is defined on the vertex set $\bfi$ as $\widehat G^\cT|_{\bfi}=(D-z)^{-1}|_{\bfi}$. We can extend it to the vertex set $\bfi^+$ by $\widehat G^\cT|_{\bfi^+}:=(D-z)^{-1}|_{\bfi^+}$. And the Green's function $\widehat G^{\cT^+}$ is given by $\widehat G^{\cT^+}|_{\bfi^+}:=(\widetilde D-z)^{-1}|_{\bfi^+}$. The key to the proof of Proposition \ref{p:replaceR} is the following proposition, which states that $\widehat G^{\cT}|_{\bfi^+}$ and $\widehat G^{\cT^+}|_{\bfi^+}$ satisfy the same resolvent expansion as the Green's function $G$ of $d$-regular graphs.

\begin{proposition}\label{c:restrictR3}
Fix a forest $\cT=(\bfi, E)$. Let $\cT^+=(\bfi^+, E^+),\xi$ be as constructed in Propositions \ref{c:intbp} or \ref{c:intbp2}, then
\begin{align}\label{e:restrictR3}
\left.\wh{G}^{\cT^+}\right|_{\bfi^+}-\left.\wh{G}^{\cT}\right|_{\bfi^+}  =\left.\sum_{n\geq 1}\frac{1}{(d-1)^{n/2}}\wh G^{\cT^+}\left(-\xi \wh G^{\cT^+}\right)^n\right|_{\bfi^+}.
\end{align}
%Let $i\in \bfi$ and $\cT^+=(\bfi km\ell, E\cup\{i,m\}\cup\{k,\ell\})$ as constructed in Proposition \ref{c:intbp2}.  Let $\tilde \cT=(\bfi\bfj, E\cup\{i,k\}\cup\{m,\ell\})$, then
%\begin{align}\label{e:restrictR4}
%\wh{G}^{\cT^+}-\wh{G}^{\tilde \cT}  =\sum_{n\geq 1}\frac{1}{(d-1)^{n/2}}\wh G^{\cT^+}\left(-\sum_{e\in E}\xi_{ik}^{m\ell}\wh G^{\cT^+}\right)^n.
%\end{align}

\end{proposition}

\begin{proof}
We will only prove Proposition \ref{c:restrictR3} for the case that $\cT^+,\xi$ are as constructed in Propositions \ref{c:intbp}. In this case, $\cT^+=(\bfi\bfj, \cup_{e\in E}\{i_e, j_e\}\cup \{i'_e, j_e'\})$, and $\xi=\sum_{e\in E} \xi_{i_e i_e'}^{j_ej_e'}$. 
%We introduce an intermediate graph $\wt \cT=(\bfi \bfj, E\cup_{e\in E}\{j_e, j_e'\})$.  Then from the discussion$\wh G^{\cT}|_{\bfi}=\wh G^{\wt\cT}|_{\bfi}$. We extend $\wh G^{\cT}$ to the vertex set $\bfi\bfj$ as $\wh G^{\cT}|_{\bfi\bfj}:=\wh G^{\wt\cT}|_{\bfi\bfj}$. 
By viewing $\cT$ as a subgraph of $\theta(\cT)+|E|$ copies of infinite $d$-regular trees, from the discussion above Figure \ref{fig:treeswitch}, this also gives an embedding
of $\cT^+$ into $\theta(\cT)+|E|$ copies of infinite $d$-regular tree $\wt \cD$, and we denote the adjacency matrix of $\cD$ and $\wt  \cD$ as $D$ and $\wt D$ respectively. Then by our construction $D=\wt D+\xi$, 
\begin{align*}
\wh{G}^{\cT}|_{\bfi \cup\bfj}=\left.\left(D-z\right)^{-1}\right|_{\bfi\cup \bfj}=\left.\left(\wt D+\xi-z\right)^{-1}\right|_{\bfi\cup \bfj},\quad \wh{G}^{\cT^+}|_{\bfi \cup\bfj}=\left.(\wt D-z)^{-1}\right|_{\bfi\cup\bfj}.
\end{align*}  
By the resolvent identity, we have 
\begin{align*}\begin{split}
\left.\wh{G}^{\cT^+}\right|_{\bfi\cup\bfj}-\left.\wh{G}^{ \cT} \right|_{\bfi\cup\bfj}
& = \left.(\wt D-z)^{-1}-\left(\wt D+\xi-z\right)^{-1}\right|_{\bfi\cup\bfj}\\
&=\left.\sum_{n\geq 1}\frac{1}{(d-1)^{n/2}}(\wt D-z)^{-1}\left(-\xi(\wt D-z)^{-1}\right)^n\right|_{\bfi\cup\bfj}\\
&=\left.\sum_{n\geq 1}\frac{1}{(d-1)^{n/2}}\wh G^{\cT^+}\left(-\xi\wh G^{\cT^+}\right)^n\right|_{\bfi\cup\bfj}.
\end{split}\end{align*}
\end{proof}

\begin{proof}[Proof of Proposition \ref{p:replaceR}]

For Proposition \ref{c:one-off}, if $R_\bfi$ contains at least one off-diagonal Green's function entry, $G_{ij}$, and $i,j\in \bfi$ are in different connected components of $\cT$, then by our construction $\wh G^{\cT}_{ij}=0=\wh R^{\cT}_{\bfi}$, and
\begin{align}\label{e:hatAGij}
\frac{1}{N^{\theta(\cT)}d^{|E|}}\sum^*_{\bfi}A_{\cT}\wh R^{\cT}_{\bfi}\wh U_\bfi^{\cT}=0,
\end{align}
This gives that Proposition \ref{c:one-off} holds if we replace $\Phi_q$ to $\wh\Phi_q$ and $R, U$ to $\wh R, \wh U$ as in \eqref{e:replace}.

In Proposition \ref{c:case1exp}, given a forest $\cT=(\bfi, E)$, we construct $\cT^+=(\bfi\bfj, E^+)$ with $E^+=\cup_{e\in E}\{i_e, j_e\}\cup\{i_e', j_e'\}$, and prove 
\begin{align}\begin{split} \label{e:rec2c}
\frac{1}{N^{\theta(\cT)}d^{|E|}}\sum^*_{\bfi}\bE\left[A_{\cT}R_{\bfi} U_\bfi\right]=
\frac{1}{N^{\theta(\cT^+)}d^{|E^+|}}\sum^*_{\bfi^+}\bE\left[A_{\cT^+}D_{\cT^+} R_{\bfi} (1+D_{\cT^+})U_\bfi)\right]+\OO_\prec\left(\Phi_q\right).
\end{split}\end{align}
Let $\xi=\sum_{e\in E}{\xi_{i_e i_e'}^{j_e j_e'}}$. Then by the Taylor expansion \eqref{e:D-product}, 
\begin{align}\label{e:DcT+G}
D_{\cT^+} G=\sum_{n\geq 1}\frac{1}{(d-1)^{n/2}} G\left(-\xi  G\right)^n,
\end{align}
and we further expand  $D_{\cT^+} R_{\bfi}$ in \eqref{e:rec2c} as a linear combination of monomials of Green's function entries.

%We denote the intermediate graph $\wt \cT=(\bfi \bfj, E\cup_{e\in E}\{j_e, j_e'\})$. 
Proposition \ref{c:restrictR3} leads to 
\begin{align*}\begin{split}
\wh R^{\cT}_{\bfi } \wh U_\bfi^{\cT}
&=\wh R_{\bfi }^{\cT } \wh U_\bfi^{\cT}
=\left(\wh R_{\bfi }^{ \cT^+}+(\wh R_{\bfi }^{ \cT } -\wh R_{\bfi }^{ \cT^+})\right) \wh U_\bfi^{\cT}=(\wh R_{\bfi }^{ \cT } -\wh R_{\bfi }^{ \cT^+}) \wh U_\bfi^{\cT}
=(\wh R_{\bfi }^{ \cT } -\wh R_{\bfi }^{ \cT^+}) (1+D_{\cT^+})\wh U_\bfi^{\cT},
\end{split}\end{align*}
where in the third equality we used that $R_{\bfi }$ contains an off-diagonal term, say $G_{ij}$. Then $i,j$ are in different connected components of $\cT^+$, and $\wh R_{\bfi }^{ \cT^+}=0$.
In the last equality, we used that $D_{\cT^+}\wh U_\bfi^\cT=0$.
 The expressions \eqref{e:restrictR3} and \eqref{e:DcT+G} are the same if we replace $G$ by $\wh G^{\cT^+}$. Thus,  up to negligible error, $(\wh R_{\bfi }^{ \cT } -\wh R_{\bfi }^{ \cT^+})$ can be obtained from $D_{\cT^+} R_{\bfi}$ by replacing $G$ by $\wh G^{\cT^+}$. We conclude that $N^{-\theta(\cT)}d^{-|E|}\sum^*_{\bfi} A_{\cT}\wh R^{\cT}_{\bfi} \wh U^{\cT}_\bfi=c_\cT \wh R^{\cT}_{\bfi} \wh U^{\cT}_\bfi+\OO_\prec(\wh \Phi_q)$ is a linear combination of terms in the following form
\begin{align*}
c_\cT \wh R^{\cT^+}_{\bfi^+} \wh U^{\cT^+}_{\bfi^+}=c_{\cT^+} \wh R^{\cT^+}_{\bfi^+} \wh U^{\cT^+}_{\bfi^+}=\frac{1}{N^{\theta(\cT^+)}d^{|E^+|}}\sum^*_{\bfi^+}A_{\cT^+}\wh R^{\cT^+}_{\bfi} \wh U_{\bfi^+}^{\cT^+}+\OO_\prec\left(\wh\Phi_q\right),
\end{align*}
with the same coefficients as in \eqref{e:nextod1}. Here $\wh U_{\bfi^+}^{\cT^+}=(1+D_{\cT^+})\wh U^\cT_\bfi$, and we used $c_\cT =c_{\cT^+}$ from Remark \ref{r:changec}. This gives that Proposition \ref{c:case1exp} holds if we replace $\Phi_q$ to $\wh\Phi_q$ and $R, U$ to $\wh R, \wh U$ as in \eqref{e:replace}.

The proof that Proposition \ref{c:case2exp} holds if we replace $\Phi_q$ to $\wh\Phi_q$ and $R, U$ to $\wh R, \wh U$ as in \eqref{e:replace} is analogues to that of Proposition \ref{c:case1exp}. So we omit the proof.

For Proposition \ref{p:rewrite}, the key estimate is the relation \eqref{e:induction}. We show it holds if we replace $\Phi_q$ to $\wh\Phi_q$ and $R, U$ to $\wh R, \wh U$
\begin{align*}
\frac{1}{N^{\theta(\cT^+)}d^{|E^+|}}\sum^*_{\bfi^+}A_{\cT^+} m_d^\fd \wh U^{\cT^+}_{\bfi^+}
=\frac{C}{N^{\theta(\cT)}d^{|E |}}\sum^*_{\bfi }A_{\cT } m_d^\fd \wh U^{\cT}_\bfi+\OO_\prec(\wh\Phi_q),
\end{align*}
We will only discuss the case that $\cT^+$ is constructed from $\cT$ as in Proposition \ref{c:case1exp}. In this case, $C=1$, $\xi=\sum_{e\in E}\xi_{i_e i'_e}^{j_e j_e'}$ and $\wh U^{\cT^+}_{\bfi^+}=(1+D_{\xi})\wh U^{\cT}_\bfi=\wh U^{\cT}_\bfi$. We have
\begin{align*}
&\frac{1}{N^{\theta(\cT^+)}d^{|E^+|}}\sum^*_{\bfi^+}A_{\cT^+} m_d^\fd \wh U^{\cT^+}_{\bfi^+}=c_{\cT^+}m_d^\fd \wh U^{\cT}_{\bfi}+\OO_\prec(\wh\Phi_q)=c_{\cT}m_d^\fd \wh U^{\cT}_{\bfi}+\OO_\prec(\wh\Phi_q)\\
&=\frac{1}{N^{\theta(\cT)}d^{|E |}}\sum^*_{\bfi }A_{\cT } m_d^\fd \wh U^{\cT}_\bfi+\OO_\prec(\wh\Phi_q),
\end{align*}
where we used $c_\cT=c_{\cT^+}$ from Remark \ref{r:changec}. This gives that Proposition \ref{p:rewrite} holds if we replace $\Phi_q$ to $\wh\Phi_q$ and $R, U$ to $\wh R, \wh U$ as in \eqref{e:replace}. 
\end{proof}

\begin{proof}[Proof of Proposition \ref{p:idfP}]
Using Proposition \ref{p:replaceR} as input, by the same argument as for Proposition \ref{p:DSE}, we will have
\begin{align*}
|P(z,m_d)|^{2q}=\OO_{\prec}\left(\wh \Phi_q\right),\quad \wh\Phi_q= \frac{d}{N}|P(z,m_d)|^{2q-1},
\end{align*}
and Proposition \ref{p:idfP} follows. 
\end{proof}

\subsection{Proof of Theorem \ref{t:rigidity} }
\label{s:proofrigidity}

In this section we prove Theorem \ref{t:rigidity} by analyzing the high order moment estimates of $P(z, m (z))$ from Proposition \ref{p:DSE}. We prove Theorem \ref{t:rigidity} for $z$ close to the right edge $E=2$. The statement about the left edge can be proven in the same way. We construct the shifted spectral domain 
\begin{align}\label{e:recalldefD}
  \wt{\mathbf  D} \deq \{w=\kappa+\ri \eta:  -3/2\leq \kappa\leq \fK, 0<\eta\leq 1, N\eta\sqrt{|\kappa|+\eta}\geq N^\fa\}.
\end{align}
 The following stability Proposition is from \cite[Proposition 2.11]{huang2020transition}.

\begin{proposition}\label{p:stable}
Let $E=2$. Suppose that $\delta:\wt{\mathbf  D}\rightarrow \mathbb{R}$ is a function so that,
\begin{align*}
|P(E+w, m (E+w))|\leq \delta(w),\quad w\in  \wt{\mathbf  D} .
\end{align*}
Suppose that $N^{-2}\leq \delta(w)\ll1$ for $w\in  \wt{\mathbf  D} $, that $\delta$ is Lipschitz continuous with Lipschitz constant $N$ and moreover that for each fixed $\kappa$ the function $\eta\mapsto \delta(\kappa+\ri\eta)$ is nonincreasing for $\eta>0$.   Then,
\begin{align*}
|m (E+ w)-  m_d(E+ w)|=\O\left(\frac{\delta(w)}{\sqrt{|\kappa|+\eta+\delta(w)}}\right),
\end{align*}
where the implicit constant is independent of $N$.
\end{proposition}

\begin{proof}[Proof of Theorem \ref{t:rigidity}]
  Let  $E=2$ and $z=E+w$ with $w\in \wt{\textbf D}$. Since the Kesten-Mckay distribution has square root behavior, 
\begin{align*}
\Im[ m_d(  z)]\asymp \Phi(w)\deq \left\{
\begin{array}{cc}
\sqrt{|\kappa|+\eta}, & \kappa\leq 0,\\
\eta/\sqrt{|\kappa|+\eta}, & \kappa\geq 0.
\end{array}
\right.
\end{align*}
and we notice that $w\in \wt{\mathbf  D}$, we have that 
\begin{align*}
\frac{d}{N}\leq |\kappa|+\eta , \quad \frac{d^{3/2}}{N}\frac{1}{\sqrt{N\eta}}\leq |\kappa|+\eta,\quad \frac{d^{3/2}\Lambda_o}{N}\leq |\kappa|+\eta.
\end{align*}
So \eqref{e:derPa} gives
\begin{align*}
|\del_2 P(  z,    m_d(  z))|\asymp \sqrt{|\kappa|+\eta}.
\end{align*}
%We roughly have the estimate 
%\begin{align}\begin{split}
%\bE[|P(  z)|^{2r}]
%&\lesssim \bE\left[\frac{1}{q^3} \frac{\Im[m_N(  z)]}{N\eta}|\del_2 P(  z)||P(  z)|^{2r-2}\right]\\
%&+\bE\left[ \frac{\Im[m_N(  z)]}{N\eta}|P(  z)|^{2r-1}\right]+\bE\left[ \frac{\Im[m_N(  z)]}{(N\eta)^2}|\del_2 P(  z)||P(  z)|^{2r-2}\right].
%\end{split}\end{align}
Then we have
\begin{align*}
\Im[m(  z)]\lesssim \Phi(w)+|m(z)- m_d(z)|,
\end{align*}
and 
\begin{align*}
\del_2 P(  z, m(  z))
=\del_2 P(  z, m_d(  z))+\O(|m(  z)-m_d(  z)|)
=\O(\sqrt{|\kappa|+\eta}+|m(  z)-m_d(  z)|).
\end{align*}
The same argument as in \cite[Proposition 8.6]{bauerschmidt2020edge}, we can first show that $|m(z)- m_d( z)|\prec \sqrt{|\kappa|+\eta}$. We assume that there exists some deterministic control parameter $\Lambda(w)$ such that the prior estimate holds
\begin{align*}
|m(z)-   m_d(z)|\prec \Lambda(w)\lesssim \sqrt{|\kappa|+\eta}.
\end{align*}
%
%and Markov's inequality, we get
%\begin{align}\begin{split}\label{e:estimateP}
%|P(  z, m_N(  z))|
%&\prec \sqrt{\left(\frac{1}{q^3}+\frac{1}{N\eta}\right) \frac{\Phi(w)+\Lambda(w)}{N\eta}(\sqrt{|\kappa|+\eta}+\Lambda(w))}+\frac{\Phi(w)+\Lambda(w)}{N\eta}.
%\end{split}
%\end{align}
Since $\Phi(w)\gtrsim\sqrt{|\kappa|+\eta}$ and $\Lambda(w)\prec \sqrt{|\kappa|+\eta}$, \eqref{e:defP}  combining with the Markov's inequality leads to
\begin{align}\begin{split}\label{e:estimateP}
|P(  z,  m(  z))|
&\prec \frac{d^{3/2}\Lambda_o}{N}+\frac{1}{N\eta}\left((\Lambda(w)+\Phi(w))\sqrt{|\kappa|+\eta}\right)^{1/2}\\
&\prec \frac{d}{N}+\frac{1}{N\sqrt{\eta}}+\frac{1}{N\eta}\left((\Lambda(w)+\Phi(w))\sqrt{|\kappa|+\eta}\right)^{1/2},
\end{split}\end{align}
where in the last line we used $\Lambda_o\prec 1/\sqrt{d}+1/\sqrt{N\eta}$ and $d\ll N^{1/3}$.

If $\kappa\geq 0$, then $\Phi(w)=\eta/ \sqrt{|\kappa|+\eta}$, and \eqref{e:estimateP} simplifies to
\begin{align}\label{e:boundoutside}
|P(  z,  m(  z))|
\prec \frac{d}{N}+\frac{1}{N\eta^{1/2}}+\frac{(|\kappa|+\eta)^{1/4}\Lambda(w)^{1/2}}{N\eta}.
\end{align}
Thanks to Proposition \ref{p:stable}, by taking $\delta(w)$ the righthand side of \eqref{e:boundoutside} times $N^{\varepsilon}$ with arbitrarily small $\varepsilon$, we have
\begin{align}\begin{split}\label{e:outS0}
| m(  z)-m_d(  z)|\prec \frac{1}{\sqrt{|\kappa|+\eta}}\left(\frac{d}{N}+\frac{1}{N\eta^{1/2}}+\frac{(|\kappa|+\eta)^{1/4}\Lambda(w)^{1/2}}{N\eta}\right).
\end{split}\end{align}
%On the region $|\kappa|+\eta\gg 1/q^3N^{1/2}$, $N\eta\sqrt{|\kappa|+\eta}\gg1$ and $N\eta(|\kappa|+\eta)q^3\gg1$, the coefficients of $\Lambda^{1/2}$ is less than one. 
By iterating \eqref{e:outS0}, we get
\begin{align}\label{e:outS}
| m(  z)-m_d(  z)|\prec \frac{1}{\sqrt{|\kappa|+\eta}}\left(\frac{d}{N}+\frac{1}{N\eta^{1/2}}+\frac{1}{(N\eta)^2}\right).
\end{align}

If $\kappa\leq 0$, then $\Phi(w)=\sqrt{|\kappa|+\eta}$ and $\Lambda(w)\prec \sqrt{|\kappa|+\eta}$, \eqref{e:estimateP} simplifies to
\begin{align}\label{e:boundinside}
|P(  z, m(  z)|
\prec \frac{d}{N}+\frac{(|\kappa|+\eta)^{1/2}}{N\eta}.
\end{align}
It follows from Proposition \ref{p:stable}, by taking $\delta(z)$ the righthand side of \eqref{e:boundinside} times $N^{\varepsilon}$ with arbitrarily small $\varepsilon$, we have
\begin{align}\label{e:inS}
| m(  z)-m_d(  z)|\prec \frac{1}{N\eta}+\frac{d}{N\sqrt{|\kappa|+\eta}}.
\end{align}
The claim \eqref{e:largeeig} follows from the estimates of the Stieltjes transform \eqref{e:outS} and \eqref{e:inS}, see \cite[Section 11]{erdHos2017dynamical}.

\end{proof}

\section{Edge Universality} \label{sec:universality}

In this section we prove the edge universality of random $d$-regular graphs in the regime $N^\fc\leq d\leq N^{1/3-\fc}$ where Theorem \ref{t:rigidity} provides optimal bounds on the extremal eigenvalues. 

Our strategy is based on the now standard three-step approach of random matrix theory \cite{MR3699468}. Starting from the (rescaled) adjacency matrix $H = H(0)$ from \eqref{def_H}, we run a matrix-valued Brownian motion $H(t)$. Using the rigidity estimates from Theorem \ref{t:rigidity}  combined with \cite{landon2017edge}, we deduce that for $t \gg N^{-1/3}$ the matrix $H(t)$ has GOE edge statistics; see Proposition \ref{t:universalityHt} below. The main work in this section is a comparison argument to show that $H(0)$ and $H(t)$ with $t \gg N^{-1/3}$, have the same edge statistics; see Proposition \ref{thm:comp} below. 

%{\cob This comparison argument has the same spirit as the one from \cite{MR3729611}. Its underlying principle is that for large $d$, a Markovian switching dynamics of the graph that leaves the random regular graph invariant is well approximated by Brownian motion, when considering observables that characterize the local spectral statistics. The main difference to \cite{MR3729611} is that, since we are working at the edge, we need to incorporate precise rigidity estimates on the locations of the eigenvalues near the spectral edge, and the necessary cancellations are more delicate.}

We recall the constrained GOE $W$ as introduced in \cite[Section 2.1]{MR3729611}. It may be viewed as the usual Gaussian Orthogonal Ensemble restricted to matrices with vanishing row and column sums.  Formally, $W$  is the centered Gaussian process on the space $\cal M \deq \h{H \in \bR^{N \times N} \col H = H^*, H \bld 1 = 0}$ with covariance $\bE \scalar{W}{X} \scalar{W}{Y} = \scalar{X}{Y}$, where $\scalar{X}{Y} \deq \frac{N}{2} \Tr (XY)$ for $X,Y \in \cal M$. Explicitly, its covariance is given by
\begin{equation} \label{W_ibp}
\bE [W_{ij} W_{k\ell}] = \frac{1}{N} \pbb{\delta_{ik} - \frac{1}{N}} \pbb{\delta_{jl} - \frac{1}{N}}
+ \frac{1}{N} \pbb{\delta_{il} - \frac{1}{N}} \pbb{\delta_{jk} - \frac{1}{N}}.
\end{equation}
The following result is a straightforward consequence of \eqref{W_ibp}, and Gaussian integration by part.

\begin{lemma}\label{p:intbypart}
For the constrained GOE $W$ we have the integration by parts formula
\begin{align}\label{e:intbypart}
\bE[W_{ij}F(W)]=\frac{1}{N^3}\sum_{k\ell}\bE[\del_{ij}^{k\ell} F(W)].
\end{align}
\end{lemma}

Next, we define the matrix-valued process
\begin{align}\label{e:Ht}
H(t) \deq e^{-t/2}H+\sqrt{1-e^{-t}} \, W,
\end{align}
where $H$ was defined in \eqref{def_H} in terms of the adjacency matrix of the $d$-regular graph. Thus, $H(0)=H$ and $H(\infty) = W$. The matrix $H(t)$ has a trivial eigenvalue $\la_1(t)=e^{-t/2}d/\sqrt{d-1}$ with eigenvector $\bf1$. We denote the remaining eigenvalues of $H(t)$ by $\la_2(t)\geq \la_{3}(t)\geq \cdots \la_{N-1}(t)\geq \la_N(t)$, and corresponding normalized eigenvectors $\bmu_2(t), \bmu_3(t),\cdots, \bmu_N(t)$.

We recall the projection matrix  $P_{\perp}=I-{\bf 1}{\bf 1^*}/N$ from Section \ref{sec:notation}. Then the matrix $H(t)$ and $P_\perp$ commute, i.e.\ $H(t)P_\perp=P_\perp H(t)$.
For $z \in \bC_+ =\{z\in \bC\col \Im[z]>0\}$, we define the \emph{time-dependent Green's function} by
\begin{equation*}
  G(z;t) \deq P_\perp (H(t)-z)^{-1}P_{\perp}=\sum_{i=2}^N \frac{\bmu_i(t)\bmu_i(t)^\top}{\la_i(t)-z},
\end{equation*}
so that $G(z;t)$ and $(H(t)-z)^{-1}$ agree on the image of $P_\perp$, i.e.\ the subspace of $\bR^N$ perpendicular to $\bm 1$ which carries the nontrivial spectrum of $H(t)$. We denote the Stieltjes transform of the empirical eigenvalue distribution of $P_\perp H(t) P_\perp$ by $m(z;t)$,
\begin{align} \label{def_mtz}
m(z;t)
=\frac{1}{N}\Tr G(z;t)=\frac{1}{N}\sum_{i=2}^N\frac{1}{\la_i(t)-z}.
\end{align}

We have the following corollary of Theorems \ref{t:rigidity}.
\begin{corollary}\label{c:rigidity3}
Fix a constant $\fc>0$ and suppose that $N^{\fc}\leq d\leq N^{1/3-\fc}$. Then we have with overwhelming probability
\begin{align*}
|\la_2(0)-2|, |\la_N(0)+2|\prec N^{-2/3},
\end{align*}
and uniformly for $w\in \wt{\mathbf D}$ (recall from \eqref{e:recalldefD}),  $z=2+w$ or $z=-2-\bar w$, 
\begin{align*}
|m(0;z)-\md(z)|\prec \left( \frac{1}{N\eta}+\frac{d}{N\sqrt{\eta}}\right).
\end{align*}
\end{corollary}

%\begin{remark}
%\cob As an easy consequence of Corollary \ref{c:rigidity3}, we have that for any $z=E+\ri\eta$, with $-4\leq E\leq 4$ and $\eta\geq N^{-2/3+\fc}$, $\Im[m(0;z)]\asymp \Im[\md(z)]$.
%\end{remark}

\subsection{Free convolution}\label{s:fc}
The asymptotic eigenvalue density of the matrix $H(t)$ is governed by the free additive convolution of the rescaled Kesten--McKay measure with the semicircle law at time $s = 1 - e^{-t}$. We recall some properties of measures obtained by the free convolution with a semicircle distribution from \cite{MR1488333}. The semicircle density $\rhosc(x)$ is given by \eqref{def_sc}, and the semicircle density at time $s$ is $s^{-1/2}\rhosc(s^{-1/2}x)$. Given a probability measure $\mu$ on $\bR$, we denote its free convolution with a semicircle distribution of variance $s$ by $\mu_s$. The Stieltjes transforms of $\mu$ and $\mu_s$ are given by $G_\mu(z)=\int \frac{\rd \mu(x)}{x-z}$ and $G_{\mu_s}(z)=\int \frac{\rd \mu_{s}(x)}{x-z}$ respectively. Then the following holds \cite{MR1488333}.
\begin{enumerate}
\item We denote the set $U_s = \{z\in\bC_+: \int \frac{\rd \mu(x)}{|z-x|^2}<\frac{1}{s}\}$. Then $z \mapsto z-s G_\mu(z)$ is a homeomorphism from $\bar{U}_s$ to $\bC_+\cup \bR$ and conformal from $U_s$ to $\bC_+$. We denote its inverse by $F_{\mu_s}: \bC_+\cup \bR\mapsto \bar U_s$.
\item The Stieltjes transform of $\mu_s$ is characterized by $G_\mu(z)=G_{\mu_s}(z-sG_\mu(z))$, for any $z\in U_s$.
\end{enumerate} 

%By the inversion formula for the Stieltjes transform, we deduce from (ii) that the density of $\mu_s$ is given by $\rd \mu_s(x) / \rd x= \Im [G_\mu(F_{\mu_s}(x))]/\pi$.

The asymptotic eigenvalue density of $W$ is the semicircle density $\rhosc(x)$ and the asymptotic eigenvalue density of $\sqrt{1 - e^{-t}} \, W$ is the semicircle density at time $s = 1 - e^{-t}$.
The asymptotic eigenvalue density of $H(t)$ is the free convolution of rescaled Kesten--McKay law $\mu(\rd x) = e^{t/2} \rho_d(e^{t/2} x) \rd x$ at  time $e^{-t}$ and the semicircle density at time $1 - e^{-t}$. We denote its density by $\rho_d(x;t)$ and its Stieltjes transform by $\md(z;t) = \int \frac{\rho_d(x;t)}{x - z} \rd x$.
 Since $G_\mu(e^{-t/2} z) = e^{t/2} m_d(z)$,  we deduce from (i) and (ii) above that
\begin{align}\label{e:selffc}
\md(\xid(z;t);t)=e^{t/2}\md(z),
\end{align}
where
\begin{align} \label{e:defxid}
\xid(z;t)\deq e^{-t/2}z-e^{t/2}(1-e^{-t})\md(z)
=e^{-t/2}z-(1-e^{-t})\md(\xid(z;t);t)
\end{align}
is a homeomorphism from the set $\{z\in \bC_+: \int \frac{\rho_d(x)}{|x-z|^2}\rd x\leq\frac{1}{e^t-1}\}$ to $\bC_+\cup \bR$.  To find the support of the measure $\rho_d(t;x)$, we notice that there exists $z_d^{\pm}(t)\in \bR$ such that $\{z\in \bC_+: \int \frac{\rho_d(x)}{|x-z|^2}\rd x=\frac{1}{e^t-1}\}$ consists of the intervals $(-\infty, z_d^-(t)]\cup [z_d^+(t), \infty)$ and an arc $\cC$ from $z_d^-(t)$ to $z_d^+(t)$. Those two endpoints $z_{d}^\pm (t)\in\bR$ are the largest and smallest real solutions to
\begin{align}\label{e:edgeeqn1}
\int \frac{\rho_d(x)}{|x-z|^2}\rd x=\frac{1}{e^t-1}\quad \Leftrightarrow \quad  m_d'(z)=\frac{1}{e^t-1}.
\end{align}
As a consequence,  the right and left edges of the measure $\rho_d(x;t)$ are given by
\begin{align}\label{e:edgeeqn2}
E_{d}^\pm(t)=\xi_d(t;z_d^\pm(t))=e^{-t/2}z_{d}^\pm(t)-e^{t/2}(1-e^{-t})\md(z_{d}^\pm(t)).
\end{align}

 Recall from Proposition \ref{p:minfty}, $  m_d(z)$ satisfies the functional equation 
\begin{align}\label{e:m0deq} 
    1+z  m_d(z)+  m^2_d(z)+Q(  m_d(z))=\OO(d/N),
\end{align}

The next proposition states that $m_d(z;t)$ satisfies a similar equation, and we also characterize the change of the edges along time $\del_t E^\pm_d(t)$. 
\begin{proposition}\label{p:mtEQ}
Adapt the assumptions in Theorem \ref{t:rigidity}, and recall $m_d(z;t)$ from \eqref{e:selffc}.It satisfies the following equation
\begin{align}\label{e:mdteq}
    1+z  m_d(z;t)+  m^2_d(z;t)+Q(e^{-t/2}  m_d(z;t))=\OO(d/N).
\end{align}
And the spectral edge $E^{\pm}_d(t)$ satisfies
\begin{align}\label{e:deltE}
\del_t E^{\pm}_d(t)=\frac{e^{-t/2}}{2} Q'(e^{-t/2}  m_d(E_d^\pm;t))+\OO\left(\sqrt{\frac{d}{N}}\right).
\end{align}
\end{proposition}
\begin{proof}
By taking $z$ to be $\xi_d(z;t)$ in \eqref{e:m0deq}, and using the relation \eqref{e:selffc}
\begin{align}\label{e:dppmt}
1+ze^{-t/2}   m_d(\xi_d(z;t);t) +e^{-t}  m^2_d(\xi_d(z;t);t)+Q(e^{-t/2}  m_d(\xi_d(z;t);t))=\OO(d/N).
\end{align}
From the definition \eqref{e:defxid} of $\xi_d(z;t)$, we have
\begin{align*}
   \xi_d(z;t) +(1-e^{-t})  m_d(\xi_d(z;t);t)=ze^{-t/2} ,
\end{align*}
and \eqref{e:dppmt} simplifies to 
\begin{align*}
    1+z   m_d(z;t) +  m^2_d(z;t)+Q(e^{-t/2}  m_d(z;t))=\OO(d/N),
\end{align*}
which is \eqref{e:mdteq}.

By taking the time derivative of both sides of \eqref{e:edgeeqn2} and use the relation \eqref{e:edgeeqn1}, we get
\begin{align}\label{e:edgeeqn3}
\del_t E_{d}^\pm(t)=-\frac{1}{2}\left(e^{-t/2}z_{d}^\pm(t)+(e^{t/2}+e^{-t/2})\md(z_{d}^\pm(t))\right)
=-\frac{1}{2}\left(E_{d}^\pm(t)+2m_d(E_{d}^\pm(t);t)\right),
\end{align}
where we used \eqref{e:edgeeqn2} and $e^{t/2}m_d(z_d^{\pm}(t))=m_d(E_d^{\pm}(t);t)$ for the last equality.
Since $m_d(z;t)$ has square root behavior at the edges $E_{d}^\pm(t)$. 
Let $z=E_d^\pm(t)+w$ with $w\in \wt{\mathbf D}$ (recall from \eqref{e:recalldefD}), then
 $|m_d(z;t)-m_d(E^\pm_{d}(t);t)|\asymp \sqrt{|w|}$.
\begin{align*}
&\OO(d/N)=\left.\left(1+u  m_d(u;t)+  m^2_d(u;t)+Q(e^{-t/2}  m_d(u;t))\right)\right|_{u=E^\pm_d(t)}^{u=z}\\
&=(m_d(z,t)-m_d(E^\pm_d(t),t))\left(E^\pm_d(t)+2m_d(E^\pm_d(t),t)+e^{-t/2} Q'(e^{-t/2}  m_d(E^\pm_{d}(t);t))\right)\\
&+
\OO(|w|+|m_d(z,t)-m_d(E^\pm_d(t),t)|^2).
\end{align*}
It follows by rearranging, 
\begin{align}\label{e:Etexp}
E^\pm_d(t)+2m_d(E^\pm_d(t),t))+e^{-t/2} Q'(e^{-t/2}  m_d(E^\pm_{d}(t);t)
=\OO\left(\frac{d}{N \sqrt{|w|}}+\sqrt{|w|}\right).
\end{align}
We can minimize the righthand side by taking $|w|\asymp d/N$, then the righthand side becomes $\OO(\sqrt{d/N})$. The claim \eqref{e:deltE} follows by plugging \eqref{e:Etexp} into \eqref{e:edgeeqn3}.
\end{proof}

\subsection{Rigidity and edge universality of $H(t)$}

In this section we collect some estimates on the Green's function $G(z;t)$ and the Stieltjes transform $m(z;t)$ of $H(t)$,
and state the edge universality result for $H(t)$ when $t \gg N^{-1/3}$.
All statements and estimates in this section follow directly from \cite{MR3690289,landon2017edge, adhikari2020dyson}.

For sufficiently regular initial data, it has been proven in \cite{landon2017edge}, after short time the eigenvalue statistics at the spectral edge of \eqref{e:Ht} agree with GOE. A modified version of this theorem was proven in \cite{adhikari2020dyson}, which assumes that the initial data is sufficiently close to a nice profile. To use these results, we need to restrict  $H(0)$ to a subset, on which the optimal rigidity holds.
We denote $\cA$ to be the set of sparse random matrices $H$, such that \eqref{e:outS} and \eqref{e:inS} hold at edges $\pm   2$:
\begin{align}\label{e:defA}
\cA\deq \{H: \text{\eqref{e:mbond} holds.} \}.
\end{align}
By Theorem \ref{t:rigidity}, we know that the event holds with overwhelming probability.

First, using the rigidity estimates of Corollary~\ref{c:rigidity3} as input,
the rigidity estimates on the Stieltjes transform $m(z;t)$ of $H(t)$ follow from \cite{landon2017edge, adhikari2020dyson}.

\begin{proposition}\label{p:rigidity2}
Fix a constant $\fc>0$ and $N^{\fc}\leq d\leq N^{1/3-\fc}$. For any time $0\leq t\ll 1$, let $L_t=E_d^+(t)$ from \eqref{e:edgeeqn2}, with overwhelming probability we have
\begin{align*}
|\la_2(t)-L_t|, |\la_N(t)+L_t|\prec N^{-2/3},
\end{align*}
and uniformly for any $w=\kappa+\ri \eta\in \wt{\mathbf D}$ (recall from \eqref{e:recalldefD}), and $z=L_t+w$ or $-L_t-\bar w$,
\begin{align}\label{e:mtmdt}
|m(z;t)-\md(z;t)|\prec \left( \frac{1}{N\eta}+\frac{d}{N\sqrt{\eta}}\right),
\end{align}
where $m(z;t)$ from \eqref{def_mtz} is the Stieltjes tranform and $\md(z;t)$ is defined by \eqref{e:selffc}.
\end{proposition}

Using the rigidity estimates from Corollary~\ref{c:rigidity3} and the estimates on the Green's function entries of $H(0)$
from Theorem~\ref{thm:rigidity} as input,
the entrywise estimates on Green's function of $H(t)$ with $t>0$ follow from an argument similar to the proof of \cite[Theorem 2.1]{MR3690289}.
%We remark that in the statement of \cite[Theorem 2.1]{MR3690289}, it assumed that $\Im[m_0]$ is bounded from below and that $t\gg \eta_*$.
%However, in the proof of \cite[Theorem 2.3]{MR3690289}, these assumptions are only used to show that $|m_t(z)-m_{{\rm fc}, t}(z)|$ is small.
%With the required estimate of $|m_t(z)-m_{{\rm fc}, t}(z)|$ already established by \eqref{e:mtmdt},
%the remaining part of \cite[Theorem 2.1]{MR3690289} does not use that $\Im[m_0]$ is bounded from below or that $t\gg \eta_*$.
%Therefore, with \eqref{e:mtmdt} given,
%the remaining proof of \cite[Theorem 2.1]{MR3690289} applies and gives the following result on the entrywise estimates of Green's function of $H(t)$.

\begin{proposition}\label{p:entry-wiselaw}
Fix constant $\fc>0$ and suppose that $N^{\fc}\leq d\leq N^{1/3-\fc}$. For any time $0\leq t\ll 1$,  let $L_t=E_d^+(t)$ from \eqref{e:edgeeqn2}, 
with overwhelming probability we have
\begin{align*}
|G_{ij}(z;t)-\delta_{ij}\md(z;t)|\prec \frac{1}{d^{1/2}}+\frac{1}{\sqrt{N\eta}},
\end{align*}
uniformly for any $w=\kappa+\ri \eta\in \wt{\mathbf D}$, and $z=L_t+w$ or $-L_t-\bar w$.
\end{proposition}

Next we prove the following theorem. It states that for time $t\gg N^{-1/3}$ the fluctuations of extreme eigenvalues of $H(t)$ conditioning on $H(0)\in \cA$ are given by the Tracy-Widom distribution.
The edge universality of $H(t)$ for $t \gg N^{-1/3}$ follows from the following result due to \cite{landon2017edge}.

\begin{proposition}\label{t:universalityHt}
Fix a constant $\fc>0$ and suppose that $N^{\fc}\leq d\leq N^{1/3-\fc}$.
Let $\fd,\fe>0$ be a sufficiently small constant, set $t=N^{-1/3+\fd}$ and $L_t=E_d^+(t)$ from \eqref{e:edgeeqn2}.
Let $H(t)$ be as in \eqref{e:Ht}, which has an eigenvalue $\la_1(t)=e^{-t/2}d/\sqrt{d-1}$, and we denote its remaining eigenvalues by $\la_2(t)\geq \la_{3}(t)\geq \cdots \la_{N-1}(t)\geq \la_N(t)$.  Fix $k\geq 1$ and $s_1,s_2,\cdots, s_k \in \bR$. Then
\begin{align}\begin{split}\label{e:compGOE}
&\phantom{{}={}} \bP_{H(t)}\left( N^{2/3} ( \lambda_{i+1}(t) - L_t )\geq s_i,1\leq i\leq k \right)\\
&= \bP_{\mathrm{GOE}}\left( N^{2/3} ( \mu_i - 2  )\geq s_i,1\leq i\leq k \right) ++\OO(N^{-\fe}),
\end{split}\end{align}
where $\mu_1\geq \mu_2\geq \cdots \geq\mu_N$ are the eigenvalues of a the Gaussian Orthogonal Ensemble.
The analogous statement holds for the smallest eigenvalues.
\end{proposition}

\begin{proof}
Take small $\fd>0$, $\eta^*=N^{-2/3+\fd/2}$ and $r=N^{-2/3+10\fd}$. Let $\mathbb B(r)$ be the radius $r$ disk centered at $0$,  $z=2+ w$ with $w=\kappa+\ri\eta\in \mathbb B(r)\cap \wt{\mathbf  D} $ from \eqref{e:recalldefD}. For any $H\in \cA$ (recall from \eqref{e:defA}), from the defining relations of $\cA$, i.e. \eqref{e:mbond}, we have for $w\in \mathbb B(r)\cap \wt{\mathbf  D} $
\begin{align*}
| m(  z)-    m_d(  z)|\prec \frac{1}{N\eta}+\frac{d}{N\sqrt{\kappa+\eta}}\lesssim \frac{1}{N\eta},
\end{align*}
provided $5\fd\leq \fc$. Moreover,\eqref{e:mbond} also implies \eqref{e:largeeig} such that $\lambda_2 (0) -  {{  2}}  \leq N^{-2/3+\fd/2}/2=\eta^*/2$.   Hence, $H$ satisfies \cite[Assumption 4.1]{adhikari2020dyson} with parameter $\eta^*$. We remark that there is an additional assumption in \cite[(4.2) Assumption 4.1]{adhikari2020dyson}, which requires a bound of $m(z)-m_d(z)$ when $z$ is bounded away from the spectral edge $2$. This assumption was only used to deal with the error term \cite[Claim 4.13]{adhikari2020dyson}. This error term comes from the potential $V$ in the Dyson's Brownian motion. However, in our setting $V=0$ and the error term in \cite[Claim 4.13]{adhikari2020dyson} vanishes. 
The result \cite[Theorem 6.1]{adhikari2020dyson} applies for $t=N^{-1/3+\fd}$ as above. This result gives the limiting distribution of the extreme eigenvalues of $H(t)$, and Theorem \ref{t:universalityHt} follows.   
\end{proof}

\begin{remark} \label{rem:independence}
By an appropriate modification of the analysis of Dyson Brownian motion from~\cite{landon2017edge, adhikari2020dyson}, Proposition~\ref{t:universalityHt} also holds for the joint distribution of the $k$ largest and smallest eigenvalues. In particular, this implies that, under the same assumption as in the previous proposition, the asymptotic joint distribution of $N^{2/3}(\lambda_2(t) - L_t, -\lambda_N(t) + L_t)$ is a pair of independent Tracy--Widom$_1$ distributions. 
\end{remark}

%\begin{theorem} \label{thm:univ_gen}
%Suppose that $1 \ll d \ll N^{1/3}$.
%Fix $k\geq 1$ and $s_1^\pm,s_2^\pm,\cdots, s_k^\pm \in \bR$. Then
%\begin{multline*}
%\bP_H \pB{N^{2/3}(\lambda_{i+1} - 2) \geq s_i^+, N^{2/3}(-\lambda_{N+1 - i} - 2) \geq s_i^- , 1 \leq i \leq k}
%\\
%= \bP_{\mathrm{GOE}} \pB{N^{2/3}(\mu_{i} - 2) \geq s_i^+, N^{2/3}(-\mu_{N+1 - i} - 2) \geq s_i^- , 1 \leq i \leq k} + \oo(1).
%\end{multline*}
%\end{theorem}

\subsection{Green's function comparison}

In this section we prove the following short-time comparison result for the edge eigenvalue statistics of $H(t)$. 
\begin{proposition}\label{thm:comp}
Fix a constant $\fc>0$ and suppose that $N^{\fc}\leq d\leq N^{1/3-\fc}$.
Let $\fd,\fe>0$ be a sufficiently small constant and set $t=N^{-1/3+\fd}$. Let $H(t)$ be as in \eqref{e:Ht}, which has an eigenvalue $\la_1(t)=e^{-t/2}d/\sqrt{d-1}$, and we denote its remaining eigenvalues by $\la_2(t)\geq \la_{3}(t)\geq \cdots \la_{N-1}(t)\geq \la_N(t)$.    Fix $k\geq 1$ and $s_1,s_2,\cdots, s_k \in \bR$. Then
\begin{align}\begin{split}\label{e:comp1}
&\phantom{{}={}} \bP_{H}\left( N^{2/3} ( \lambda_{i+1}(0) - 2 )\geq s_i,1\leq i\leq k \right)\\
&= \bP_{H(t)}\left( N^{2/3} ( \lambda_{i+1}(t) - L_t  )\geq s_i,1\leq i\leq k \right) +\OO(N^{-\fe}),
\end{split}\end{align}
where $L_t=E^+_d(t)$ is as defined in Lemma \ref{e:edgeeqn2}. The analogous statement holds for the smallest eigenvalues.
\end{proposition}

\begin{proof}[Proof of Theorem \ref{thm:Tracy-Widom}]
Combine Propositions \ref{t:universalityHt} and \ref{thm:comp}.
\end{proof}

\begin{remark}\label{rem:ind2}
The proof of Proposition \ref{thm:comp} can be modified to show that the joint distributions of the $k$ largest and smallest eigenvalues of $H$ and $H(t)$ are asymptotically the same. And Corollary \ref{c:rate} follows from combining with Remark \ref{rem:independence}.
\end{remark}

Proposition \ref{thm:comp} follows from the following comparison theorem for the product of the functions of Stieltjes transform, See \cite{MR3034787}.

\begin{proposition}\label{p:comp}
Fix constant $  \fc>0$ and $N^{  \fc}\leq d\leq N^{1/3-  \fc}$.
Let $\fd,\delta>0$ be sufficiently small, set $t=N^{-1/3+\fd}$ and $\eta_0=N^{-2/3-\delta}$ and $L_t=E_d^+(t)$ from \eqref{e:edgeeqn2}. Let $F:\bR^k\mapsto \bR$ be a fixed smooth test function. Then for $E_1,E_2,\cdots, E_k=\OO( N^{-2/3})$, and $w=y+\ri\eta_0$, we have
\begin{align}\begin{split}\label{e:comparison} 
&\phantom{{}={}}\bE_{H}\left[F\left(\left\{\Im\left[N\int_{E_i}^{N^{-2/3+\delta}} m(L_0+w;t) \,  \rd y \right]\right\}_{i=1}^k\right)\right]\\
&=\bE_{H(t)}\left[\left\{F\left(\Im\left[N\int_{E_i}^{N^{-2/3+\delta}} m(L_t+w;t) \, \rd y \right]\right\}_{i=1}^k\right)\right]+\OO\left(\frac{N^{1/3+C\delta}t}{d^{1/2}}\right).
\end{split}\end{align}
\end{proposition}

Before proving Propostion \ref{p:comp}, we first state some useful estimates, which will be used repeatedly in the rest of this section. 

Recall that $H(t)=e^{-t/2}H+\sqrt{1-e^{-t}}W$ from \eqref{e:Ht}. We use the notation $\partial_{ij}^{k\ell}$ applied to a function of $H(t)$, such as $G(t)$ or $m(t)$, to denote the directional derivative with respect to $H(t)$.  
By the chain rule, we therefore have
\begin{equation*}
\frac{\partial}{\partial H} F(H(t)) = e^{-t/2} \, \partial F(H(t)), \qquad
\frac{\partial}{\partial W} F(H(t)) = \sqrt{1 - e^{-t}} \, \partial F(H(t)).
\end{equation*}

\begin{proposition}\label{p:Ltder2}
Adapt the assumptions in Proposition \ref{p:comp}, and let $L_t=E_d^+(t)$ from \eqref{e:edgeeqn2}.  Uniformly for any $w=y+\ri\eta$ with $ |y|\leq N^{-2/3+\delta}, \eta_0=N^{-2/3-  \delta}$, the following estimates hold with overwhelming probability
 \begin{align}\label{e:dermtha}
   |\Im[m(L_t+w;t)]|\leq N^{-1/3+2  \delta},\quad  \max_a |\Im[G_{aa}(L_t+w;t)]|\leq N^{-1/3+3  \delta},
\end{align}
\begin{align}
\label{e:dermtha2}
    |\del_w m(L_t+w;t)|\leq N^{1/3+3  \delta}, \quad  |\del_wG_{ab}(L_t+w;t)|\leq N^{1/3+3  \delta},
\end{align}
and
\begin{align}
\label{e:dermtha3}
    \max_a | G_{aa}(L_t+w;t)|\leq 2, \quad  \max_{a\neq b}|G_{ab}(L_t+w;t)|\leq  \frac{N^\delta}{\sqrt{d}}.
\end{align}
Let $\del_{\bm\xi}=\del_{\xi_1}\del_{\xi_2},\cdots\del_{\xi_b}$ be as in Definition \ref{d:defD} with $\bm V(\bm\xi)\subset \bfi$. For $b\geq 1$ the derivatives of the  Stieltjes transform $m(z;t)$ satisfy
\begin{align}\label{e:impp2}
     \del_{\bm \xi}m(L_t+w;t)\leq N^{-2/3+4  \delta}.
\end{align}
\end{proposition}

\begin{proof}[Proof of Proposition \ref{p:Ltder2}]
    Let $0\leq \fa\leq \delta/2$ and $w'=y+N^{-2/3+\fa}\ri$, then it is easy to see that $w'\in  \wt{\mathbf  D} $ (recall from \ref{e:recalldefD}), and Proposition \ref{p:rigidity2} gives
    \begin{align*}\begin{split}
        \Im[m(L_t+w';t)]
        &\leq \Im[m_d(L_t+w';t)]+\OO_\prec\left(\frac{1}
        {N\Im[w']}\right)\\
        &\lesssim \sqrt{|w'|}+\OO_\prec\left(\frac{1}
        {N\Im[w']}\right)\lesssim \frac{N^{  \delta/2}}{N^{1/3}},
    \end{split}\end{align*}
    where for the second line we used that $m_d(z;t)$ has square root behavior around the edge $L_t$.
For the diagonal Green's function entries, thanks to the delocalization of eigenvectors
\begin{align*}
\Im[G_{aa}(L_t+w;t)]
&=\sum_{i=2}^N\frac{\Im[w](\bmu_i(t)\bmu_i(t)^\top)_{aa}}{|L_t+w-\la_\al(t)|^2}
\leq \frac{N^\delta }{N}\sum_{i=2}^N\frac{\Im[w]}{|L_t+w-\la_\al(t)|^2}\\
&=N^\delta \Im[m(L_t+w;t)]\leq N^{-1/3+3\delta},
\end{align*}    
hold with overwhelming probability.
    
    The derivative of $m(z;t)$ satisfies
    \begin{align}\label{e:dermt}
        |\del_z \Im[m(z;t)]|\leq |\del_z m(z;t)|\leq \frac{\Im[m(z;t)]}{\Im[z]},
    \end{align}
    which gives that $\Im[m(E+\ri \eta/M;t)\leq M\Im[m(E+\ri \eta;t)]$ for any $M\geq 1$. In particular, we have
    $\Im[m(L_t+w;t)]\leq N^{\fa+  \delta}\Im[m(L_t+w';t)]\leq N^{\fa+  \delta}N^{-1/3+  \delta/2}\leq N^{-1/3+2  \delta}$ with overwhelming probability, provided $\fa\leq   \delta/2$. 
    
Using \eqref{e:dermt} again, $\Im[m(L_t+w;t)]\leq N^{-1/3+2  \delta}$ implies 
    $|\del_z m(L_t+w;t)|\leq N^{1/3+3  \delta}$ with overwhelming probability.
For the derivative of the Green's function, Ward identity \eqref{e:WdI} and the delocalization of eigenvectors implies 
\begin{align*}
|\del_w G_{ab}(L_t+w;t)|
&\leq \sum_{i=1}^N |G_{ai}(L_t+w;t)G_{bi}(L_t+w;t)|\\
&\leq \frac{1}{2}\sum_{i=1}^N (|G_{ai}(L_t+w;t)|^2+|G_{bi}(L_t+w;t)|^2)\\
&\prec \frac{\Im[m(L_t+w;t)]}{\eta}\leq N^{1/3+3  \delta},
\end{align*}
with overwhelming probability.

The entry-wise estimates \eqref{e:dermtha3} follow from 
\begin{align*}
G_{ab}(L_t+w;t)
&=G_{ab}(L_t+w';t)-\int_{w}^{w'}\del_z G_{ab}(L_t+w;t)\rd w\\
&=G_{ab}(L_t+w';t)+\OO\left(|w-w'| N^{1/3+3\delta}\right)\leq \bm2(a=b)+\frac{N^\delta}{\sqrt{d}},
\end{align*}
where we used the entry-wise bound of the Green's function entries from Proposition \ref{p:entry-wiselaw}.

The derivative $\del_{\bm\xi} m(z,t)$ of $m(z;t)=(1/N)\sum_{k=1}^NG_{kk}(z;t)$ is a linear combination of terms in the form
\begin{align*}
\frac{1}{N}\sum_{k=1}^NG_{k x_1}G_{x_1x_2}\cdots G_{x_b k}
\end{align*}
where $x_1, x_2,\cdots, x_b\in\bm V(\bm\xi)$. Using the bounds \eqref{e:dermtha},\eqref{e:dermtha3} and Ward identity, \eqref{e:impp2} follows
\begin{align*}
\left|\frac{1}{N}\sum_{k=1}^NG_{k x_1}G_{x_1x_2}\cdots G_{x_b k}\right|
&\lesssim \frac{1}{N}\sum_{k=1}^N |G_{kx_1}||G_{x_b k}|
\leq \frac{1}{2N}\sum_{k=1}^N |G_{kx_1}|^2+|G_{x_b k}|^2\\
&=\frac{1}{2}\frac{\Im[G_{x_1x_1}]+\Im[G_{x_b x_b}]}{N\eta_0}\lesssim N^{-2/3+4\delta}.
\end{align*}

\end{proof}

As an easy consequence of Proposition \ref{p:Ltder2}, for any $w=y+\ri\eta$ with $|y|\leq N^{-2/3+  \delta}, \eta_0=N^{-2/3-  \delta}$ by Ward identity \eqref{e:WdI}
\begin{align}\begin{split}\label{e:derofGG}
&\frac{1}{N^2}\sum_{ij}|G_{ij}(L_t+w;t)|^2= \frac{\Im[  m(L_t+w;t)]}{N\eta}\leq N^{-2/3+3  \delta},\\
& \frac{1}{N^2}\sum_{ij}|G_{ij}(L_t+w;t)|\leq  \sqrt{\frac{\Im[  m(  L_t+w;t)]}{N\eta}}\leq N^{-1/3+3  \delta/2},
\end{split}\end{align}
with overwhelming probability.

\subsection{Proof of Proposition \ref{p:comp}}\label{s:p1}

For simplicity of notation, we only prove the case $k=1$; the general case can be proved in the same way. Let
\begin{align*}
X_t=\Im\left[N\int_{E}^{N^{-2/3+\delta}}m(L_t +y+\ri \eta_0;t) \, \rd y\right].
\end{align*}
We shall prove that 
\begin{align}\label{e:onecomp}
|\bE[F(X_t)]-\bE[F(X_0)]|\leq  \frac{N^{1/3+C\delta}t}{d^{1/2}}.
\end{align}
We notice that \eqref{e:impp2} implies that 
\begin{align}\label{e:derofXt}
|\del_{\bm\xi}X_t|
=\left| \Im\left[N\int_{E}^{N^{-2/3+  \delta}}\del_{\bm\xi} m( L_t +w;t) \rd y\right]\right|\lesssim N^{-1/3+5  \delta},
\end{align}
with overwhelming probability.

\begin{proposition}\label{p:DSEt}
Adapt the assumptions in Proposition \ref{p:comp} and let $L_t=E_d^+(t)$ from \eqref{e:edgeeqn2}.
Let $w=y+\ri\eta_0$ with $\eta_0=N^{-2/3-  \delta}$ and $z=L_t+w$, we have the following estimates
\begin{align}\begin{split}\label{e:dFXt}
     \frac{\rd}{\rd t}\bE[F(X_t)]&= \left.\bE\left[F'(X_t)\Im \left(\sum_{ij}\dot{H}_{ij}(t)G_{ij}(t)-\frac{N}{2}Q(e^{-t/2}m(t))\right)\right]\right|_{y=E}^{y=N^{-2/3+\delta}}\\
     &+\OO\left( \frac{N^{10  \delta+1/3}}{d^{1/2}}\right),
\end{split}\end{align}
and uniformly for any $ |y|\leq N^{-2/3+  \delta}$, 
\begin{align}\label{e:dFX0}
 \bE\left[F'(X_t)\Im \left(\frac{1}{N}\sum_{ij}\dot{H}_{ij}(t)G_{ij}(t)-\frac{1}{2}Q(e^{-t/2}m(t))\right)\right]\leq \frac{N^{8  \delta}}{N^{2/3}d^{1/2}}.
\end{align}
\end{proposition}

The derivative of $\bE[F(X_t)]$ with respect to the time $t$ is
\begin{align}\begin{split}\label{e:derF}
&\phantom{{}={}}\frac{\rd}{\rd t}\bE[F(X_t)]
=\bE\left[F'(X_t)\frac{\rd X_t}{\rd t}\right]\\
&
=\bE\left[F'(X_t)\Im\int_{E}^{N^{-2/3+\delta}}\left(N \sum_{ij}\dot{H}_{ij}(t)\frac{\del m(t)}{\del H_{ij}(t)}+\dot{L}_t\sum_{ij}G^2_{ij}(t)\right)\rd y\right]\\
&=\bE\left[F'(X_t)\Im\int_{E}^{N^{-2/3+\delta}}\left(-\sum_{ijk}\dot{H}_{ij}(t)G_{ik}(t)G_{jk}(t)+\dot{L}_t\sum_{ij}G_{ij}^2(t)\right)\rd y\right],
\end{split}\end{align}
where we abbreviate $G(L_t+ y +\ri \eta_0;t)$ and $m(L_t+y +\ri \eta_0;t)$ by $G(t)$ and $m(t)$ respectively.

To estimate the second term on the righthand side of \eqref{e:derF}, the key is to compute the time derivative of $L_t$, which is given by  the following Proposition.
\begin{proposition}\label{p:Ltder}
Adapt the assumptions in Proposition \ref{p:comp}. 
We have the following estimate with overwhelming probability, uniformly for any $w=y+\ri\eta_0$ with $ |y|\leq N^{-2/3+  \delta}, \eta_0=N^{-2/3-  \delta}$
\begin{align}\label{e:Ltder}
\left|\del_t L_t -\frac{e^{-t/2}}{2}Q'(e^{-t/2}m(L_t+w;t))\right|\leq \frac{N^  {4\delta}}{N^{1/3}\sqrt d},
\end{align}
where $L_t=E_d^+(t)$ is constructed in Proposition \ref{p:mtEQ}.
\end{proposition}
\begin{proof}[Proof of Proposition \ref{p:Ltder}]
We recall from Proposition \ref{p:mtEQ}, the derivative of the  spectral edge $L_t$ satisfies
\begin{align}\label{e:deltEcopy}
\del_t L_t=\frac{e^{-t/2}}{2} Q'(e^{-t/2}  m_d(L_t;t))+\OO\left(\sqrt{\frac{d}{N}}\right).
\end{align}
Let $w'=y+N^{-2/3+\fa}\ri$, then it is easy to see that $w'\in  \wt{\mathbf  D} $ (recall from \ref{e:recalldefD}).
By \eqref{e:dermtha} and the rigidity estimates \eqref{e:mtmdt}
\begin{align*}
    &\phantom{{}={}}|m(L_t+w;t)-m_d(L_t; t)|
    \leq |m(L_t+w';t)-m(L_t+w;t)|\\
    &+|m_d(L_t+w';t)-m(L_t+w';t)|
    +|m_d(L_t+w';t)-m_d(L_t;t)|\\
    &\leq|w'-w| N^{1/3+3  \delta}+\O_\prec\left( \frac{1}{N\Im[w']}\right)+\sqrt{|w'|},
\end{align*}
where for the first term in the last line, we used $|\del_z m(L_t+y+\ri\eta;t)|\leq N^{1/3+3\delta}$ for $\eta\geq N^{-2/3-\delta}$ from \eqref{e:dermtha2}; the last term in the last line, we used the square root behavior of $m_d(z;t)$: close to the spectral edge we have $| m_d(L_t+w';t)- m_d(L_t;t)|\lesssim |w'|^{1/2}\lesssim N^{-1/3+  \delta/2}$.
Therefore it follows that 
\begin{align}\label{e:mtbb}
    |m(L_t+w;t)-m_d(L_t; t)|\leq  \frac{N^{4  \delta}}{N^{1/3}},
\end{align}
holds with overwhelming probability.
By plugging \eqref{e:mtbb} into \eqref{e:deltEcopy}, and using that $ Q'(m)$ is a finite polynomial in $m$ with coefficients bounded by $\O_\prec(d^{-1/2})$, we conclude
\begin{align*}
    \del_t L_t=\frac{e^{-t/2}}{2} Q'(e^{-t/2} m_d(L_t;t))
    &=\frac{e^{-t/2}}{2}Q'(e^{-t/2}m(L_t+w;t))+\O\left(\frac{N^{4  \delta}}{N^{1/3}d^{1/2}}\right).
\end{align*}
\end{proof}

\begin{proof}[Proof of \eqref{e:dFXt}]
We can rewrite the summation over $k$ in the first term on the righthand side of \eqref{e:derF} as
$-\sum_{k}G_{ik}(t)G_{jk}(t)=\del_y G_{ij}(L_t+y+\ri\eta; t)$.
For the second term on the righthand side of \eqref{e:derF}, thanks to Proposition \ref{p:Ltder}, 
\begin{align*}
\dot{L_t}\sum_{ij}G_{ij}^2(t)
&=-N\frac{e^{-t/2}}{2}\left(Q'(e^{-t/2}m(t))+\O\left(\frac{N^{4  \delta}}{N^{1/3}d^{1/2}}\right)\right)\del_y m(t)\\
&=-\frac{N}{2}\del_y Q(e^{-t/2}m(t))+\O\left(\frac{N^{2/3+7  \delta}}{d^{1/2}}\right),
\end{align*}
holds with overwhelming probability, where we used \eqref{e:dermtha2} to bound $\del_y m(t)$.
From the discussion above, we can rewrite the righthand side of \eqref{e:derF}
as
\begin{align*}
&\phantom{{}={}} \bE\left[F'(X_t)\Im\int_{E}^{N^{-2/3+\delta}}\del_y \left(\sum_{ij}\dot{H}_{ij}(t)G_{ij}(t)-\frac{N}{2}Q(e^{-t/2}m(t))\right)\rd y\right]
\\
&= \left.\bE\left[F'(X_t)\Im \left(\sum_{ij}\dot{H}_{ij}(t)G_{ij}(t)-\frac{N}{2}Q(e^{-t/2}m(t))\right)\right]\right|_{y=E}^{y=N^{-2/3+\delta}},
\end{align*}
with an error $\O(N^{2/3+8  \delta}/d^{1/2})$. This finishes the proof of \eqref{e:dFXt}.
\end{proof}

%
%\begin{proof}[Proof of \eqref{e:dFXt}]
%By using Proposition \ref{p:Ltder}, we can rewrite the second term on the righthand side of \eqref{e:dFt} as
%\begin{align}\begin{split}\label{e:fft2}
%&N\bE\left[F'(X_t)\del_t{E}_t\del_w m(L_t+w;t)\right]
%=\frac{Ne^{-t/2}}{2}\bE\left[F'(X_t)\del_mQ'(e^{-t/2}m(L_t+w;t))\del_w m(L_t+w;t)\right]\\
%&+\O_\prec\left(\frac{N^{2/3+  \delta}}{d}\right)\bE[|\del_w m(L_t+w;t)|]
%=\frac{N}{2}\bE\left[F'(X_t)\del_wQ(e^{-t/2}m(L_t+w;t))\right]
%+\O_\prec\left(\frac{N^{1+4  \delta}}{d}\right),
%\end{split}\end{align}
%where we used Proposition \ref{p:Ltder} in the first line, and \eqref{e:dermtha} in the second line. 
%
%
%\end{proof}

\subsection{Proof of \eqref{e:dFX0}}\label{s:p2}

In the following, we estimate the first term in \eqref{e:dFX0}. By definition,
\begin{align}\label{h:derive}
\dot{H}_{ij}(t)=-\frac{1}{2}e^{-t/2}H_{ij}+\frac{e^{-t}}{2\sqrt{1-e^{-t}}}W_{ij}.
\end{align}
Plugging \eqref{h:derive} into \eqref{e:derF}, and using the Gaussian integration by parts \eqref{e:intbypart}, we obtain
\begin{align}\begin{split}\label{e:derexp}
&\frac{1}{N}\sum_{ij}\bE\left[\dot H_{ij}(t)F'(X_t) G_{ij}(t)\right]
=\frac{1}{2N}\sum_{ij}\bE\left[\left(-e^{-t/2}H_{ij}+\frac{e^{-t}}{\sqrt{1-e^{-t}}}W_{ij}\right)F'(X_t) G_{ij}(t)\right]\\
&=-\frac{e^{- t / 2}}{2N \sqrt{d-1}}\sum_{ ij}\bE\left[A_{ij}G_{ij}(t)F'(X_t)\right] +\frac{e^{-t}}{2 N^4}\sum_{ijk\ell}\bE[\del_{ij}^{k\ell}(G_{ij}(t) F'(X_t))].
\end{split}\end{align}

If we replace $F'(X_t)$ by $P^q\bar P^q$, the expression on the righthand side of \eqref{e:derexp} is essentially the same as \eqref{e:EPP}, up to a $-e^{-t/2}/2$ factor.
We have these $e^{-t/2}$ factors in \eqref{e:derexp}, because \eqref{e:derexp} is a function of $H(t)$. The discrete derivative gives extra $e^{-t/2}$ factors: 
\begin{align}\begin{split}\label{e:hijt}
D_{\xi}F(H(t))&=F\left(H(t)+\frac{e^{-t/2}\xi}{\sqrt{d-1}}\right)-F(H(t))\\
&= \sum_{n=1}^{\fb-1}  \frac{e^{-nt/2}\del_{\xi}^nF(H(t))}{n!(d-1)^{n/2}} +  \OO\left((d-1)^{-\fb/2}\right).
\end{split}\end{align}

The same as in Section \ref{sec:P-construct}, we can expand the first term on the righthand side of \eqref{e:derexp} by repeatedly using Corollaries \ref{c:intbp} and \ref{c:intbp2}.
Similarly to \eqref{e:newterm}, all the terms we will get  in the expansion are in the form
\begin{align}\label{e:newterms}
 \frac{1}{d^{\fh/2}}\frac{e^{-\fd t/2}}{N^{\theta(\cT)}d^{|E|}}\sum^*_{\bf i}\bE\left[A_{\cT}R_{\bfi}(A(t)) V_\bfi(A(t))\right],\quad \deg(R_{\bfi}(A(t))=\fd,
\end{align}
where the forest $\cT=(\bfi, E)$ is as in Definition \ref{d:forest}, $R_\bfi$ is a monomial of Green's function entries as in Definition \ref{d:evaluation}, and $V_\bfi$ is in the form (same as in \eqref{e:defU} after replacing $P^{q-1}\bar P^q$ to $F'(X_t)$)
\begin{align}\label{e:defV}
\prod_{a=1}^b(1+D_{\xi_a})^{e_a}(\sqrt{d-1}D_{\xi_a})^{1-e_a}  F'(X_t), \quad e_1, e_2, \cdots, e_b\in\{0,1\}.
\end{align}
If $b=0$, the above simplifies to $F'(X_t)=0$.

Analogues to Proposition \ref{p:replaceR}, we have the following Proposition. By using estimates from Proposition \ref{p:Ltder2} as input, the proofs are similar to those of Proposition \ref{p:replaceR}. So we will omit the proofs.
\begin{proposition}\label{p:replaceRV}
Propositions \ref{c:case1exp}, \ref{c:case2exp}, \ref{p:rewrite}  and \ref{p:one-off} hold if we replace $R_\bfi$ to $e^{-\deg(R_\bfi) t/2}R_\bfi$, $U_\bfi$ to $ V_\bfi$ as defined in \eqref{e:defV}, and the error $\OO_{\prec}(\bE[\Phi_q])$ to $\OO(N^{C  \delta-2/3})$.
\end{proposition}

We recall $\cI(t)$ from \eqref{e:cI}, the derivative $\del_{ij}^{k\ell}G_{ij}(t)$ is given by $-G_{ii}(t)G_{jj}(t)-G_{ik}(t)G_{j\ell}(t)-\cI(t)$. For the first  term on the righthand side of \eqref{e:derexp} we will show that \begin{align}\begin{split}\label{e:firstterm}
&-\frac{e^{- t / 2}}{N \sqrt{d-1}}\sum_{ ij}\bE\left[A_{ij}G_{ij}(t)F'(X_t)\right]=\frac{e^{- t }d}{ N^4(d-1)}\sum^*_{ijk\ell}\bE\left[(G_{ik}(t)G_{j\ell}(t)+\cI) F'(X_t))\right]\\
&-\frac{e^{- t }d}{ N^4(d-1)}\sum^*_{ ijk\ell}\bE\left[G_{ij}(t)\del_{ij}^{k\ell}(F'(X_t))\right]+\frac{e^{- t }d}{N^4(d-1)}\sum^*_{ ijk\ell}\bE\left[A_{ik}A_{j\ell} G_{ik}(t)G_{j\ell}(t) F'(X_t))\right]\\
&+\{\text{higher order term}\}+\OO\left(\frac{N^{7\delta}}{N^{2/3}\sqrt{d}}\right),
\end{split}\end{align}
where the higher order term is a linear combination of 
\begin{align}\label{e:rec1acopy}
 \frac{1}{d^{\fh/2}}\frac{1}{N^{\theta(\cT_1)}d^{|E_1|}}\sum^*_{\bf i}\bE\left[A_{\cT_1}R_{\bfi_1}V_{\bfi_1}\right], \quad V_{\bfi_1}=F'(X_t) \text{ or }V_{\bfi_1}=\sqrt{d-1}D_{ij}^{k\ell}F'(X_t),
\end{align}
with $\fh\geq 1$ and $\cT_1=\{ijk\ell, \{i,k\}\cup \{j,\ell\}\}$. There is a one-to-one correspondence between these terms with those in \eqref{e:rec1a} with $\fh\geq 1$.

For the second term on the righthand side of \eqref{e:derexp} we will show that 
\begin{align}\begin{split}\label{e:secondterm}
&\phantom{{}={}}\frac{e^{-t}}{ N^4}\sum_{ijk\ell}\bE[\del_{ij}^{k\ell}(G_{ij}(t) F'(X_t))]=-\bE[e^{-t}m^2(t)F'(X_t)]\\
&+\frac{e^{-t}}{ N^4}\sum^*_{ijk\ell}\bE[-(G_{ik}(t)G_{j\ell}(t)+\cI(t)) F'(X_t)+G_{ij}(t)\del_{ij}^{k\ell}(F'(X_t))]+\OO(1/N).
\end{split}\end{align}

\begin{proof}[Proof of \eqref{e:dFX0}]

Using the estimates \eqref{e:firstterm}, \eqref{e:secondterm} and Proposition \ref{p:replaceRV} as input, the proof of \eqref{e:dFX0} is parallel to the proof of Proposition \ref{p:DSE} as given in Section \ref{s:prove}. 

The first two terms on the righthand side of \eqref{e:firstterm} cancels with the second term on the righthand side of \eqref{e:secondterm}, 
\begin{align}
\frac{e^{- t }}{N^4(d-1)}\sum^*_{ijk\ell}\bE\left[(G_{ik}(t)G_{k\ell}(t)+\cI(t)) F'(X_t))-G_{ij}(t)\del_{ij}^{k\ell}(F'(X_t))\right]\leq \frac{1}{d}\frac{N^{C\delta}}{N^{2/3}},
\end{align}
where we used Proposition \ref{p:replaceRV}, since each term in $(G_{ik}(t)G_{k\ell}(t)+\cI(t))$ contains an off-diagonal Green's function entry.

For the third term on the righthand side of \eqref{e:firstterm}, and terms in \eqref{e:rec1acopy}, we can further use Proposition \ref{p:replaceRV} to expand them, with errors bounded by $N^{C\delta}/\sqrt{d}N^{2/3}$. The same as in the proof of Proposition \ref{p:DSE}, we will obtain  a sequence of graphs 
\begin{align*}
\cT_0=(\bfi_0, E_0),\quad  \cT_1=(\bfi_1, E_1), \quad \cT_2=(\bfi_2, E_2), \quad \cT_3=(\bfi_3, E_3),\quad \cdots,
\end{align*}
and end at terms in the form
\begin{align}\label{e:Trtermk}
\frac{1}{d^{\fh/2}}\frac{e^{-\fd t/2}}{N^{\theta(\cT_t)}d^{|E_t|}}\sum^*_{\bfi_t}\bE\left[A_{\cT_t}m^{\fd}(t) V_{\bfi_t}\right], \quad \fd=\deg(R_{\bfi_t}),\quad \fh\geq 1.
\end{align}
Then we can use Proposition \ref{p:replaceRV} to rewrite \eqref{e:Trtermk} back to $\bE[Q(e^{-t/2}m(t))F'(X_t)]$.

\end{proof}

\begin{proof}[Proof of \eqref{e:firstterm}]
For \eqref{e:firstterm}, we apply Corollary~\ref{c:intbp} with the random variable $F_{ij}=F'(X_t) G_{ij}(t)$.
Since
$\cal C(F_{ij},A) \leq N^{\delta}/\sqrt{d}$ with overwhelming probability,
\begin{align}\begin{split}\label{e:dFG2}
&\phantom{{}={}}-\frac{e^{- t / 2}}{ N \sqrt{d-1}}\sum^*_{ ij}\bE\left[A_{ij}G_{ij}(t) F'(X_t)\right]=-\frac{e^{- t / 2}}{  N\sqrt{d-1}}\frac{d}{N}\sum^*_{ ij}\bE\left[F'(X_t)G_{ij}(t)\right]\\
& 
-\frac{e^{- t / 2}}{  N^2d\sqrt{d-1}}\sum^*_{ ijk\ell}\bE\left[A_{ik}A_{j\ell}D_{ij}^{k\ell}(F'(X_t)G_{ij}(t))\right] +\OO\left(\frac{d N^{\delta}}{N}\right).
\end{split}\end{align}
For the first term on the righthand side of \eqref{e:dFG2}, using $\sum_{ij} G_{ij}(t)=0$ and $\sum^*_{ij}G_{ij}(t)\lesssim N$ with overwhelming probability, it is bounded by $\OO(\sqrt{d}/N)$. We rewrite the second term on the righthand side of \eqref {e:dFG2} as the sum
\begin{align}\label{e:dFG3}
-\frac{e^{- t / 2}}{  N^2d\sqrt{d-1}}\left(\sum^*_{ ijk\ell}\bE\left[A_{ik}A_{j\ell}G_{ij}(t)D_{ij}^{k\ell}(F'(X_t))\right]+\bE\left[A_{ik}A_{j\ell}D_{ij}^{k\ell}G_{ij}(t)(1+D_{ij}^{k\ell})F'(X_t))\right]\right).
\end{align}
We recall the Taylor expansion from \eqref{e:hijt}, the first term in \eqref{e:dFG3}
\begin{align}\begin{split}\label{e:erterm}
&-\frac{e^{- t }}{  N^2d(d-1)}\sum^*_{ ijk\ell}\bE\left[A_{ik}A_{j\ell}G_{ij}(t)\del_{ij}^{k\ell}(F'(X_t))\right]
+\OO\left(\frac{1}{ N^2d^{5/2}}\sum^*_{ ijk\ell}\bE\left[A_{ik}A_{j\ell}|G_{ij}(t)|\frac{N^{5\delta}}{N^{1/3}}\right]\right)\\
&=-\frac{e^{- t }}{  N^2d(d-1)}\sum^*_{ ijk\ell}\bE\left[A_{ik}A_{j\ell}G_{ij}(t)\del_{ij}^{k\ell}(F'(X_t))\right]
+\OO\left(\frac{N^{7  \delta}}{ N^{2/3}\sqrt{d}}\right),
\end{split}\end{align}
where we used \eqref{e:derofXt} to bound $(\del_{ij}^{k\ell})^n(F'(X_t))$ and \eqref{e:derofGG} to sum over $|G_{ij}(t)|$. We can further expand \eqref{e:erterm} using Corollary \ref{c:intbp} with $F_{ijk\ell}=G_{ij}(t)\del_{ij}^{k\ell}(F'(X_t))$. Then $(d/N)\cC(F,A)\lesssim N^{-2/3}$, and
\begin{align*}
&\phantom{{}={}}-\frac{e^{- t }}{  N^2d(d-1)}\sum^*_{ ijk\ell}\bE\left[A_{ik}A_{j\ell}G_{ij}(t)\del_{ij}^{k\ell}(F'(X_t))\right]\\
&=
-\frac{e^{- t }d}{  N^4(d-1)}\sum^*_{ ijk\ell}\bE\left[G_{ij}(t)\del_{ij}^{k\ell}(F'(X_t))\right]
+\OO\left(\frac{N^{7 \delta}}{ N^{2/3}\sqrt{d}}\right).
\end{align*}
This gives the second term on the righthand side of \eqref{e:firstterm}.

Expanding $D_{ij}^{k\ell}G_{ij}(t)$ as in \eqref{e:hijt}, the second term in \eqref{e:dFG3}  can be rewritten as a linear combination of 
\begin{align}\label{e:rec1a2}
 \frac{1}{d^{\fh/2}}\frac{1}{N^{\theta(\cT_1)}d^{|E_1|}}\sum^*_{\bf i}\bE\left[A_{\cT_1}R_{\bfi_1}V_{\bfi_1}\right],\quad V_{\bfi_1}=F'(X_t) \text{ or }V_{\bfi_1}=\sqrt{d-1}D_{ij}^{k\ell}F'(X_t),
\end{align}
with $\cT_1=\{\bfi_1=ijk\ell, \{i,k\}\cup\{j,\ell\}\}$ and $\fh\geq 0$. The terms corresponding to $\fh=0$ are from taking $n=1$ in \eqref{e:hijt} and $V_{\bfi_1}=F'(X_t)$, given by 
\begin{align}\begin{split}\label{e:It}
&-\frac{e^{- t }}{  N^2d(d-1)}\sum^*_{ ijk\ell}\bE\left[A_{ik}A_{j\ell}\del_{ij}^{k\ell}G_{ij}(t)F'(X_t))\right]\\
&=\frac{e^{- t }d}{  N^4(d-1)}\sum^*_{ ijk\ell}\bE\left[A_{ik}A_{j\ell}(G_{ii}(t)G_{jj}(t)+G_{ik}(t)G_{j\ell}(t)+\cI(t))F'(X_t))\right]\\
&=\frac{e^{-t}d}{ (d-1)}\bE[m(t)^2 F'(X_t)]+\frac{e^{- t }d}{  N^4(d-1)}\sum^*_{ \bfi_1}\bE\left[A_{\cT_1} (G_{ik}(t)G_{j\ell}(t)+\cI(t)) F'(X_t))\right],
\end{split}\end{align}
where $\cI(t)$ is from \eqref{e:cI}.  We can further expand the last term in \eqref{e:It} using Corollary \ref{c:intbp} with $F_{ijk\ell}=\cI(t)F'(X_t)$. Then $(d/N)\cC(F,A) \leq \sqrt{d}/N^{1-\delta}\lesssim 1/(N^{2/3 }\sqrt d)$, and
\begin{align}\begin{split}\label{e:T1expand}
&\phantom{{}={}}\frac{e^{- t }d}{  N^4(d-1)}\sum^*_{ \bfi_1}\bE\left[A_{\cT_1} \cI(t) F'(X_t))\right]\\
&=\frac{e^{- t }d}{  N^4(d-1)}\left(\sum^*_{\bfi_1}\bE\left[\cI(t) F'(X_t))\right]
+\sum^*_{ \bfi_2}\bE\left[A_{\cT_2} \cI(t) D_{\cT_2} F'(X_t))\right]\right)\\
&+\frac{e^{- t }d}{  N^4(d-1)}\sum^*_{ \bfi_2}\bE\left[A_{\cT_2} D_{\cT_2}\cI(t) (1+D_{\cT_2} )F'(X_t))\right]
+\OO\left(\frac{1}{\sqrt{d}N^{2/3}}\right).
\end{split}\end{align}

 Since $|D_{\cT_2} F'(X_t)|\leq N^{5\delta}/N^{1/3}\sqrt{d}$, using \eqref{e:derofGG} the second term on the righthand side of \eqref{e:T1expand} is bounded 
 \begin{align*}
 \frac{e^{- t }d}{  N^4(d-1)}
\sum^*_{ \bfi_2}\bE\left[A_{\cT_2} \cI(t) D_{\cT_2} F'(X_t))\right]=\OO(N^{7\delta}/\sqrt{d}N^{2/3}).
\end{align*}
For the third term in \eqref{e:T1expand}, we notice that all terms in $\cI(t)$,  contains some $G_{xy}(t)$ such that $x,y$ are in different connected components of $\cT_1$. By Proposition \ref{p:DFU}, each term in $D_{\cT_2}\cI(t)$ also contains 
some $G_{x'y'}(t)$ such that $x',y'$ are in different connected components of $\cT_2$ and at least one extra $1/\sqrt{d-1}$ factor. These terms can be bounded by $N^{C\delta}/\sqrt{d}N^{2/3}$ thanks to Proposition \ref{p:replaceRV}. 
%These terms from $D_{\cT_2}(G_{ik}G_{j\ell})$ are in the form
%\begin{align}\label{e:secondt}
%\frac{1}{d^{\fh/2}}\frac{1}{ N^{|\theta(\cT_2)|d^{|E_{2}|}}}\sum^*_{ \bfi_2}\bE\left[A_{\cT_2} R_{\bfi_2} V_{\bfi_2})\right],\quad V_{\bfi_2}=(1+D_{\cT_2})F'(X_t),\quad \fh\geq 1.
%\end{align}

\end{proof}
\begin{proof}[Proof of \eqref{e:secondterm}]
We can rewrite \eqref{e:secondterm} as
\begin{align}\label{e:stt}
\frac{e^{-t}}{  N^4}\sum_{ijk\ell}\bE[\del_{ij}^{k\ell}(F'(X_t) G_{ij}(t))]
=\frac{e^{-t}}{  N^4}\sum_{ijk\ell}\bE[\del_{ij}^{k\ell}(G_{ij}(t)) F'(X_t)+G_{ij}(t)\del_{ij}^{k\ell}(F'(X_t))].\end{align}

By plugging $\del_{ij}^{k\ell}G_{ij}(t)=-G_{ii}(t)G_{jj}(t)-G_{ik}(t)G_{j\ell}(t)-\cI(t)$ into \eqref{e:stt}, we get
\begin{align}\begin{split}\label{e:sterm}
&-\bE[e^{-t}m^2(t)F'(X_t)]\\
&+\frac{e^{-t}}{  N^4}\sum_{ijk\ell}\bE[-(G_{ik}(t)G_{j\ell}(t)+\cI(t)) F'(X_t)+G_{ij}(t)\del_{ij}^{k\ell}(F'(X_t))].
\end{split}\end{align}
For the last summation in \eqref{e:sterm}, we can restrict it to $ijk\ell$ distinct, with an error $\OO(1/N)$. This gives \eqref{e:secondterm}.
\end{proof}

\bibliography{all.bib}{}
\bibliographystyle{abbrv}

\end{document}